\documentclass[preprint,leqno]{imsart}
\RequirePackage[OT1]{fontenc}
\usepackage[margin=1.22in]{geometry}

\RequirePackage{dsfont}
\RequirePackage{amsthm,amsmath,amssymb,nicefrac,bm,upgreek,verbatim}
\RequirePackage{natbib}
\RequirePackage[final]{hyperref}
\RequirePackage[mathscr]{eucal}
\RequirePackage[normalem]{ulem}

\allowdisplaybreaks
\usepackage{color}

\startlocaldefs

\definecolor{dmagenta}{rgb}{.4,.1,.5}
\definecolor{dred}{rgb}{.5,.0,.0}
\definecolor{mred}{rgb}{.7,.0,.0}
\definecolor{dgreen}{rgb}{.0,.5,.0}
\definecolor{dblue}{rgb}{.0,.0,.5}
\definecolor{mblue}{rgb}{.0,.0,.7}

\usepackage[refpage]{nomencl}
\makenomenclature

\makeatletter
\newcommand \Dotfill {\leavevmode \cleaders \hb@xt@ .53em{\hss .\hss }\hfill \kern \z@}
\makeatother

\numberwithin{equation}{section}

\newcounter{dummy} \numberwithin{dummy}{section}
\theoremstyle{plain}
\newtheorem{theorem}[dummy]{Theorem}
\newtheorem{lemma}[dummy]{Lemma}
\newtheorem{proposition}[dummy]{Proposition}
\newtheorem{corollary}[dummy]{Corollary}

\theoremstyle{definition}
\newtheorem{definition}[dummy]{Definition}
\newtheorem{example}[dummy]{Example}
\newtheorem{hypothesis}[dummy]{Hypothesis}

\theoremstyle{remark}
\newtheorem{remark}[dummy]{Remark}

\newcommand{\cA}{\mathcal{A}}   % Operator
\newcommand{\sB}{\mathscr{B}}   % Balls
\newcommand{\Cc}{\mathcal{C}}   % Set of continuous functions
\newcommand{\Ccc}{\mathcal{C}_{\mathrm{c}}}   % Continuous functions compact support
\newcommand{\sE}{\mathscr{E}}   % Control Effort?
\newcommand{\sG}{\mathscr{G}}   % In Lemma~4.1
\newcommand{\cG}{\mathcal{G}}   % Class of Matrices M-G Hurwitz
\newcommand{\cH}{\mathcal{H}}   % Space of L2 solutions of sde
\newcommand{\sJ}{\mathscr{J}}   % Value of ergodic cost
\newcommand{\cJ}{\mathcal{J}}   % Optimal value (LQG)
\newcommand{\fJ}{\mathfrak{J}}  % Optimal values in different regimes
\newcommand{\fJs}{\mathfrak{J}_{\mspace{0.5mu}\mathrm{s}}}   % Optimal values-sub
\newcommand{\fJc}{\mathfrak{J}_{\mspace{0.5mu}\mathrm{c}}}   % Optimal values-critical
\newcommand{\cK}{\mathcal{K}}   % Neighborhood
\newcommand{\Lg}{\mathcal{L}}   % Operator
\newcommand{\cN}{\mathcal{N}}   % Neighborhood
\newcommand{\sN}{\mathscr{N}}   % Normal distribution
\newcommand{\cP}{\mathcal{P}}   % Probability Measures
\newcommand{\fP}{\mathfrak{P}}  % Ergodic occupation measures
\newcommand{\sR}{\mathscr{R}}   % Running cost
\newcommand{\sS}{\mathfrak{S}}   % Minimal stochastically stable set
\newcommand{\cS}{\mathcal{S}}   % Equilibria
\newcommand{\cSs}{\mathcal{S}_{\mspace{0.5mu}\mathrm{s}}}   % Stable Equilibria
\newcommand{\cSo}{\mathcal{S}_{\mspace{0.5mu}\mathrm{o}}}   % Minimal Equilibria
\newcommand{\cU}{\mathcal{U}}   % Finite effort deterministic controls
\newcommand{\Lyap}{\mathcal{V}}  % Lyapunov function
\newcommand{\cX}{\mathcal{X}}   % Borel space
\newcommand{\sX}{\mathscr{X}}   % Topological space
\newcommand{\cZ}{\mathcal{Z}}   % Support set
\newcommand{\cZs}{\mathcal{Z}_{\mspace{0.5mu}\mathrm{s}}}   %Support set
\newcommand{\cZc}{\mathcal{Z}_{\mspace{0.5mu}\mathrm{c}}}   %Support set

\newcommand{\Lp}{{L}}       % L^p
\newcommand{\Lpl}{L_{\text{loc}}}       % L^p loc
\newcommand{\Sob}{{\mathscr{W}}}   % Sobolev Space
\newcommand{\Sobl}{\mathscr{W}_{\mathrm{loc}}} % Sobolev Space(loc)

\newcommand{\grad}{\nabla}

\DeclareMathOperator{\Exp}{\mathbb{E}} %Expectation
\DeclareMathOperator{\Prob}{\mathbb{P}} %Probability
\newcommand{\D}{\mathrm{d}} %differential
\newcommand{\E}{\mathrm{e}}  %exponent
\newcommand{\RR}{\mathbb{R}} %Real numbers
\newcommand{\Rd}{{\mathbb{R}^d}}
\newcommand{\NN}{\mathbb{N}}  %Natural numbers
\newcommand{\Act}{\mathbb{U}}    %Action Set
\newcommand{\Uadm}{\mathfrak{U}} %Admissible Controls
\newcommand{\Usm}{\mathfrak{U}_{\mathrm{SM}}}  %Stationary Markov controls
\newcommand{\Ussm}{\mathfrak{U}_{\mathrm{SSM}}} %Stable stationary Markov controls
\newcommand{\bUssm}{\overline{\mathfrak{U}}_{\mathrm{SSM}}} %Stable for LQG

\newcommand{\sF}{\mathfrak{F}}  %sigma field
\newcommand{\Ind}{\mathds{1}}        % indicator function
\newcommand{\transp}{^{\mathsf{T}}}  %transpose

\newcommand{\abs}[1]{\lvert#1\rvert}
\newcommand{\norm}[1]{\lVert#1\rVert}
\newcommand{\order}{{\mathscr{O}}} % Order of
\newcommand{\sorder}{{\mathfrak{o}}} % small Order of
\newcommand{\babs}[1]{\bigl\lvert#1\bigr\rvert}

\newcommand{\babss}[1]{\biggl\lvert#1\biggr\rvert}

\newcommand{\df}{:=}
\DeclareMathOperator*{\Argmin}{Arg\,min}
\DeclareMathOperator*{\argmin}{arg\,min}

\DeclareMathOperator{\diag}{diag}
\DeclareMathOperator{\trace}{trace}
\DeclareMathOperator{\dist}{dist}

\endlocaldefs   

\begin{document}
\begin{frontmatter}

\title
{Controlled equilibrium selection in\\
stochastically perturbed dynamics}
\runtitle{Controlled equilibrium selection}

\begin{aug}
\author{\fnms{Ari} \snm{Arapostathis}\corref{}
\thanksref{e1}\ead[label=e1,mark]{ari@ece.utexas.edu}},
\author{\fnms{Anup} \snm{Biswas}
\thanksref{e2}\ead[label=e2,mark]{anup@iiserpune.ac.in}},
\and
\author{\fnms{Vivek S.} \snm{Borkar}
\thanksref{e3}\ead[label=e3,mark]{borkar@ee.iitb.ac.in}}
\affiliation{The University of Texas at Austin\thanksmark{m1},\\
Indian Institute of Science Education and Research, Pune\thanksmark{m2},\\ and
Indian Institute of Technology Bombay\thanksmark{m3}}

\runauthor{A.~Arapostathis, A.~Biswas, and V.~S.~Borkar}

\address{Dept.\ of Electrical and Computer Eng.\\
The University of Texas at Austin\\
1616 Guadalupe St., UTA 7.508\\
Austin, TX~~78701\\
\printead{e1}}

\address{Indian Institute of Science Education\\ and Research\\
Dr. Homi Bhabha Road\\
Pune 411008, India\\
\printead{e2}}

\address{Department of Electrical Engineering\\
Indian Institute of Technology\\
Powai, Mumbai,
India\\
\printead{e3}}
\end{aug}

%%%%%%%%%%%%%%%%%%%%%%%%%%%%%%%%%%%%%%%%%%%%%%%%%%%%%%%%%%%%%%%%%%%%%%%%%%%%%%%%
\begin{abstract}
We consider a dynamical system with finitely
many equilibria and perturbed by small noise, in addition to being
controlled by an `expensive' control.
The controlled process is optimal for an ergodic criterion
with a running cost that consists of the sum of
the control effort and a penalty function on the state space.
We study the optimal stationary distribution of the controlled process
as the variance of the noise becomes vanishingly small.
It is shown that depending
on the relative magnitudes of the noise variance and the `running
cost' for control, one can identify three regimes, in each
of which the optimal control forces the invariant distribution of
the process to concentrate near equilibria that can be characterized according
to the regime.
We also obtain moment bounds for the optimal stationary distribution.
Moreover, we show that in the vicinity of
the points of concentration the density of optimal stationary distribution approximates
the density of a Gaussian, and we explicitly solve for its covariance matrix.
\end{abstract}

\begin{keyword}[class=MSC]
\kwd[Primary ]{35R60}
\kwd[; secondary ]{93E20}
\end{keyword}

\begin{keyword}
\kwd{Controlled diffusion}
\kwd{equilibrium selection}
\kwd{large deviations}
\kwd{small noise}
\kwd{expensive controls}
\kwd{ergodic control}
\kwd{HJB equation}
\kwd{ergodic LQG}
\end{keyword}

\end{frontmatter}

\tableofcontents

%%%%%%%%%%%%%%%%%%%%%%%%%%%%%%%%%%%%%%%%%%%%%%%%%%%%%%%%%%%%%%%%%%%%%%%%%%%%%%%%
\section{Introduction} %\label{S1}

The study of dynamical systems has a long and profound history.
A lot of effort has been devoted to understand the behavior
of the system when it is perturbed by an additive noise
\cite{Berg-Gent,FrWe,Oliv-Var}.
Small noise diffusions have found applications in climate modeling
\cite{benzi, Berg-Gent-02}, electrical engineering \cite{Bobrov, Zei-Zak},
finance \cite{Feng} and many other areas.
Recent work on `stochastic resonance' (see, e.g., \cite{Mos}) introduces an
additional external input to the dynamics that may be viewed as a control.
This is the main motivation for the study of the model we introduce next.

%%%%%%%%%%%%%%%%%%%%%%%%%%%%%%%%%%%%%%%%%%%%%%%%%%%%%%%%%%%%%%%%%%%%%%%%%%%%%%%%
\subsection{The model}
In this paper we consider a controlled dynamical system with
small noise, which is modelled as
a $d-$dimensional controlled diffusion $X =[X_{1},\dotsc,X_d]\transp$
governed by the stochastic integral equation
\begin{equation}\label{E-sde}
X_{t} \;=\; X_{0} + \int_{0}^{t}\bigl(m(X_{s}) + \varepsilon\,U_{s}\bigr)\,\D{s}
+ \varepsilon^{\nu}W_{t}\,, \quad
t \;\ge\; 0\,. 
\end{equation}
Here all processes live in a complete probability space
$(\Omega,\sF,\Prob)$ and the data of \eqref{E-sde} satisfies
the following.
\begin{itemize}
\item[(a)] $m = [m_{1},\dotsc,m_d]\transp\colon \Rd\to\Rd$
is a bounded $\Cc^{\infty}$ function with bounded derivatives.
\smallskip
\item[(b)] $W$ is a standard Brownian motion in $\Rd$.
\smallskip
\item[(c)] $U$ is an $\Rd-$valued control process which is jointly measurable
in $(t,\omega)\in [0,\infty)\times\Bar\Omega$ (in particular it has
measurable paths), and is \emph{nonanticipative}: for
$t > s$, $W_{t} - W_{s}$ is independent of
\begin{equation*}
\sF_s\,\df\, \text{the completion of\ }
\cap_{y>s}\sigma( X_{0}, W_{r}, U_{r}\,\colon r\le y) \text{\ relative to\ }
(\sF,\Prob)\,.
\end{equation*}
Such a control is called \emph{admissible},
and we denote the set of admissible controls by $\Uadm$.
As pointed out in \cite[p.~18]{Bor}, we may, without loss of
generality, assume that an admissible $U$ is adapted to the natural
filtration of $X$.
\smallskip
\item[(d)]
$0 < \varepsilon \ll 1$.
\smallskip
\item[(e)]
$\nu  > 0$.
\end{itemize}

Let $\sR\colon\Rd\times\Rd\to\Rd$ be a \emph{running cost} of the form
\begin{equation}\label{E-cRnew}
\sR(x,u)\;\df\; \ell(x) + \frac{1}{2}\, \abs{u}^{2}\,,
\end{equation}%
\nomenclature[Ab]{$\sR(x,u)$}{running cost, equation \eqref{E-cRnew}}%
where $\ell\colon\Rd \to \RR_+$ is a prescribed smooth, Lipschitz function
satisfying the condition 
\begin{equation*}
\lim_{\abs{x}\to\infty}\ell(x) \;=\; \infty\,.
\end{equation*}
The control objective is to minimize the long run average
(or \emph{ergodic}) cost
\begin{equation}\label{cost}
\sJ^\varepsilon(U) \;\df\; \limsup_{T\to\infty}\;\frac{1}{T}\,
\Exp\biggl[\int_{0}^{T} \sR(X_{s},U_{s})\,\D{s}\biggr]\,,
\end{equation}%
\nomenclature[Aa]{$\sJ^\varepsilon(U)$}{ergodic cost, equation \eqref{cost}}%
over all admissible controls.

We view \eqref{E-sde} as a perturbation of the o.d.e.\ (for \emph{ordinary
differential equation})
\begin{equation}\label{ode}
\dot{x}(t) \;=\; m\bigl(x(t)\bigr)\,, 
\end{equation}
perturbed by the `small noise' $\varepsilon^{\nu}W_{t}$ (`small'
because $\varepsilon \ll 1$), and a control term $\varepsilon U_{t}$.
Since $\varepsilon$ is small, the optimization criterion
in \eqref{cost} implies that the control is `expensive'.
We assume that the set of non-wandering points of the flow of \eqref{ode}
consists of finitely many hyperbolic equilibria, and that these are contained
in some bounded open set which is positively invariant under the flow
(see Hypothesis~\ref{H1.1}).

For the case when the control $U\equiv0$,
Freidlin and Wentzell developed a general framework for
the analysis of small noise perturbed dynamical
systems in \cite{FrWe} that is based on the theory of large deviations.
Under a stochastic Lyapunov condition we introduce later
(Hypothesis~\ref{H1.1}), the cost is
finite for $U= 0$,
ensuring in particular that the set of controls $U\in\Uadm$
resulting in a finite value for $\sJ^\varepsilon(U)$ is nonempty.
It is quite evident from ergodic theory that for $U=0$ the limit \eqref{cost}
is the expectation of $\ell$ with
respect to the invariant probability measure of \eqref{E-sde}.

The qualitative properties of the dynamics
are best understood if we consider the special case $d=1$,
and $m =- \frac{\D F}{\D x}$
for some smooth function $F\colon \RR\to \RR$.
Then the trajectory of \eqref{ode} converges to a critical point of $F$.
In fact, generically (i.e., for $x(0)$ in an open dense set) it converges
to a stable one, i.e., to a local minimum.
If one views the graph of $F$ as a `landscape', the
local minima are the bottoms of its `valleys'.
The behavior of the stochastically perturbed (albeit uncontrolled) version of
this model, notably the analysis of where the stationary distribution
concentrates, has been of considerable interest to physicists
(see, e.g., \cite[Chapter~8]{Sch} or \cite[Chapter~6]{FrWe}).
To find the  actual support of the limit in the case of multiple equilibria,
one often looks at the large deviation properties of these
invariant measures \cite{FrWe}. 
There are several studies in literature that deal with the
large deviation principle of invariant measures of
dynamical systems.
Among the most relevant to the  present
are \cite{Sheu-86, Day-87} which
obtain a large deviation principle for invariant measures
(more precisely, invariant densities) of \eqref{E-sde} under the assumption
that there is a unique equilibrium point.
This has been extended to multiple equilibria in \cite{BiBo}.
A large deviation principle
for invariant measures for a class of reaction-diffusion systems is
established in \cite{Cerrai-05}.
However, none of the above mentioned studies have any control component in
their dynamics.

The model in \eqref{E-sde} goes a step further and considers the
full-fledged optimal control version of this, wherein one tries to
induce a preferred equilibrium behavior through a feedback control.
The reason the latter has to be `expensive' is because
this captures the physically realistic situation that one can
`tweak' the dynamics but cannot replace it by something altogether
different without incurring considerable expense.
The function $\ell$ captures the relative preference among different
points in the state space.
Thus, the model in \eqref{E-sde} is closely related to the model of
stochastic resonance
which has applications in neuron modelling, physics, electronics, physiology, etc.
We refer to \cite[Chapter~1]{Herrmann} for various applications in the
presence of small noise.
In particular our model is closely related to the celebrated FitzHugh--Nagumo
model \cite{Lindner} in the presence of noise. The control
in \eqref{E-sde} should be seen as an external input.
In practice it is convenient to take
$U$ to be periodic in time, whereas we do not impose any periodicity
constraint on $U$. 
The $\varepsilon$ factor in
the control could be interpreted as the \emph{weak modulation} in \cite{Mos}.
We refer the reader to \cite{Mos,Russell-99} for a discussion on the interplay
between noise variance and the control magnitude and its relation to
stochastic resonance.
Nonlinear control theory has been useful in understanding
classes of systems that exhibit stochastic resonance \cite{Repperger-10}.
Optimization theory has also been applied with the aim of enhancing the
stochastic resonance effect for engineered systems \cite{Repperger-06,
Repperger-09}.

In our controlled setting we are interested in achieving
a desired value of $\beta^\varepsilon_*$, reflecting the desired behavior of the
corresponding stationary distribution. 
Although one can fix a suitable penalty
function $\ell$ beforehand, we will see in Theorem~\ref{T1.11}
in Section~\ref{S1.4} that the value of
$\beta^\varepsilon_*$, as well as the concentration of
the stationary distribution, change with $\nu$. Therefore a desired value of
$\beta^\varepsilon_*$ or a desired profile of
the stationary distribution
might be obtained for some specific values of $\nu$ for small $\varepsilon$.

We also wish to point out that, since the control and noise are scaled differently,
the ergodic control problem described can be viewed
as a multi-scale diffusion problem.

%%%%%%%%%%%%%%%%%%%%%%%%%%%%%%%%%%%%%%%%%%%%%%%%%%%%%%%%%%%%%%%%%%%%%%%%%%%%%%%%
\subsubsection{Assumptions on the vector field $m$}

Recall that a continuous-time dynamical system on a topological space $\sX$
is specified by a map $\phi_t\colon \sX\to\sX$,
where $\{\phi_t\}$ is a one parameter continuous abelian group action on $\sX$
called the \emph{flow}.
A point $x\in\sX$ is called \emph{non-wandering} if for every open neighborhood
$U$ of $x$ and every time $T>0$ there exists $t>T$ such that
$\phi_{t}(U)\cap U\ne \varnothing$.

Recall also that a
critical point $z$ of a smooth vector field $m$ is called \emph{hyperbolic}
if the Jacobian matrix $Dm(z)$ of $m$ at $z$ has no eigenvalues on
the imaginary axis.
For a hyperbolic critical point $z$ of a vector field $m$,
we let $\mathcal{W}_{\text{s}}(z)$ and $\mathcal{W}_{\text{u}}(z)$
denote the stable and unstable manifolds of its flow.

The following hypothesis on the vector field $m$ is in effect throughout
the paper.

%%%%%%%%%%%%%%%%%%%%%%%%%%%%%%%%%%%%%%%%%%%%%%%%%%%%%%%%%%%%%%%%%%%%%%%%%%%%%%%%
\begin{hypothesis}\label{H1.1}
The vector field $m$ is bounded and smooth and satisfies
\begin{enumerate}
\item
The set of non-wandering points of the flow of $m$ is a finite
set $\cS=\{z_{1},\dotsc,z_{n}\}$ of hyperbolic critical points.
\smallskip
\item
If $y$ and $z$ are critical points of $m$,
then $\mathcal{W}_{\text{s}}(y)$ and $\mathcal{W}_{\text{u}}(z)$
intersect transversally (if they intersect).
\item
There exist a smooth function $\Bar\Lyap\colon \Rd \to \RR_+$
and a bounded open neighborhood of the origin
$\cK\subset\Rd$ containing $\cS$, with the following properties.
\begin{itemize}
\item[(3a)]
$c_{1} \abs{x}^{2}\le\Bar\Lyap(x)\le c_{2} (1+\abs{x}^{2})$
for some positive constants $c_{1}$, $c_{2}$, and all $x\in\cK^{c}$.
\smallskip
\item[(3b)]
$\nabla \Bar\Lyap$ is Lipschitz and satisfies
\begin{equation}\label{Lyapunov}
\bigl\langle m(x),\nabla \Bar\Lyap(x)\bigr\rangle < -\gamma \abs{x} 
\end{equation}
for some $\gamma > 0$, and all $x\in\cK^{c}$.
\end{itemize}
\end{enumerate}
\end{hypothesis}

%%%%%%%%%%%%%%%%%%%%%%%%%%%%%%%%%%%%%%%%%%%%%%%%%%%%%%%%%%%%%%%%%%%%%%%%%%%%%%%%
\begin{remark}
The vector field $m$ is assumed bounded for simplicity.
The reader however might notice that
the characterization of optimality (see Theorem~\ref{T1.4}) is based
on the regularity results in \cite{Bens-Fre-book},
and
the hypotheses in \cite[Section~4.6.1]{Bens-Fre-book}
permit $m$ to be unbounded as long as
\begin{equation*}
\limsup_{\abs{x}\to\infty}\;\frac{\abs{m(x)}^{2}}{\ell(x)}\;<\;\infty\,.
\end{equation*}
Provided that this condition is satisfied, the assumption that the drift is
bounded can be waived and all the results of this paper
hold unaltered, with the proofs requiring no major modification.
\end{remark}

The outline of the paper is as follows.
Section~\ref{S1.2} summarizes the notation, and provides
a glossary of special symbols used in the paper.
In Section~\ref{S1.5} we present an important property of LQG systems,
which plays a crucial role in the study of the critical regime
and also in the proof of Theorem~\ref{T1.13}.

In Section~\ref{S2} we discuss energy functions for gradient-like flows
(Theorem~\ref{T2.2}).
These are heavily used in the study of the subcritical regime.
The proofs of the main results comprise Sections~\ref{S3}--\ref{S5}.
Section~\ref{S3} is devoted to the study of the minimal stochastically
stable sets, Section~\ref{S4} is primarily devoted to the
proof of Theorem~\ref{T1.12}, while Section~\ref{S5} studies
the optimal stationary distribution under an appropriate scaling,
which leads to Theorem~\ref{T1.13}.
Appendix~\ref{AppA} contains the proofs of Lemma~\ref{L1.3} and 
Theorem~\ref{T1.4}, while Appendix~\ref{AppB} is devoted to the
proof of Lemma~\ref{L1.16} and Theorem~\ref{T1.18}.

%%%%%%%%%%%%%%%%%%%%%%%%%%%%%%%%%%%%%%%%%%%%%%%%%%%%%%%%%%%%%%%%%%%%%%%%%%%%%%%%
\subsection{Notation}\label{S1.2}
The following notation is used in this paper.
The symbol $\RR$ denotes the field of real numbers,
and $\NN$ denotes the set of natural numbers.
The Euclidean norm on $\Rd$ is denoted by $\abs{\,\cdot\,}$,
and $\langle\cdot\,,\cdot\rangle$ denotes the inner product.
For two real numbers $a$ and $b$, $a\wedge b\df\min(a,b)$
and $a\vee b\df\max(a,b)$.
For a matrix $M$, $M\transp$ denotes its transpose, and
$\norm{M}$ denotes the operator norm relative to the Euclidean vector norm.
Also $I$ denotes the identity matrix.

The composition of two functions $f$ and $g$ is denoted by $f\circ g$.
A ball of radius $r>0$ in $\Rd$ around a point $x$ is denoted by $B_{r}(x)$,
or as $B_{r}$ if $x=0$.
For a compact set $K$,
we let $\dist(x,K)$ denote the Euclidean distance of $x\in\Rd$ from the set $K$,
and $B_r(K)\df\{y\in\Rd\,\colon \dist(y, K)< r\}$.
For a set $A\subset\Rd$, we use
$\Bar A$, $A^{c}$, and $\partial A$ to denote the closure,
the complement, and the boundary of $A$, respectively.
We define $\Cc^{k}_{b}(\RR^{d})$, $k\ge 0$, as the set of functions
whose $i^{\mathrm{th}}$ derivatives, $i=0,1,\dotsc,k$, are continuous and bounded
in $\RR^{d}$ and denote by
$\Cc^{k}_{c}(\RR^{d})$ the subset of $\Cc^{k}_{b}(\RR^{d})$ with compact support.
The space of all probability measures on a
Polish space $\cX$ with the Prohorov topology is denoted by $\cP(\cX)$.
\nomenclature[Ad]{$\cP(\cX)$}{space of probability measures on a Polish space $\cX$}
The density of the $d$-dimensional Gaussian distribution with
mean $0$ and covariance matrix $\Sigma$ is denoted by $\rho^{~}_{\Sigma}$.

The term \emph{domain} in $\RR^{d}$
refers to a nonempty, connected open subset of the Euclidean space $\RR^{d}$.
We introduce the following notation for spaces of real-valued functions on
a domain $G\subset\RR^{d}$.
The space $\Lp^{p}(G)$, $p\in[1,\infty)$, stands for the usual Banach space
of (equivalence classes of) measurable functions $f$ satisfying
$\int_{G} \abs{f(x)}^{p}\,\D{x}<\infty$, and $\Lp^{\infty}(G)$ is the
Banach space of functions that are essentially bounded in $G$.
The standard Sobolev space of functions on $G$ whose generalized
derivatives up to order $k$ are in $\Lp^{p}(G)$, equipped with its natural
norm, is denoted by $\Sob^{k,p}(G)$, $k\ge0$, $p\ge1$.

In general if $\mathcal{Y}$ is a space of real-valued functions on a domain $G$,
$\mathcal{Y}_{\mathrm{loc}}$ consists of all functions $f$ such that
$f\varphi\in\mathcal{Y}$ for every $\varphi\in\Ccc^{\infty}(G)$,
the space of smooth functions on $G$ with compact support.
In this manner we obtain for example the space $\Sobl^{2,p}(G)$.

The symbols
$\order(\abs{x}^a)$ and $\sorder(\abs{x}^a)$, for
$a\in(0,\infty)$, denote the sets
of functions $f\colon\Rd\to\RR$ having the property
\begin{equation*}
\limsup_{\abs{x}\searrow0}\;\frac{\abs{f(x)}}{\abs{x}^a}\;<\;\infty\,,
\qquad\text{and}\qquad
\limsup_{\abs{x}\searrow0}\;\frac{\abs{f(x)}}{\abs{x}^a}\;=\;0\,,
\end{equation*}
respectively.
Abusing the notation, $\order(\abs{x}^a)$ and $\sorder(\abs{x}^a)$ occasionally denote
generic members of these sets.
Thus, for example, an inequality of the form $\order(\abs{x}^2)\le f(x)
\le\order(\abs{x})$ is well defined, and is equivalent to the
statement that
$\limsup_{\abs{x}\searrow0}\;\frac{\abs{f(x)}}{\abs{x}}\;<\;\infty$,
and $\liminf_{\abs{x}\searrow0}\;\frac{\abs{f(x)}}{\abs{x}^2}\;>\;-\infty$.
\nomenclature[Fa]{$\order(\abs{x}^a)$, $\sorder(\abs{x}^a)$}{classes of functions}

Also $\kappa_{1}$, $\kappa_{2},\dotsc$ are generic constants whose definition
differs from place to place.

A glossary of commonly used symbols and the page where they are first defined
is provided below.

%%%%%%%%%%%%%%%%%%%%%%%%%%%%%%%%%%%%%%%%%%%%%%%%%%%%%%%%%%%%%%%%%%%%%%%%%%%%%%%%
%\printnomenclature[0.9in]

\begin{thenomenclature} 

 \nomgroup{A}

  \item [{$\sJ^\varepsilon(U)$}]\begingroup ergodic cost, equation \eqref{cost}\nomeqref {1.3}
		\nompageref{3}
  \item [{$\sR(x,u)$}]\begingroup running cost, equation \eqref{E-cRnew}\nomeqref {1.2}
		\nompageref{2}
  \item [{$\sR[v](x)$}]\begingroup running cost under control $v$, equation \eqref{E-sR}\nomeqref {3.1}
		\nompageref{19}
  \item [{$\cP(\cX)$}]\begingroup space of probability measures on a Polish space $\cX$\nomeqref {1.5}
		\nompageref{5}
  \item [{$\fP^\varepsilon$}]\begingroup set of infinitesimal ergodic occupation measures, equation \eqref{E-eom}\nomeqref {1.9}
		\nompageref{7}
  \item [{$\Phi^U_t$}]\begingroup set of mean empirical measures, equation \eqref{E-emp}\nomeqref {1.11}
		\nompageref{7}
  \item [{$\sJ^\varepsilon_\uppi$, $\sJ^\varepsilon_*$}]\begingroup objective and optimal value of primal problem, equation \eqref{EE1.9}\nomeqref {1.12}
		\nompageref{8}

 \nomgroup{B}

  \item [{$\Lg^{\varepsilon}_{0}$}]\begingroup operator\nomeqref {1.5}
		\nompageref{6}
  \item [{$\Lg^\varepsilon$}]\begingroup operator, equation \eqref{E-Lg}\nomeqref {1.10}
		\nompageref{7}
  \item [{$\Lg^{\varepsilon}_v$}]\begingroup operator, equation \eqref{E-Lgv}\nomeqref {1.15}
		\nompageref{9}

 \nomgroup{C}

  \item [{$V^\varepsilon$}]\begingroup solution of the HJB, equation \eqref{HJBerg}\nomeqref {1.13}
		\nompageref{8}
  \item [{$\widehat{V}_{z}^{\varepsilon}$, $\widetilde{V}^{\varepsilon}$, $\Breve{V}_{z}^{\varepsilon}$}]\begingroup scaled solutions of the HJB, Definition~\ref{D4.3}\nomeqref {4.13}
		\nompageref{29}
  \item [{$\beta^\varepsilon_*$}]\begingroup optimal value for the ergodic problem, equation \eqref{E-Birk}\nomeqref {1.14}
		\nompageref{8}
  \item [{$\eta^\varepsilon_*$}]\begingroup optimal stationary distribution, Theorem~\ref{T1.4}\nomeqref {1.14}
		\nompageref{9}
  \item [{$v^\varepsilon_*$}]\begingroup optimal stationary Markov control, Theorem~\ref{T1.4}\nomeqref {1.14}
		\nompageref{9}
  \item [{$\varrho^\varepsilon_*$}]\begingroup density of optimal stationary distribution\nomeqref {1.16}
		\nompageref{9}
  \item [{$\Hat{\eta}_{z}^{\varepsilon}$, $\mathring{\eta}_{z}^{\varepsilon}$}]\begingroup scaled optimal stationary distributions, Definition~\ref{D5.1}\nomeqref {5.0}
		\nompageref{33}
  \item [{$\Hat{\varrho}_{z}^{\varepsilon}$, $\mathring\varrho^\varepsilon_{z}$}]\begingroup scaled optimal densities, Definition~\ref{D5.1}\nomeqref {5.0}
		\nompageref{33}
  \item [{$\widehat{m}_{z}^{\varepsilon}$, $\widehat{\ell}_{z}^{\varepsilon}$}]\begingroup scaled vector field and potential, Definition~\ref{D4.3}\nomeqref {4.13}
		\nompageref{29}
  \item [{$\sG^\varepsilon_*$}]\begingroup optimal control effort, equation \eqref{E-sG}\nomeqref {1.18}
		\nompageref{10}
  \item [{$\zeta^{\varepsilon}$, $\xi^\varepsilon_1$, $\xi^\varepsilon_2$}]\begingroup constants, equation \eqref{EL3.6C}\nomeqref {3.38}
		\nompageref{26}

 \nomgroup{E}

  \item [{$\Lyap$}]\begingroup energy function, Lemma~\ref{L2.3}\nomeqref {2.6}
		\nompageref{18}
  \item [{$\cS$ ($\cSs$)}]\begingroup set of equilibria (stable equilibria) of \eqref{ode}, Definition~\ref{D1.8}\nomeqref {1.16}
		\nompageref{10}
  \item [{$\sS$}]\begingroup minimal stochastically stable set, Definition~\ref{D1.8}\nomeqref {1.16}
		\nompageref{10}
  \item [{$\cZc$, $\cZs$, $\cZ$, $\widetilde\cZ$}]\begingroup classes of equilibria, Definition~\ref{D1.10}\nomeqref {1.18}
		\nompageref{10}
  \item [{$\fJc$, $\fJs$, $\fJ$, $\widetilde\fJ$}]\begingroup Definition~\ref{D1.10}\nomeqref {1.18}
		\nompageref{10}

 \nomgroup{F}

  \item [{$\order(\abs{x}^a)$, $\sorder(\abs{x}^a)$}]\begingroup classes of functions\nomeqref {1.5}
		\nompageref{5}

 \nomgroup{G}

  \item [{$\varLambda^+(M)$}]\begingroup trace of unstable spectrum of a matrix $M$, Definition~\ref{D1.9}\nomeqref {1.16}
		\nompageref{10}
  \item [{$M_z$, $Dm(z)$}]\begingroup Jacobian of vector field $m(z)$, Definition~\ref{D1.9}\nomeqref {1.16}
		\nompageref{10}
  \item [{$\widehat{Q}_{z}$, $\widehat{\Sigma}_{z}$}]\begingroup symmetric matrices, equation \eqref{ED1.9}\nomeqref {1.17}
		\nompageref{10}

\end{thenomenclature}

%%%%%%%%%%%%%%%%%%%%%%%%%%%%%%%%%%%%%%%%%%%%%%%%%%%%%%%%%%%%%%%%%%%%%%%%%%%%%%%%
\subsection{The optimal stationary distribution}\label{S1.3}

Recall the function $\Bar\Lyap$ defined in Hypothesis~\ref{H1.1}.
Since $\nabla \Bar\Lyap$ is Lipschitz, $\Delta \Bar\Lyap$ is bounded and thus
\eqref{Lyapunov} implies that with
\begin{equation*}
\Lg^{\varepsilon}_{0}\, f(x) \;\df\;  \frac{\varepsilon^{2\nu}}{2}\,\Delta f(x)
+ \bigl\langle m(x), \nabla f(x)\bigr\rangle
\qquad\forall\,x\in\Rd\,, \quad f \in \Cc^{2}(\Rd)\,,
\end{equation*}%
\nomenclature[Ba]{$\Lg^{\varepsilon}_{0}$}{operator}%
we have
\begin{equation*}
\Lg^{\varepsilon}_{0}\,\Bar\Lyap(x) \;\le\; \gamma_0-\gamma\, \abs{x}
\qquad\forall\,\varepsilon\in(0,1)\,,
\end{equation*}
for some positive constants $\gamma$ and $\gamma_0$. This Foster--Lyapunov condition
implies in particular that the process $X$ with
$U = 0$ has a unique invariant probability measure
$\eta_{0}^{\varepsilon}$, and
\begin{equation}\label{EE1.7}
\lim_{T\to\infty}\;\frac{1}{T}\,\Exp\biggl[\int_{0}^T \abs{X_{t}}\,\D{t}\biggr]
\;=\;\int_{\Rd}\abs{x}\,\eta_{0}^{\varepsilon}(\D{x}) \;\le\;
\frac{\gamma_0}{\gamma}\qquad\forall\,\varepsilon\in(0,1)\,.
\end{equation}
Since $\ell$ is Lipschitz, \eqref{EE1.7}
implies that there exists a constant $\Bar{c}_\ell$
independent of $\varepsilon$ such that
\begin{equation}\label{EE1.8}
\int\ell\,\D\eta_{0}^{\varepsilon}\;\le\;\Bar{c}_\ell\,.
\end{equation}
Moreover, from \cite{BiBo} there exists a unique Lipschitz continuous function
$Z\ge 0$, such that $\min_{\Rd}Z=0$, $Z(x)\to\infty$ as 
$\abs{x}\to\infty$ and
\begin{equation*}
Z(x)\;=\;\inf_{\phi\,\colon \phi(t)\to x_{i},\; x_{i}\in \cS}\;
\biggl[\frac{1}{2}\int_{0}^{\infty}\babs{\dot\phi(s)+m(\phi(s))}^{2}\,\D{s}
+Z(x_{i})\biggr]\,,\qquad\phi(0)=x\,,
\end{equation*}
and if $\varrho^\varepsilon_{0}$ denotes the density of
$\eta_{0}^{\varepsilon}$, then
$-\varepsilon^{2\nu}\ln\varrho^\varepsilon_{0}(x)\to Z(x)$
uniformly on compact subsets of $\Rd$ as $\varepsilon\searrow 0$.
The function $Z$ is generally referred to as the \textit{quasi-potential},
and plays a key role in the study of $\eta_{0}^{\varepsilon}$.

For the model in \eqref{E-sde} under the optimal control criterion in
\eqref{cost}, the standard method of analysis using quasi-potentials no longer
applies.
The first important step
is to characterize the stationary probability distributions of
the controlled diffusion under optimal controls.
It is evident that optimal controls belong to the class $\widehat\Uadm$
defined by
\begin{equation}\label{E-HUadm}
\widehat\Uadm\;\df\; \biggl\{U\in\Uadm\,\colon
\Exp\biggl[\int_{0}^{t}\,\abs{U_{s}}^{2}\,\D{s}\biggr] < \infty\text{\ \ for all\ \ }
t\, \ge 0\biggr\}\,.
\end{equation}
We state the following result concerning the existence
of solutions to \eqref{E-sde}.

%%%%%%%%%%%%%%%%%%%%%%%%%%%%%%%%%%%%%%%%%%%%%%%%%%%%%%%%%%%%%%%%%%%%%%%%%%%%%%%%
\begin{lemma}\label{L1.3}
Under any $U\in\widehat\Uadm$, the diffusion in \eqref{E-sde} has a unique
strong solution.
\end{lemma}

\begin{proof}
See Appendix~\ref{AppA}.
\end{proof}

\subsubsection{The convex analytic approach}
In studying this problem, it is of course of paramount importance
to assert the existence of an optimal stationary distribution,
and ideally also prove that it is unique.

A proper framework for this study is to consider the class $\fP^\varepsilon$
of \emph{infinitesimal ergodic occupation measures}, i.e.,
measures $\uppi\in\cP(\Rd\times\Rd)$ which satisfy
\begin{equation}\label{E-eom}
\int_{\Rd\times\Rd} \Lg^{\varepsilon}[f](x,u)\,\uppi(\D{x},\D{u})
\;=\;0 \qquad \forall\,f\in\Ccc^{\infty}(\Rd)\,,
\end{equation}%
\nomenclature[Ae]{$\fP^\varepsilon$}{set of infinitesimal ergodic occupation measures,
equation \eqref{E-eom}}%
where $\Ccc^{\infty}(\Rd)$, as defined
in Section~\ref{S1.2} denotes the class of real-valued smooth functions
with compact support.
Here, the operator $\Lg^{\varepsilon}\colon\Cc^{2}(\Rd)\to \Cc(\Rd\times\Rd)$
is defined by
\begin{equation}\label{E-Lg}
\Lg^{\varepsilon} [f](x, u)\;\df\; \frac{\varepsilon^{2\nu}}{2}\Delta f(x)
+ \bigl\langle m(x) +\varepsilon u, \grad f(x)\bigr\rangle
\end{equation}%
\nomenclature[Bb]{$\Lg^\varepsilon$}{operator, equation \eqref{E-Lg}}%
for $f\in\Cc^{2}(\Rd)$.
We adopt the usual relaxed control framework,
where an admissible control is realized as a $\cP(\Rd)$-valued
measurable function (for details see \cite[Section~2.3]{book}).
Thus if we disintegrate $\uppi\in\fP^\varepsilon$ as
\begin{equation*}
\uppi(\D{x},\D{u}) \;=\; \eta(\D{x})\,v(\D{u}\,|\, x)\,,
\end{equation*}
and denote this as $\uppi=\eta\circledast v$, then $v$ is a relaxed Markov control,
and $\eta\in\cP(\Rd)$ is an invariant probability measure for the corresponding
controlled process,
provided that the diffusion under the
control $v$ in \eqref{E-sde} has a unique weak
solution for all $t\in[0,\infty)$ which is a Feller process.

Define
\begin{equation*}
\sJ^\varepsilon_\uppi\;\df\; \int_{\Rd\times\Rd} \sR(x,u)\,\uppi(\D{x},\D{u})\,,
\qquad \uppi\in\fP^\varepsilon\,.
\end{equation*}
For a control $U\in\Uadm$ under which the diffusion has
a unique weak solution
we define the collection of \emph{mean empirical measures}
$\Phi^U_t\in\cP(\Rd\times\Rd)$ by
\begin{equation}\label{E-emp}
\int_{\Rd\times\Rd} f(x,u)  \Phi^U_t(\D{x},\D{u})
\;=\; \Exp \biggl[\int_0^t f(X_s,U_s)\,\D{s}\biggr]
\end{equation}
for all $f\in\Cc_b(\Rd\times\Act)$.
\nomenclature[Af]{$\Phi^U_t$}{set of mean empirical measures, equation \eqref{E-emp}}
Recall that a continuous function $f\colon\RR^m\to\RR$ is called \emph{inf-compact}
if the set $\{x\in\RR^m\,\colon f(x)\le C\}$ is compact (or empty) for every
$C\in\RR$.
Suppose that that the ergodic cost $\sJ^\varepsilon(U)$  defined in
\eqref{cost} is finite.
Then the inf-compactness of $\sR(x,u)$ implies that $\{\Phi^U_t\}$ is tight in
$\cP(\Rd\times\Act)$.
It is standard to show, by following an argument similar
to the proof of Lemma~3.4.6 in \cite{book},
that any limit point $\uppi\in\cP(\Rd\times\Act)$
of $\Phi^U_t$ is an infinitesimal ergodic occupation measure.
Moreover, $\sJ^\varepsilon(U) \ge \inf_{\uppi\in\fP^\varepsilon}\,\sJ^\varepsilon_\uppi$
\cite[Theorem~3.4.6]{book}.
It is natural then to
consider the convex minimization problem
\begin{equation}\label{EE1.9}
\sJ^\varepsilon_*\;\df\;\inf_{\uppi\in\fP^\varepsilon}\;\sJ^\varepsilon_\uppi\,,
\end{equation}%
\nomenclature[Ag]{$\sJ^\varepsilon_\uppi$, $\sJ^\varepsilon_*$}{objective and
optimal value of primal problem, equation \eqref{EE1.9}}%
since $\sJ^\varepsilon_*$ provides a lower bound for $\sJ^\varepsilon(U)$.
This constitutes the \emph{primal problem}.
Since $\sR(x,u)$ is inf-compact,
$\uppi\mapsto\sJ^\varepsilon(\uppi)$ is lower semi-continuous,
and $\sJ^\varepsilon_\uppi$ is finite for at
least one $\uppi\in\fP^\varepsilon$ by \eqref{EE1.8}, it follows that
there exists some $\uppi_*^\varepsilon\in \fP^\varepsilon$
which attains the infimum in \eqref{EE1.9}.
If the disintegration of an optimal ergodic occupation measure
results in a Markov control under which \eqref{E-sde} has a solution,
then of course this infimum is attained for the ergodic control problem.
This is indeed the case, for a large class of problems where the
control takes values in a compact space.
For general results concerning this approach see \cite{BhBo, stoku}.
However, for problems when the control lives in $\Rd$, as is the case in the present
setup, it is in general difficult to show that under the Markov control associated with
$\uppi_*^\varepsilon$
the diffusion has a solution.

The dual of the infinite dimensional linear program in \eqref{EE1.9} consists
of a maximization over subsolutions of a HJB equation \cite{BhBo}.
We say that we have \emph{strong duality} if the optimal values of the
primal and the dual problems are equal.
To the best of our knowledge, there are no strong duality results
for ergodic control of diffusions where the control lives in $\Rd$.
In the next section we study the HJB equation and we establish strong duality
for the problem at hand.
Moreover, we establish the unicity of the optimal ergodic occupation measure
$\uppi_*^\varepsilon=\eta^\varepsilon_*\circledast v_*^\varepsilon$.
This of course implies that there exist a unique `optimal' stationary distribution
$\eta^\varepsilon_*$ and an a.e.\ unique optimal stationary Markov control,
and it turns out from the study of the HJB that this control is smooth.

\subsubsection{The HJB equation for the ergodic control problem}

Recall that a \emph{precise} stationary Markov control is specified
as $U_t=v(X_t)$ for a measurable function $v\colon\Rd\to\Rd$.
We identify the stationary Markov control with the function $v$.
Let $\Usm$ denote the class of stationary Markov controls
which are locally bounded and under
which \eqref{E-sde} has a unique strong solution for all $t\in[0,\infty)$.
Parenthetically, we note that, under a locally bounded stationary Markov control,
\eqref{E-sde} has a unique solution up to explosion time,
and it is strong Feller
\cite[Theorem~2.5]{Krylov-05}.
Linear growth of $\abs{v}$ is sufficient for the existence
of a unique strong solution for all $t\in[0,\infty)$.
We let $\Exp_{x}^{v}$ denote the expectation operator
on the canonical space of the process controlled by $v\in\Usm$,
and starting at $X_0=x$.
We say that $v\in\Usm$ is \emph{stable} if the controlled process under
$v$ is positive recurrent, and we let
$\Ussm^{\varepsilon}\subset\Usm$
denote the set of \emph{stable} controls in $\Usm$.
Parts (a)--(b) of the following theorem essentially follow from
\cite[Theorem~2.2]{Ichihara-11}.

%%%%%%%%%%%%%%%%%%%%%%%%%%%%%%%%%%%%%%%%%%%%%%%%%%%%%%%%%%%%%%%%%%%%%%%%%%%%%%%%
\begin{theorem}\label{T1.4}
There exists a critical value $\beta^\varepsilon_*\in\RR$ such that
the HJB equation for the ergodic control problem given by
\begin{equation}\label{HJBerg}
\frac{\varepsilon^{2\nu}}{2}\Delta V^{\varepsilon}
+ \min_{u\in\Rd}\;\Bigl[
\langle m + \varepsilon u, \nabla V^{\varepsilon}\rangle
+ \ell + \tfrac{1}{2}\,\abs{u}^{2}\Bigr] \;=\; \beta^{\varepsilon}\,,
\end{equation}%
\nomenclature[Ca]{$V^\varepsilon$}{solution of the HJB, equation \eqref{HJBerg}}%
has no solution if $\beta^\varepsilon>\beta^\varepsilon_*$, while
if $\beta^\varepsilon<\beta^\varepsilon_*$ for any such solution
$V^\varepsilon$ the diffusion in \eqref{E-sde} under the  control
$v=-\varepsilon\grad V^\varepsilon$ is transient.
Moreover, the following hold.
\begin{itemize}
\item[\textup{(}a\/\textup{)}]
If $V^{\varepsilon}\in\Cc^2(\Rd)$ is any solution of  \eqref{HJBerg},
then $\abs{\nabla V^{\varepsilon}(x)}$ has at most affine growth in $x$.
\smallskip
\item[\textup{(}b\/\textup{)}]
If $\beta^\varepsilon=\beta^\varepsilon_*$, then \eqref{HJBerg}
has a unique solution $V^{\varepsilon}\in\Cc^2(\Rd)$
satisfying $V^{\varepsilon}(0)=0$.
The Markov control $v^\varepsilon_*\df -\varepsilon\grad V^\varepsilon$ is stable,
and if $\eta^\varepsilon_*\in\cP(\Rd)$ denotes the invariant probability measure
of the diffusion under the control $v^\varepsilon_*$, then
\begin{equation}\label{E-Birk}
\beta^\varepsilon_* \;=\; \int_{\Rd} \sR\bigl(x,v^\varepsilon_*(x)\bigr)
\eta^\varepsilon_*(\D{x})\,.
\end{equation}%
\nomenclature[Cc]{$\beta^\varepsilon_*$}{optimal value for the
ergodic problem, equation \eqref{E-Birk}}%

\item[\textup{(}c\/\textup{)}]
\textup{(}strong duality\textup{)}
$\sJ^\varepsilon_*=\beta^\varepsilon_*$.
\smallskip
\item[\textup{(}d\/\textup{)}]
The following optimality property holds,
with $\widehat\Uadm$ as defined in \eqref{E-HUadm}.
\begin{equation*}
\liminf_{T\to\infty}\;\inf_{U\in\widehat\Uadm}\;\frac{1}{T}\,
\Exp\biggl[\int_{0}^{T} \sR(X_{s},U_{s})\,\D{s}\biggr]\;\ge\;\beta^\varepsilon_*\,,
\end{equation*}
\item[\textup{(}e\/\textup{)}]
\textup{(}uniqueness of optimal stationary distribution\textup{)}
An ergodic occupation measure
$\uppi=\eta\circledast v\in\fP^\varepsilon$ is optimal
if and only if $v$ agrees with $v^\varepsilon_*$ a.e.\ in $\Rd$.
In particular, there exists a unique optimal invariant probability measure
$\eta^\varepsilon_*$.
\end{itemize}
\end{theorem}
\nomenclature[Cd]{$v^\varepsilon_*$}{optimal stationary Markov control,
Theorem~\ref{T1.4}}
\nomenclature[Cd]{$\eta^\varepsilon_*$}{optimal stationary distribution,
Theorem~\ref{T1.4}}

\begin{proof}
The proof is contained in Appendix~\ref{AppA}.
\end{proof}

For a stationary Markov control $v$, we define the \emph{extended generator}
of \eqref{E-sde} by
\begin{equation}\label{E-Lgv}
\Lg^{\varepsilon}_v f (x) \;\df\;
\frac{\varepsilon^{2\nu}}{2}\,\Delta f(x)
+ \bigl\langle m(x)+\varepsilon v(x), \nabla f(x)\bigr\rangle\,,
\qquad x\in\Rd\,,
\end{equation}%
\nomenclature[Bc]{$\Lg^{\varepsilon}_v$}{operator, equation \eqref{E-Lgv}}%
for $f \in \Cc^{2}(\Rd)$.
It follows from \eqref{HJBerg} that
\begin{equation}\label{E-HJB2}
\frac{\varepsilon^{2\nu}}{2}\Delta V^{\varepsilon} +
\langle m, \grad V^{\varepsilon}\rangle
-\frac{\varepsilon^{2}}{2}\abs{\grad V^{\varepsilon}}^{2} +\ell
\;=\; \beta^\varepsilon_*\,.
\end{equation}

Theorem~\ref{T1.4} shows that $\beta^\varepsilon_*=\sJ^\varepsilon_*$,
and this value is attained at an a.e.\ unique
$v^\varepsilon_*\in\Ussm^\varepsilon$ and is independent of the initial
condition $X_0$.
Given these uniqueness properties, we refer to
$\eta^\varepsilon_*$ as  \emph{the optimal invariant probability measure},
or as \emph{the optimal stationary distribution},
and we let $\varrho^\varepsilon_*$ denote its density.
We also refer to  $v^\varepsilon_*$ as \emph{the optimal stationary
Markov control}, and to $\beta^\varepsilon_*$ as \emph{the optimal value}
for the ergodic problem.

\nomenclature[Ce]{$\varrho^\varepsilon_*$}{density of optimal stationary distribution}

%%%%%%%%%%%%%%%%%%%%%%%%%%%%%%%%%%%%%%%%%%%%%%%%%%%%%%%%%%%%%%%%%%%%%%%%%%%%%%%%
\begin{remark}
Due to the smoothness of coefficients, every weak solution
in $V^\varepsilon\in\Sobl^{1,\infty}(\Rd)$ of \eqref{HJBerg}
is automatically in $\Cc^k(\Rd)$ for any $k\in\NN$.
In the interest of notational economy, we often refer to any such $V^\varepsilon$
as a solution, without specifying the function space it belongs to.
\end{remark}

%%%%%%%%%%%%%%%%%%%%%%%%%%%%%%%%%%%%%%%%%%%%%%%%%%%%%%%%%%%%%%%%%%%%%%%%%%%%%%%%
\begin{remark}
Existence and uniqueness of the solution to \eqref{HJBerg} is well
known \cite{Bens, Bens-Fre-book} and in fact, the results in
\cite{Bens-Fre-book} hold
for a more general class of HJB equations.
However, we were not able to find any reference that establishes
the verification of optimality results in Theorem~\ref{T1.4},
nor strong duality.

Note also that Theorem~\ref{T1.4}\,(d) asserts a much stronger
optimality property than the usual one.
This can be in fact strengthened to pathwise optimality, and assert
that the most ``pessimistic'' pathwise performance under $v^\epsilon_*$ is no worse
than the most ``optimistic'' pathwise performance under any control
in $\widehat\Uadm$.  The proof of this fact is identical to the proofs
of Lemma~3.4.6 and Theorem~3.4.7 in \cite{book}.

Recent work as in \cite{Ichihara-12,Ichihara-13} which investigates
the optimal control problem, does not exactly fit our model.
A strict growth condition for $\ell$ is imposed in Assumption~(H2)
of \cite{Ichihara-12}, which we do not require here.
On the other hand, in \cite{Ichihara-13} where convergence of the Cauchy
problem is investigated, and therefore optimality for the ergodic control
problem is addressed, a more stringent condition is imposed
(see Hypothesis (A3)$^\prime$) which for a Hamiltonian that
is quadratic in the gradient like ours, amounts to geometric ergodicity
under the uncontrolled dynamics.

The existence of a critical value for $\beta^\varepsilon$ for
\eqref{HJBerg} and the behavior of the solutions
above or below this critical value are studied in detail in \cite{Ichihara-11}.
However, the critical value is not necessarily the optimal value.
For more recent work on the relation of the
critical value of an elliptic HJB equation of the ergodic type
and the optimal value of the control problem
see \cite{Ichihara-15}. 
\end{remark}

%%%%%%%%%%%%%%%%%%%%%%%%%%%%%%%%%%%%%%%%%%%%%%%%%%%%%%%%%%%%%%%%%%%%%%%%%%%%%%%%
\subsection{Main results}\label{S1.4}

In this section we summarize the main results of the paper.
We start with the following definition.

%%%%%%%%%%%%%%%%%%%%%%%%%%%%%%%%%%%%%%%%%%%%%%%%%%%%%%%%%%%%%%%%%%%%%%%%%%%%%%%%
\begin{definition}\label{D1.8}
Let $\cSs \subset \cS$
denote the set of stable equilibria of \eqref{ode},
i.e., the set of points $z \in \cS$ for which the
eigenvalues of $Dm(z)$ have negative real parts.
\nomenclature[Eb]{$\cS$ ($\cSs$)}{set of equilibria (stable equilibria) of
\eqref{ode}, Definition~\ref{D1.8}}

We say that a set $K\subset\Rd$ is \emph{stochastically stable}
(or that $\eta_{*}^{\varepsilon}$ \emph{concentrates} on $K$)
if it is compact, and for any open neighborhood $\cN\supset K$
we have
$\lim_{\varepsilon\searrow0}\,\eta_{*}^{\varepsilon}(\cN)=1$.
If $\mathfrak{H}$ denotes the class of stochastically stable sets,
and $\sS\df\cap_{K\in\mathfrak{H}} K$, then
$\sS$ is stochastically stable (Remark~\ref{R1.9}).
We refer to $\sS$ as the \emph{minimal stochastically stable set}.
\nomenclature[Ec]{$\sS$}{minimal stochastically stable set, Definition~\ref{D1.8}}
\end{definition}

\begin{remark}\label{R1.9}
It is straightforward to show that $\sS$ in Definition~\ref{D1.8}
is stochastically stable.
This goes as follows.
For a set $K\subset\Rd$, and $\delta>0$, let $K^{\delta}$ denote the
open $\delta$-neighborhood of $K$, i.e.,
$K^{\delta} \df \{x\in\Rd\colon d(x,K)<\delta\}$, where $d(\cdot,\cdot)$
is  the Euclidean distance.
Since the collection $\mathfrak{H}$ consists of
compact sets, it follows there exists a finite subcollection
$K^{\delta}_1,\dotsc,K^{\delta}_n$ whose intersection lies in $\sS^{2\delta}$.
Then $\eta_{*}^{\varepsilon}\bigl((\sS^{2\delta})^c\bigr) \;\le\;
\cup_{i=1}^n \eta_{*}^{\varepsilon}\bigl((K^{\delta}_i)^c\bigr)$,
from which it follows, since $\delta>0$ is arbitrary,
that $\sS$ is stochastically stable.
\end{remark}

The behavior of $\eta_{*}^{\varepsilon}$ for small $\varepsilon$ depends
crucially on the parameter $\nu$.
We distinguish three regimes: The \emph{supercritical regime} ($\nu>1$),
the \emph{subcritical regime} ($\nu<1$), and the
\emph{critical regime} ($\nu=1$).
Roughly speaking, the control `exceeds' the noise level in the supercritical
regime, while the opposite is the case in the subcritical regime.
In the critical regime, which is the most interesting and more difficult
to study, the control and noise levels are equal.
The main results can be grouped in three categories:
(1) characterization of the minimal stochastically stable set
$\sS$ and asymptotic estimates of $\beta^\varepsilon_*$
for small $\varepsilon$ in the three regimes
(Theorem~\ref{T1.11}),
(2) concentration bounds for $\eta_{*}^{\varepsilon}$ (Theorem~\ref{T1.12}),
and (3) convergence of $\varrho^\varepsilon_*$, under appropriate
scaling, to a Gaussian density (Theorem~\ref{T1.13}).

%%%%%%%%%%%%%%%%%%%%%%%%%%%%%%%%%%%%%%%%%%%%%%%%%%%%%%%%%%%%%%%%%%%%%%%%%%%%%%%%
\begin{definition}\label{D1.9}
For a square matrix $M\in\RR^{d\times d}$,
let $\varLambda^+(M)$ denote the sum of its eigenvalues
that lie in the open right half complex plane.
\nomenclature[Ga]{$\varLambda^+(M)$}{trace of unstable spectrum of a matrix
$M$, Definition~\ref{D1.9}}
For $z\in\cS$, and with $M_z\df Dm(z)$,
where as defined earlier $Dm(z)$ is the Jacobian of $m$ at $z$,
\nomenclature[Gb]{$M_z$, $Dm(z)$}{Jacobian of vector field $m(z)$,
Definition~\ref{D1.9}}
we let $\widehat{Q}_{z}$ and  $\widehat{\Sigma}_{z}$
be the symmetric, nonnegative definite, square matrices solving
the pair of equations
\begin{equation}\label{ED1.9}
\begin{split}
M_z\transp \widehat{Q}_{z}
&+ \widehat{Q}_{z} M_z \;=\; {\widehat{Q}}_{z}^{\,2}\,,\\[5pt]
\bigl(M_z-\widehat{Q}_{z}\bigr)\,\widehat{\Sigma}_{z}
&+ \widehat{\Sigma}_{z}\,\bigl(M_z-\widehat{Q}_{z}\bigr)\transp \;=\; - I\,.
\end{split}
\end{equation}
\nomenclature[Gc]{$\widehat{Q}_{z}$, $\widehat{\Sigma}_{z}$}{symmetric
matrices, equation \eqref{ED1.9}}
\end{definition}

By Theorem~\ref{T1.18}, which appears in Section~\ref{S1.5},
there exists a unique pair
$(\widehat{Q}_{z},\widehat{\Sigma}_{z})$
of symmetric positive semidefinite matrices solving \eqref{ED1.9}.
It is also evident by \eqref{ED1.9} that $\widehat{\Sigma}_{z}$ is
invertible.

In order to state the main results we need the following definition.

%%%%%%%%%%%%%%%%%%%%%%%%%%%%%%%%%%%%%%%%%%%%%%%%%%%%%%%%%%%%%%%%%%%%%%%%%%%%%%%%
\begin{definition}\label{D1.10}
We define the \emph{optimal control effort} $\sG^\varepsilon_*$ by
\begin{equation}\label{E-sG}
\sG^\varepsilon_* \;\df\;\frac{1}{2}\,
\int_{\Rd}\abs{v^\varepsilon_*}^{2}\,\D\eta_{*}^{\varepsilon}\,,
\qquad \varepsilon>0\,.
\end{equation}
\nomenclature[Cj]{$\sG^\varepsilon_*$}{optimal control effort, equation \eqref{E-sG}}%
Also define
\begin{align*}
\cZc &\;\df\; \Argmin_{z\in\cS}\;
\bigl\{\ell(z)+\varLambda^+\bigl(Dm(z)\bigr)\bigr\}\,,
 &\fJc  &\;\df\; \min_{z\in \cS}\;
\bigl[\ell(z)+\varLambda^+\bigl(Dm(z)\bigr)\bigr]\,,\\
\cZs &\;\df\; \Argmin_{z\in\cSs}\;\bigl\{\ell(z)\bigr\}\,,
 &\fJs &\;\df\; \min_{z\in \cSs}\;\bigl[\ell(z)\bigr]\,,\\
\cZ &\;\df\;\Argmin_{z\in\cS}\;\bigl\{\ell(z)\bigr\}\,,
 &\fJ &\;\df\; \min_{z\in \cS}\;\bigl[\ell(z)\bigr]\,,\\
\widetilde\cZ &\;\df\; \Argmin_{z\in\cZ}\;\bigl\{\varLambda^+\bigl(Dm(z)\bigr)\bigr\}\,,
 &\widetilde\fJ &\;\df\;
\min_{z\in \cZ}\;\bigl[\varLambda^+\bigl(Dm(z)\bigr)\bigr]\,.
\end{align*}%
\nomenclature[Ed]{$\cZc$, $\cZs$, $\cZ$, $\widetilde\cZ$}{classes of equilibria,
Definition~\ref{D1.10}}%
\nomenclature[Ee]{$\fJc$, $\fJs$, $\fJ$, $\widetilde\fJ$}{Definition~\ref{D1.10}}%
\end{definition}

Recall the definition of $\order(\,\cdot\,)$ in Section~\ref{S1.2}.
The following theorem provides a comprehensive characterization of
the minimal stochastically stable set.

%%%%%%%%%%%%%%%%%%%%%%%%%%%%%%%%%%%%%%%%%%%%%%%%%%%%%%%%%%%%%%%%%%%%%%%%%%%%%%%%
\begin{theorem}\label{T1.11}
The minimal stochastically stable set $\sS$ is a subset of $\cS$
for all $\nu>0$.
Also, the set $\sS$, the optimal value $\beta^\varepsilon_*$,
and the optimal control effort $\sG^\varepsilon_*$
depend on $\nu$ as follows.
\begin{enumerate}
\item[{\upshape(}i\/{\upshape)}]
For $\nu> 1$ {\textup(}`supercritical' regime{\textup)},
we have $\sS\,\subset\,\widetilde\cZ$.
In addition, if $\fJ=\fJs$, then
\begin{equation*}
\order\bigl(\varepsilon^{2\wedge\nu}\bigr)\;\le\;
\beta^\varepsilon_*-\fJ \;\le\;
\order\bigl(\varepsilon^{2\nu}\bigr)\,,\quad\text{and}\quad
\sG^\varepsilon_*\;\in\;\order\bigl(\varepsilon^{\nu\wedge2}\bigr)\,,
\end{equation*}
and if $\fJ<\fJs$, then
\begin{equation*}
\order\bigl(\varepsilon^{2\wedge\nu}\bigr)\;\le\;
\beta^\varepsilon_*-\fJ \;\le\;\varepsilon^{2\nu-2}\,\widetilde\fJ+
\order\bigl(\varepsilon^{2\nu}\bigr)\,,\quad\text{and}\quad
\sG^\varepsilon_*\;\in\;\order\bigl(\varepsilon^{(2\nu-2)\wedge2}\bigr)\,.
\end{equation*}

\item[\textup{(}ii\/{\upshape)}]
For $\nu < 1$ {\textup(}`subcritical' regime{\textup)}, we have
$\sS\,\subset\,\cZs$, and
\begin{equation}\label{ET1.11sub}
\order\bigl(\varepsilon^{\nu}\bigr)\;\le\;
\beta^\varepsilon_*-\fJs \;\le\;
\order\bigl(\varepsilon^{\nu\vee(4\nu-2)}\bigr)\,,
\qquad \sG^\varepsilon_*\;\in\;\order\bigl(\varepsilon^{\nu}\bigr)\,.
\end{equation}
\item[{\upshape(}iii\/{\upshape)}]
For $\nu = 1$
{\textup(}`critical' regime{\textup)},
we have  $\sS\,\subset\,\cZc$,
$\beta^\varepsilon_*\le\fJc+\order\bigl(\varepsilon^2\bigr)$, and
$\lim_{\varepsilon\searrow0}\,\beta^\varepsilon_*=\fJc$.
Moreover, if $\fJc=\fJs$, then the lower bound in \eqref{ET1.11sub} holds.
\end{enumerate}
\end{theorem}

It is not hard to show that the optimal invariant measures
$\eta^{\varepsilon}_{*}$ concentrate on $\cS$ as $\varepsilon\searrow 0$
(see Lemma~\ref{L3.1}).
In Theorem~\ref{T1.11} we distinguish the three regimes corresponding
to different values of $\nu$, and provide asymptotic bounds for
$\beta^{\varepsilon}_{*}$ for small $\varepsilon$.
For $\nu>1$ one can find a control $U$ under which the invariant measure of
the dynamics \eqref{E-sde} concentrates on a point in $\cS$. 
Construction of invariant measures with similar properties is also possible
for $z\in\cSs$ when $\nu<1$.
The important difference is that for $\nu<1$
the optimal invariant measure $\eta^{\varepsilon}_{*}$ cannot concentrate
on $\cS\setminus\cSs$ (see Lemma~\ref{L3.6}).
To show this fact we construct a suitable energy function for the
Morse--Smale dynamics (see Theorem~\ref{T2.2}).
The analysis in the critical regime $\nu=1$ turns out to be more
subtle than the other two regimes.
To facilitate the study of the critical regime, we identify
an important property which concerns a singular ergodic control problem
for Linear Quadratic Gaussian (LQG) systems (Theorem~\ref{T1.18}).
This plays a crucial role in showing that $\sS\subset\cZc$.

To guide the reader, we indicate the results presented in
Sections~\ref{S3}--\ref{S4}
which comprise the proof of Theorem~\ref{T1.11}.

\begin{proof}[Proof of Theorem~\ref{T1.11}]
That $\sS\subset\cS$ is the statement of Lemma~\ref{L3.1}.
Note that if $\fJ=\fJs$, then
$\widetilde\fJ =
\min_{z\in \cZ}\;\bigl[\varLambda^+\bigl(Dm(z)\bigr)\bigr]=0$
by the definition of $\Lambda^+$.
Thus upper bounds of $\beta^\varepsilon_*-\fJ$ in part (i)
follow by the first inequality in \eqref{EL3.3A},
while the lower bounds are in Corollary~\ref{C4.2}\,(b).
The statements concerning $\sG^\varepsilon_*$ in part (i)
are in \eqref{EC4.2A}.

That $\sS\subset\cZs$ in the subcritical regime is in the statement
of Lemma~\ref{L3.6}.
The upper bound of $\beta^\varepsilon_*-\fJs$ in part (ii) is the combination
of the two separate upper bounds given in Lemma~\ref{L3.5}\,(ii),
for $\nu\in(0, \nicefrac{2}{3})$ and $\nu\in[\nicefrac{2}{3}, 1)$,
while the lower bound is in Corollary~\ref{C4.6}\,(b), where we also find
the assertion that $\sG^\varepsilon_*\in\order\bigl(\varepsilon^{\nu}\bigr)$.
 
We now turn to the proof of part (iii).
The inequality $\beta^\varepsilon_*\le\fJc+\order\bigl(\varepsilon^2\bigr)$
is the second inequality in \eqref{EL3.3A}.
That $\lim_{\varepsilon\searrow0}\,\beta^\varepsilon_*=\fJc$ is
in the statement of Theorem~\ref{T5.4}, and that
$\sS\subset\cZc$ is equivalent to
$\lim_{\varepsilon\searrow0}\,\eta_{*}^\varepsilon\bigl(B^c_r(\cZ_c)\bigr)=0$,
which is asserted in \eqref{ET5.4A}.
Lastly, that $\beta^{\varepsilon}_{*}-\fJs\ge\order(\varepsilon)$ when
$\fJc=\fJs$ is in Remark~\ref{R4.7}.
\end{proof}

The next theorem provides concentration bounds for the optimal
stationary distribution
in terms of moments.
Let $\dist(x,\cS)$ denote the Euclidean distance of $x\in\Rd$ from the set $\cS$,
and $B_r(\cS)\df\{y\in\Rd\,\colon \dist(y,\cS)< r\}$.

%%%%%%%%%%%%%%%%%%%%%%%%%%%%%%%%%%%%%%%%%%%%%%%%%%%%%%%%%%%%%%%%%%%%%%%%%%%%%%%%
\begin{theorem}\label{T1.12}
For any $k\in\NN$ and $r>0$, there exist constants,
$\Hat{\kappa}_{0}=\Hat{\kappa}_{0}(k,r,\nu)$, and
$\Hat{\kappa}_{i}=\Hat{\kappa}_{i}(k)$, $i=1,2$, such that
with $\Hat{r}(\varepsilon) \df \Hat{\kappa}_{2}\varepsilon^{\nu\wedge 1}$
we have
\begin{equation}\label{T1.12A}
\begin{split}
\int_{B_{r}(\cS)} \bigl(\dist(x,\cS)\bigr)^{2}\,\eta_{*}^{\varepsilon}(\D{x})
&\;\le\; \Hat{\kappa}_{0}\,\varepsilon^{2(\nu\wedge 2)}\qquad
\forall\,\nu>0\,,\\[5pt]
\int_{B_{\Hat{r}(\varepsilon)}^{c}(\cS)}
\bigl(\dist(x,\cS)\bigr)^{2k}\, \eta_{*}^{\varepsilon}(\D{x})
&\;\le\; \Hat{\kappa}_{1}\, \varepsilon^{2(\nu\wedge 1)}
\qquad \forall\,\nu\in(0,2]\,,
\end{split}
\end{equation}
for all $\varepsilon\in(0,1)$.

Moreover, if $D$ is any open set such that
$\cSs\subset D$, then
\begin{equation*}
\eta^{\varepsilon}_*(D^c)\;\in\;
\order\bigl(\varepsilon^{2\nu\wedge(2-\nu)}\bigr)\,,
\end{equation*}
provided $\nu<1$, or $\fJc=\fJs$ and $\nu=1$, or
$\fJ=\fJs$ and $\nu\in(1,2)$.
\end{theorem}

\begin{proof}
The first inequality in \eqref{T1.12A} is the same as \eqref{EL4.1A},
while the second is established in Proposition~\ref{P4.5}.

That $\eta^{\varepsilon}_*(D^c)\in
\order\bigl(\varepsilon^{2\nu\wedge(2-\nu)}\bigr)$ when $\fJ=\fJs$ and $\nu\in(1,2)$,
or when $\nu<1$ is asserted in Corollary~\ref{C4.6}, and that the same inclusion
holds when $\fJc=\fJs$ and $\nu=1$ is explained in
Remark~\ref{R4.7}.
\end{proof}

Exploiting the results in Theorem~\ref{T1.12}, we scale the space suitably
and show that the resulting invariant measures are also tight.
In particular, we examine the asymptotic behavior of $\eta_{*}^{\varepsilon}$
and show that under an appropriate spatial scaling
it `converges' to a Gaussian distribution in the vicinity of
the minimal stochastically stable set.
This is the subject of the next theorem.

%%%%%%%%%%%%%%%%%%%%%%%%%%%%%%%%%%%%%%%%%%%%%%%%%%%%%%%%%%%%%%%%%%%%%%%%%%%%%%%%
\begin{theorem}\label{T1.13}
Assume $\nu\in(0,2)$.
Let $z\in\cS$, and $\cN$ an open neighborhood of $z$
whose closure does not contain any other elements of $\cS$.
Suppose that along some sequence
$\varepsilon_{n}\searrow0$ we have
$\liminf_{\varepsilon_{n}\searrow0}\,\eta_{*}^{\varepsilon_{n}}(\cN)>0$.
Then along this sequence it holds that
\begin{equation}\label{ET1.13}
\frac{\varepsilon^{\nu d}\,
\varrho^\varepsilon_*\bigl(\varepsilon^{\nu}x+z\bigr)}
{\eta_{*}^{\varepsilon}(\cN)}\;\xrightarrow[\varepsilon\searrow0]{}\;
\frac{1}{(2\pi)^{\nicefrac{d}{2}}\,
\abs{\det\widehat{\Sigma}_{z}}^{\nicefrac{1}{2}}}
\exp\Bigl(-\tfrac{1}{2}\,\bigl\langle x,
\widehat{\Sigma}_{z}^{-1} x\bigr\rangle\Bigr)\,,
\end{equation}
uniformly on compact sets, where `\,$\det$' denotes
the determinant, and $\widehat{\Sigma}_{z}$
is given by \eqref{ED1.9}.
\end{theorem}

\begin{proof}
This follows from Theorems~\ref{T5.3} and \ref{T5.7}.
\end{proof}

We present a simple example to demonstrate the results.

%%%%%%%%%%%%%%%%%%%%%%%%%%%%%%%%%%%%%%%%%%%%%%%%%%%%%%%%%%%%%%%%%%%%%%%%%%%%%%%%
\begin{example}\label{E1.14}
Let $m$ be a vector field in $\RR$ of the form
$m = - \nabla F$, with $F$ a `double well potential' given by
$F(x) \df \frac{x^4}{4} - \frac{x^3}{3} - x^{2}$ on $[-10, 10]$,
with $F$ suitably extended so that it is globally Lipschitz and does not
have any critical points outside the interval $[-10, 10]$.
Then $\nabla F$ vanishes at exactly three points: $-1, 0, 2$.
Of these, $0$ is a local maximum, hence an unstable equilibrium for the o.d.e.\
$\dot{x}(t) = m(x(t))$, and both $-1$ and $2$ are local minima, hence stable
equilibria thereof.
Let $\ell(x) = c\abs{x}^{2}$ on $[-10, 10]$ for a suitable $c > 0$, modified
suitably outside $[-10, 10]$ to render it globally Lipschitz.
Note that $F(0) = 0$, $F(-1) = -\frac{5}{12}$, $F(2) = -\frac{8}{3}$. 
Thus $x = 2$ is the unique global minimum of $F$.
Since $\ell(0)=0$, and $Dm(0)=2$, the results of Theorem~\ref{T1.11}
indicate that 
\begin{itemize}
\item in the supercritical regime $\sS=\{0\}$,
and $\beta^\varepsilon_* \approx \ell(0) = 0$ for $\varepsilon$ small;
\item in the subcritical regime $\sS=\{-1\}$,
$\beta^\varepsilon_* \approx \ell(-1) = c$ for $\varepsilon$ small;
\item
in the critical regime, we have $\sS=\{0\}$ if $c>2$,
with $\beta^\varepsilon_* \approx \ell(0) + Dm(0) = 2$ for $\varepsilon$ small,
and $\sS=\{-1\}$ if $c<2$, with
$\beta^\varepsilon_* \approx \ell(-1) = c$ for $\varepsilon$ small.
\end{itemize}

Next we change the data so that
$$F(x) \df \frac{x^6}{6} - \frac{x^5}{5} - \frac{7x^4}{4}+\frac{x^3}{3}
+3x^{2}\,\quad{on}\ [-10, 10]\,.$$
Then $\nabla F$ vanishes at exactly five points,
and $\cS=\{-2,-1,0,1,3\}$.
Of these, $-1$ and $1$ are local maxima of $F$, hence unstable equilibria
for the o.d.e.\
$\dot{x}(t) = m(x(t))$, while the rest are stable equilibria.
Hence $\cSs=\{-2,0,3\}$.
Let $\ell(x) = 5x^4-x^3-20x^2+16$ on $[-10, 10]$.
The critical point $z=3$ is the unique global minimum for $F$, which
means that it is stochastically stable for the uncontrolled dynamics.
Calculating the values of $\ell$ at $\cS$ we obtain
$\ell(-2)=24$, $\ell(-1)=2$, $\ell(0)=16$, $\ell(1)=0$, and $\ell(3)=214$.
Also, we have $Dm(-1)=8$, $Dm(1)=12$.
By Theorem~\ref{T1.11}, we have the following.
\begin{itemize}
\item in the supercritical regime, $\sS=\{1\}$,
and $\beta^\varepsilon_* \approx \ell(1) = 0$ for $\varepsilon$ small;
\item
in the critical regime, $\sS=\{-1\}$,
and $\beta^\varepsilon_* \approx \ell(-1)+Dm(-1) = 10$ for $\varepsilon$ small;
\item in the subcritical regime, $\sS=\{0\}$,
$\beta^\varepsilon_* \approx \ell(0) = 16$ for $\varepsilon$ small.
\end{itemize}
Note that in this example
 the stochastically stable sets are distinct in the three regimes.
\end{example}

%%%%%%%%%%%%%%%%%%%%%%%%%%%%%%%%%%%%%%%%%%%%%%%%%%%%%%%%%%%%%%%%%%%%%%%%%%%%%%%%
\begin{remark}
Theorems~\ref{T1.11}--\ref{T1.12} suggest that $\nu=2$ is a critical value.
We present an example with linear drift and quadratic penalty,
so that explicit calculations are possible,
to show that indeed $\nu=2$ is a critical
value.
Consider a one-dimensional
model with data $m(x)=x$ and $\ell(x)=(x+1)^{2}$.
Direct substitution shows that the solution of the HJB
equation (see \eqref{E-HJB2}) is
\begin{align*}
V^{\varepsilon}(x)&\;=\;
\frac{1+\sqrt{1+2\varepsilon^{2}}}{2\varepsilon^2}\,
\biggl(x + \frac{2\varepsilon^2}
{\bigl(1+\sqrt{1+2\varepsilon^{2}}\bigr)\,\sqrt{1+2\varepsilon^2}}\biggr)^{2}\,,\\[5pt]
\beta^\varepsilon_*&\;=\; \frac{1}{1+2\varepsilon^2}
+ \varepsilon^{2\nu-2}\,\frac{1+\sqrt{1+2\varepsilon^{2}}}{2}\,.
\end{align*}
The closed loop drift is
\begin{align}\label{ER1.5a}
x - \varepsilon^2 \grad V^{\varepsilon}(x)
&\;=\; -\sqrt{1+2\varepsilon^2}\, x -
\frac{2\varepsilon^2}{\sqrt{1+2\varepsilon^2}}\\[5pt]
&\;=\; -\sqrt{1+2\varepsilon^2}
\biggl(x + \frac{2\varepsilon^2}{1+2\varepsilon^2}\biggr)\,.\nonumber
\end{align}
Thus, the optimal stationary distribution $\eta^\varepsilon_*$ is Gaussian
with variance $(\upsigma^\varepsilon_*)^2$ and mean $\mathfrak{m}^\varepsilon_*$
given by
\begin{equation}\label{ER1.5b}
(\upsigma^\varepsilon_*)^2\;\df\;
\frac{\varepsilon^{2\nu}}{2\sqrt{1+2\varepsilon^2}}\,,
\qquad
\mathfrak{m}^\varepsilon_*\;\df\;
-\frac{2\varepsilon^2}{1+2\varepsilon^2}\,.
\end{equation}
Consider the scaled distribution $\Hat\eta^\varepsilon_*$ with density
$\varepsilon^{\nu}\,
\varrho_{*}^{\varepsilon}\bigl(\varepsilon^{\nu}x+z\bigr)$.
Let $\sN(\mathfrak{m},\upsigma^2)$ denote the Normal distribution
with mean $\mathfrak{m}$ and variance $\upsigma^2$.
We have
\begin{itemize}
\item
For $\nu\in(0,2)$,
$\Hat\eta^\varepsilon_*$ converges to $\sN(0,\nicefrac{1}{2})$.
\item
For $\nu=2$, $\Hat\eta^\varepsilon_*$ converges to $\sN(-2,\nicefrac{1}{2})$.
\item
For $\nu>2$, we have $\frac{\mathfrak{m}^\varepsilon_*}{\upsigma^\varepsilon_*}
\to -\infty$, and thus $\Hat\eta^\varepsilon_*$ does not converge
as $\varepsilon\searrow0$.
\end{itemize}
Thus \eqref{ET1.13} does not hold for $\nu\ge2$.

A simple calculation also shows that the optimal control effort is
given by
\begin{align*}
\sG^\varepsilon_*&\;=\; \frac{\varepsilon^{-2}}{2}
\Bigl(1+\sqrt{1+2\varepsilon^{2}}\Bigr)^2\, (\upsigma^\varepsilon_*)^2%\\[3pt]
%&\mspace{90mu}
+ \frac{\varepsilon^{-2}}{2} \Bigl(1+\sqrt{1+2\varepsilon^{2}}\Bigr)^2\,
\biggl(\frac{2\varepsilon^2}
{\bigl(1+\sqrt{1+2\varepsilon^{2}}\bigr)\,\sqrt{1+2\varepsilon^2}}
+\mathfrak{m}^\varepsilon_*\biggr)^{2}
\\[5pt]
&\;=\; \varepsilon^{2\nu-2}\,
\frac{\bigl(1+\sqrt{1+2\varepsilon^{2}}\bigr)^2}{4\sqrt{1+2\varepsilon^2}}
+\frac{2 \varepsilon^2}{\bigl(1+2\varepsilon^2\bigr)^2}\,.
\end{align*}
Thus $\sG^\varepsilon_*\;\in\;\order\bigl(\varepsilon^{(2\nu-2)\wedge2}\bigr)$,
which matches the estimate in Theorem~\ref{T1.11}\,(i).

A better understanding of this can be reached by considering the limit
$\nu\to\infty$, in which case the dynamics are deterministic.
A simple calculation shows that
\begin{equation*}
\Bar{x} \;\df\; \argmin_x\;\Bigl\{\ell(\varepsilon x) + \tfrac{1}{2} \abs{x}^2\Bigr\}
\;=\; -\frac{2\varepsilon^2}{1+2\varepsilon^2}\,.
\end{equation*}
Thus for a feedback control to be optimal, the
point $\Bar{x}$ should be asymptotically stable for the closed loop system.
As a result, for the LQG problem, the optimal stationary distribution
is  centered at the point $\Bar{x}$ for all values of $\nu$.
The criticality at $\nu=2$ is generic, since in the
vicinity of an equilibrium $z$, solving
the minimization problem we have $\Bar{x} \approx \varepsilon^2 \grad\ell(z)$.

There is a similar behavior if the drift is stable. Let $m(x)=-x$.
We obtain
\begin{align*}
V^{\varepsilon}(x)&\;=\;
\frac{-1+\sqrt{1+2\varepsilon^2}}{2\varepsilon^2}\,
\biggl(x + \frac{2\varepsilon^2}
{\bigl(-1+\sqrt{1+2\varepsilon^2}\bigr)\,\sqrt{1+2\varepsilon^2}}\biggr)^{2}\,,\\[5pt]
\beta^\varepsilon_*&\;=\; \frac{1}{1+2\varepsilon^2}
+ \varepsilon^{2\nu-2}\,\frac{-1+\sqrt{1+2\varepsilon^2}}{2}\\[5pt]
&\;=\; 1-\frac{2\varepsilon^2}{1+2\varepsilon^2}
+ \varepsilon^{2\nu}\,\frac{1}{1+\sqrt{1+2\varepsilon^2}}\,.
\end{align*}
The closed loop drift, variance, and mean are as in \eqref{ER1.5a}--\eqref{ER1.5b}.
Using the identity
\begin{equation*}
\frac{-1+\sqrt{1+2\varepsilon^2}}{2\varepsilon^2}
\;=\;\frac{1}{1+\sqrt{1+2\varepsilon^2}}\,,
\end{equation*}
the optimal control effort takes the form
\begin{align*}
\sG^\varepsilon_*&\;=\;
\frac{2 (\upsigma^\varepsilon_*)^2}
{\bigl(1+\sqrt{1+2\varepsilon^2}\bigr)^2}\, 
+ \frac{2\varepsilon^2}{\bigl(1+\sqrt{1+2\varepsilon^2}\bigr)^2}\,
\biggl(\frac{1+\sqrt{1+2\varepsilon^2}}{\sqrt{1+2\varepsilon^2}}
+\mathfrak{m}^\varepsilon_*\biggr)^{2}
\\[5pt]
&\;=\; 
\frac{ \varepsilon^{2\nu}}
{\bigl(1+\sqrt{1+2\varepsilon^2}\bigr)^2\sqrt{1+2\varepsilon^2}}
+\frac{2 \varepsilon^2}{\bigl(1+2\varepsilon^2\bigr)^2}\,.
\end{align*}
Thus $\sG^\varepsilon_*\in\order\bigl(\varepsilon^{2\nu\wedge 2}\bigr)$.
\end{remark}

%%%%%%%%%%%%%%%%%%%%%%%%%%%%%%%%%%%%%%%%%%%%%%%%%%%%%%%%%%%%%%%%%%%%%%%%%%%%%%%%
\subsection{A property of LQG systems}\label{S1.5}
As mentioned earlier, the study of the critical regime, and
also the proof of Theorem~\ref{T1.13}
rely on an important property of LQG systems which we describe next.
A matrix $M\in\RR^{d\times d}$ is called \emph{exponentially dichotomous}
if it has no eigenvalues on the imaginary axis.
Consider the diffusion
\begin{equation}\label{E-LQG}
\D{X}_{t}\;=\; \bigl(MX_{t}+v(X_{t})\bigr)\,\D{t} + \D{W}_{t}\,,
\end{equation}
with $M\in\RR^{d\times d}$ exponentially dichotomous.
Let $\bUssm$ denote the class of locally bounded stationary Markov controls $v$,
under which the diffusion in \eqref{E-LQG} has a unique strong solution,
is positive recurrent, and satisfies
\begin{equation}\label{E-effort}
\sE(v) \;\df\; \frac{1}{2}\,\int_{\Rd} \abs{v(x)}^2 \,\mu_v(\D{x})\;<\; \infty\,,
\end{equation}
where $\mu_v$ denotes the associated invariant probability measure.

As Theorem~\ref{T1.18} below asserts,
the minimal control effort, defined by
\begin{equation*}
\sE_* \;\df\; \inf_{v\in\bUssm}\; \sE(v),
\end{equation*}
which is required to render the diffusion positive recurrent by controls
in $\bUssm$, equals the
trace of the unstable spectrum of the matrix $M$,
which was denoted as $\varLambda^+(M)$ in Definition~\ref{D1.9}.
This result is related to classical results in deterministic
linear control systems and the Riccati equation \cite{Kucera-72b,Martensson-71,
Willems-71b},
but since we could not locate
it in this form in the literature, a proof is included in Appendix~\ref{AppB},
where the proof of the following auxiliary lemma is also located.

%%%%%%%%%%%%%%%%%%%%%%%%%%%%%%%%%%%%%%%%%%%%%%%%%%%%%%%%%%%%%%%%%%%%%%%%%%%%%%%%
\begin{lemma}\label{L1.16}
Provided $M$ is exponentially dichotomous,
there exists a constant $\widetilde{C}_0$ depending only on $M$
such that
\begin{equation*}
\int_{\Rd} \abs{x}^2 \,\mu_v(\D{x})\;\le\;
\widetilde{C}_0
\biggl(1 + \int_{\Rd} \abs{v(x)}^2 \,\mu_v(\D{x})\biggr)
\qquad \forall\,v\in\bUssm\,.
\end{equation*}
\end{lemma}

Recall that a real square matrix is called \emph{Hurwitz} if its eigenvalues
lie in the open left half complex plane.
We need the following definition.

%%%%%%%%%%%%%%%%%%%%%%%%%%%%%%%%%%%%%%%%%%%%%%%%%%%%%%%%%%%%%%%%%%%%%%%%%%%%%%%%
\begin{definition}\label{D1.17}
Let $M\in\RR^{d\times d}$ be fixed.
Let $\cG(M)$ denote the collection of all matrices
$G\in\RR^{d\times d}$ such that $M-G$ is Hurwitz.
For $G\in\cG(M)$, let $\Sigma_{G}$ denote
the (unique) symmetric solution of the Lyapunov equation
\begin{equation}\label{ED1.17A}
(M-G)\,\Sigma_{G} + \Sigma_{G}\,(M-G)\transp \;=\; - I\,,
\end{equation}
and define
\begin{equation}\label{ED1.17B}
\begin{split}
\cJ_G(M) &\;\df\;
\frac{1}{2}\trace\bigl( G\,\Sigma_{G}\, G\transp\bigr)\,,\\[5pt]
\cJ_*(M) &\;\df\; \inf_{G\in\cG(M)}\;\cJ_G(M)\,.
\end{split}
\end{equation}
\end{definition}

Let $v_G(x)= - G x$  for some $G\in\RR^{d\times d}$.
It is clear that for the diffusion in \eqref{E-LQG} to be positive recurrent under
the linear control $v_G$, it is necessary that $M-G$ be Hurwitz.
If so, then the invariant probability distribution of the controlled
diffusion is Gaussian with covariance matrix $\Sigma_{G}$ given
by \eqref{ED1.17A}.
It is clear then that the control effort $\sE(v_G)$ defined in
\eqref{E-effort} satisfies $\sE(v_G)=\cJ_G(M)$. 
Therefore, provided the infimum in \eqref{ED1.17B} is attained,
then $\cJ_*(M)$ is the minimal control effort, as defined by \eqref{E-effort},
required to render \eqref{E-LQG} positive recurrent using
a linear stationary Markov control.
Theorem~\ref{T1.18} asserts that the infimum in \eqref{ED1.17B} is indeed
attained and that $\cJ_*(M)=\varLambda^+(M)$.
Moreover, linear stationary Markov controls are optimal for this task
within the class $\bUssm$.

%%%%%%%%%%%%%%%%%%%%%%%%%%%%%%%%%%%%%%%%%%%%%%%%%%%%%%%%%%%%%%%%%%%%%%%%%%%%%%%%
\begin{theorem}\label{T1.18}
Suppose that $M\in\RR^{d\times d}$ is exponentially dichotomous.
Then the following hold.
\begin{itemize}
\item[\textup{(}a\/\textup{)}]
There exists a unique positive semidefinite symmetric solution $Q$ of the
matrix Riccati equation
\begin{equation}\label{ET1.18A}
M\transp Q + Q M\;=\; Q^{2}\,,
\end{equation}
satisfying
\begin{equation}\label{ET1.18B}
(M-Q) \Sigma + \Sigma (M-Q)\transp\;=\; -I
\end{equation}
for some symmetric positive definite matrix $\Sigma$.
Moreover, $A=M-Q$ attains the infimum in \eqref{ED1.17B} subject
to \eqref{ED1.17A}, and it holds that
\begin{equation*}
\cJ_*(M)\;=\;\varLambda^+(M)\;=\;\frac{1}{2}\trace(Q)\,.
\end{equation*}
\item[\textup{(}b\/\textup{)}]
With $\mu_v$ denoting the invariant probability measure
of \eqref{E-LQG} under a control $v\in\bUssm$, we have
\begin{equation}\label{ET1.18C}
\inf_{v\in\bUssm}\;\int_{\Rd}\tfrac{1}{2}\abs{v(x)}^{2}\,\mu_{v}(\D{x})
\;=\; \varLambda^+(M)\,.
\end{equation}
Moreover, any control $v_*\in\bUssm$ which attains the infimum in
\eqref{ET1.18C} satisfies $v_*(x)=-Qx$  for almost all $x$ in $\Rd$.
\item[\textup{(}c\/\textup{)}]
Let $\Bar\beta\in\RR$.  The equation
\begin{equation}\label{ET1.18D}
\frac{1}{2}\,\Delta \Bar{V}(x) + \bigl\langle Mx,\grad \Bar{V}(x)\bigr\rangle
-\frac{\abs{\grad \Bar{V}(x)}^2}{2}\;=\; \Bar\beta
\end{equation}
has no solution if $\Bar\beta>\varLambda^+(M)$.
If $\Bar\beta=\varLambda^+(M)$, then $\Bar{V}(x)=\tfrac{1}{2} \langle x,Qx\rangle$
is the unique solution of \eqref{ET1.18D} satisfying $\Bar{V}(0)=0$.
If $\Bar\beta<\varLambda^+(M)$ and $\Bar{V}$ is a solution of
\eqref{ET1.18D}, then the diffusion in \eqref{E-LQG} under the control
$v=-\grad\Bar{V}$ is transient.
\end{itemize}
\end{theorem}

%%%%%%%%%%%%%%%%%%%%%%%%%%%%%%%%%%%%%%%%%%%%%%%%%%%%%%%%%%%%%%%%%%%%%%%%%%%%%%%%
\begin{remark}
Optimality and uniqueness of the optimal control $v(x)=-Qx$ in
Theorem~\ref{T1.18}\,(b) holds
over a larger class of Markov controls.
Indeed combining the results of \cite{Bogachev-12,Krylov-05},
we can replace `locally bounded' in the definition of $\bUssm$ by
$v\in \Lpl^p(\Rd)$ for some $p>d$.
Then the results of Theorem~\ref{T1.18}\,(b) hold for this class of controls.
\end{remark}

%%%%%%%%%%%%%%%%%%%%%%%%%%%%%%%%%%%%%%%%%%%%%%%%%%%%%%%%%%%%%%%%%%%%%%%%%%%%%%%%
\section{Gradient-like flows and energy functions}\label{S2}

%%%%%%%%%%%%%%%%%%%%%%%%%%%%%%%%%%%%%%%%%%%%%%%%%%%%%%%%%%%%%%%%%%%%%%%%%%%%%%%%
\subsection{Gradient-like Morse--Smale dynamical systems}

It is well known in the theory of dynamical systems
that if the set of non-wandering points of a flow on a compact manifold
consists of hyperbolic fixed points,
then the associated vector field
is generically \emph{gradient-like} (see Definition~\ref{D2.1} and
Theorem~\ref{T2.2} below).
This is also the case under Hypothesis~\ref{H1.1},
since the `point at infinity' is a source for the flow of $m$.

Recall that the \emph{index} of a hyperbolic critical point $z\in\Rd$
of a smooth vector field is defined as the dimension of the unstable
manifold
$\mathcal{W}_{\text{u}}(z)$.
This agrees with the number of eigenvalues of $Dm(z)$ which have
positive real parts.
The theorem below is well known \cite{Smale,Meyer}.
What we have added in its statement
is the assertion that the energy function can be chosen
in a manner that its Laplacian
at critical points of the vector field with positive index is negative.

We start with the following definition.

%%%%%%%%%%%%%%%%%%%%%%%%%%%%%%%%%%%%%%%%%%%%%%%%%%%%%%%%%%%%%%%%%%%%%%%%%%%%%%%%
\begin{definition}\label{D2.1}
We say that $\Lyap\in \Cc^{\infty}(\RR^{d})$ is an \emph{energy function}
if it is inf-compact, and has a finite set $\cS=\{z_{1},\dotsc,z_{n}\}$ of
critical points, which are all nondegenerate.
A $\Cc^{\infty}$ vector field $m$ on $\RR^{d}$ is called
\emph{gradient-like relative to} an energy function $\Lyap$ provided
that every point in $\cS$ is a
hyperbolic critical point of $m$, and
\begin{equation*}
\bigl\langle m(x),\nabla\Lyap(x)\bigr\rangle \;<\;
0\qquad \forall x\, \in\RR^{d}\setminus\cS\,.
\end{equation*}
If $m$ satisfies these properties,
we also say that $m$ is \emph{adapted to} $\Lyap$.
\end{definition}

%%%%%%%%%%%%%%%%%%%%%%%%%%%%%%%%%%%%%%%%%%%%%%%%%%%%%%%%%%%%%%%%%%%%%%%%%%%%%%%%
\begin{theorem}\label{T2.2}
Suppose that $m$ is a smooth vector field in $\RR^{d}$ for which
Hypothesis~\ref{H1.1} holds.
Let $G$ be any domain of $\Rd$ of the form
$\{x\in\Rd\,\colon \Bar\Lyap<c\}$ for some $c\in\RR$,
satisfying $G\supset\cK$, and let $\{a_z\,\colon z\in\cS\}$ be any set of
distinct real numbers
such that if $z$ and $z'$ are the $\alpha$- and $\omega$-limit points
of some trajectory, respectively, then $a_{z}>a_{z'}$.
Then there exists a function $\widehat\Lyap\in\Cc^\infty(\Bar G)$,
with the following properties.
\begin{itemize}
\item[\textup{(}i\/\textup{)}]
$\bigl\langle m(x),\nabla\widehat\Lyap(x)\bigr\rangle<0$ for all
$x\in \Bar G\setminus\cS$.
\smallskip
\item[\textup{(}ii\/\textup{)}]
For each $z\in\cS$, there exists a neighborhood $\cN_z$ of $z$ and
a symmetric matrix $Q_z\in\RR^{d\times d}$ such that
$\widehat\Lyap(x) = a_z+\langle x-z, Q_z(x-z)\rangle
+ \sorder(\abs{x-z}^{2})$ for all $x\in\cN_z$.
\smallskip
\item[\textup{(}iii\/\textup{)}]
$\Delta\widehat\Lyap(z) <0$, for all $z\in \cS\setminus\cSs$, where
$\cSs$, as defined earlier, denotes the stable equilibria
of the flow of $m$.
\smallskip
\item[\textup{(}iv\/\textup{)}]
There exists a constant $C_0>0$ such that 
\begin{equation}\label{ET2.2}
C_0\,\bigl(\dist(x,\cS)\vee\abs{\grad\widehat\Lyap(x)}\bigr)^2 \;\le\;
\babs{\bigl\langle m(x),\grad\widehat\Lyap(x)\bigr\rangle}
\;\le\; C_0^{-1}\,\bigl(\dist(x,\cS)\wedge\abs{\grad\widehat\Lyap(x)}\bigr)^2
\end{equation}
for all $x\in G$.
\end{itemize}
\end{theorem}

\begin{proof}
Since $m$ is smooth and bounded,
and $m(z)=0$ for $z\in\cS$, there exists a constant $\Tilde{C}_m>0$
such that
\begin{equation}\label{E-Cm2}
\abs{M_z x-m(x)}\;\le\; \Tilde{C}_m \abs{x}^{2}\qquad\forall\,x\in\Rd\,,
\quad\forall\,z\in\cS\,.
\end{equation}
Let $z\in\cS$ be a critical point of $m$ of index $q\ge 0$.
Translating the coordinates we may assume that $z=0$.
Since $m(0)=0$, then by \eqref{E-Cm2}, $m(x)$ takes the form
\begin{equation*}
m(x) \;=\;  M x + \order\bigl(\abs{x}^2\bigr)
\end{equation*}
locally around $x=0$, where $M=Dm(0)$.
By hypothesis $M$ has exactly $q$ ($d-q$) eigenvalues in the open right half
(left half) complex space.
Therefore since the corresponding eigenspaces are invariant under $M$,
there exists a linear coordinate transformation $T$ such that, in the
new coordinates $\Tilde{x} = T(x)$, the linear map $x\mapsto Mx$ has the
matrix representation
$\Tilde{M}= T M T^{-1}$ and $\Tilde{M}=\diag(\Tilde{M}_{1},-\Tilde{M}_{2})$, where
$\Tilde{M}_{1}$ and $\Tilde{M}_{2}$ are square Hurwitz matrices of dimension
$d-q$ and $q$ respectively.
By the Lyapunov theorem there exist positive definite matrices
$\Tilde{Q}_{i}$, $i=1,2$, satisfying
\begin{equation}\label{ET2.2a}
\begin{split}
\Tilde{M}_{1}\transp\Tilde{Q}_{1}+\Tilde{Q}_{1}\Tilde{M}_{1} &\;=\; -I_{d-q}\,,
\\[3pt]
\Tilde{M}_{2}\transp\Tilde{Q}_{2}+\Tilde{Q}_{2}\Tilde{M}_{2} &\;=\; -I_{q}\,,
\end{split}
\end{equation}
where $I_{d-q}$ and $I_{q}$ are the identity matrices of dimension $d-q$
and $q$, respectively.
Suppose $q>0$, and
let $\theta>1$ be such that
\begin{equation}\label{ET2.2b}
\theta\,\trace\bigl(T\transp \diag(0,\Tilde{Q}_{2}) T\bigr)
\;>\; \trace\bigl(T\transp \diag(\Tilde{Q}_{1},0) T\bigr)\,,
\end{equation}
and define $\widehat\Lyap$ in some neighborhood of $0$ by
\begin{equation}\label{ET2.2c}
\widehat\Lyap(x) \;\df\;
a+\bigl\langle x, T\transp \diag(\Tilde{Q}_{1},-\theta \Tilde{Q}_{2}) Tx\bigr\rangle\,,
\end{equation}
where $a$ is a constant to be determined later.
By \eqref{ET2.2b} we obtain
$\Delta\widehat\Lyap(0) <0$, and thus (iii) holds.

Using \eqref{E-Cm2}, we have
\begin{equation*}
\bigl\langle m(x),\grad\widehat\Lyap(x)\bigr\rangle \;=\;
x\transp\bigl[ M\transp T\transp \diag(\Tilde{Q}_{1},-\theta \Tilde{Q}_{2}) T
+ T\transp \diag(\Tilde{Q}_{1},-\theta \Tilde{Q}_{2}) TM\bigr]x
+ \order\bigl(\abs{x}^3\bigr)\,.
\end{equation*}
Expanding we obtain
\begin{align*}
T\transp \diag(\Tilde{Q}_{1},-\theta \Tilde{Q}_{2}) TM &\;=\;
T\transp \diag(\Tilde{Q}_{1},-\theta \Tilde{Q}_{2}) T T^{-1} \Tilde{M} T\\
&\;=\;T\transp\diag(\Tilde{Q}_{1}\Tilde{M}_{1},\theta\Tilde{Q}_{2}\Tilde{M}_{2}) T\,.
\end{align*}
By \eqref{ET2.2a} we obtain
\begin{equation*}
\bigl\langle m(x),\grad\widehat\Lyap(x)\bigr\rangle \;=\;
- \bigl\langle x, T\transp\diag(I_{d-q},\theta I_{q}) T x\bigr\rangle
 + \order\bigl(\abs{x}^3\bigr)\,.
\end{equation*}
Therefore, since $\theta>1$, we have
\begin{equation}\label{ET2.2d}
- \abs{Tx}^{2} + \order\bigl(\abs{x}^3\bigr) \;\le\;
\bigl\langle m(x),\grad\widehat\Lyap(x)\bigr\rangle
\;\le\; - \theta\,\abs{Tx}^{2} + \order\bigl(\abs{x}^3\bigr)\,.
\end{equation}

As shown in \cite{Smale} one can select any real numbers $a_{i}$
and define $\widehat\Lyap$ on $\cS$ by setting $\widehat\Lyap(z_{i})=a_{i}$ as long
as the following consistency condition
is met.  If $z_{i}$ and $z_{j}$ are the $\alpha$- and $\omega$-limit points
of some trajectory then $a_{i}>a_{j}$.
Thus $\widehat\Lyap$ can be defined in non-overlapping neighborhoods of the
critical points
by \eqref{ET2.2c} so as to satisfy \eqref{ET2.2d} and parts (i)--(iii)
of the theorem.
Since $G$ is positively invariant under the flow of $m$,
the stable and unstable manifolds of $\cS$ intersect transversally
by Hypothesis~\ref{H1.1}\,(2),
and $m$ is transversal to the boundary of $\partial G$ by
Hypothesis~\ref{H1.1}\,(3b),
this function can then be extended to $\Bar{G}$ by the handlebody
decomposition technique introduced by Smale.
For details see \cite[Theorem~B]{Smale} and
\cite[Theorem~1]{Meyer}.

It is clear by \eqref{ET2.2c}--\eqref{ET2.2d} that \eqref{ET2.2} holds in some open
neighborhood of each $z\in\cS$, and thus, $\cS$ being a finite set,
it also holds in some neighborhood of $\cN$ of $\cS$.
Since $\langle m,\grad\widehat\Lyap\rangle$ is strictly negative on the
compact set $\Bar{G}\setminus\cN$
and $\bigl\langle m(x),\grad\widehat\Lyap(x)\bigr\rangle<0$
for all $x\notin\cS$, a constant $C_0$ can be selected so
that \eqref{ET2.2} holds on $G$.
This completes the proof.
\end{proof}

The function $\widehat\Lyap$ in Theorem~\ref{T2.2} can be extended
to $\Rd$, and constructed in a manner so that it agrees,
outside some ball, with
the Lyapunov function $\Bar\Lyap$ in Hypothesis~\ref{H1.1}.
This is stated in the following lemma.

%%%%%%%%%%%%%%%%%%%%%%%%%%%%%%%%%%%%%%%%%%%%%%%%%%%%%%%%%%%%%%%%%%%%%%%%%%%%%%%%
\begin{lemma}\label{L2.3}
Under the assumptions of Theorem~\ref{T2.2}
the vector field $m$ is adapted to an energy function $\Lyap$ which satisfies
\nomenclature[Ea]{$\Lyap$}{energy function, Lemma~\ref{L2.3}}%
$\Lyap = \Bar\Lyap$ on the complement of some open ball
which contains $\cS$.
Also parts {\upshape(}i\/{\upshape)}---{\upshape(}iv\/{\upshape)}
of Theorem~\ref{T2.2} hold, and for every bounded domain $G$
there exists a constant $C_0=C_0(G)$ such that
\eqref{ET2.2} holds for all $x\in G$.
Moreover there exists a  constant $\overline{C}_0>0$ such that
with
\begin{align*}
\overline{\mathscr{V}}(x)&\;\df\;
\max\,\bigl\{\bigl(\dist(x,\cS)\bigr)^{2}\wedge\dist(x,\cS),\,
\abs{\grad\Lyap(x)}^{2}\wedge \abs{\grad\Lyap(x)}\bigr\}\,,\\[5pt]
\underline{\mathscr{V}}(x)&\;\df\;
\min\,\bigl\{\bigl(\dist(x,\cS)\bigr)^{2}\wedge\dist(x,\cS),\,
\abs{\grad\Lyap(x)}^{2}\wedge \abs{\grad\Lyap(x)}\bigr\}\,,
\end{align*}
we have
\begin{equation}\label{EL2.3}
\bigl(\overline{C}_0)^{-1}\,\overline{\mathscr{V}}(x)\;\le\;
\babs{\bigl\langle m(x),\grad\Lyap(x)\bigr\rangle}
\;\le\; \overline{C}_0\,\underline{\mathscr{V}}(x)
\qquad\forall\, x\in \Rd\,.
\end{equation}
\end{lemma}

\begin{proof}
Select $c\in\RR$ such that $G_1\df\{x\in\Rd\,\colon \Bar\Lyap<c\}$
contains $\cK$.
Let $G_2\df\{x\in\Rd\,\colon \Bar\Lyap<2c\}$.
By Theorem~\ref{T2.2} there exists $\widehat\Lyap\in\Cc^\infty(G_2)$
with the properties stated.
Without loss of generality we can assume that $\widehat\Lyap=2c$ on
$\partial G_2$ \cite[Theorem~B]{Smale}.
Let $c_1\df \sup_{G_1} \widehat\Lyap$.
Then $c_1<2c$ by the positive invariance of $G_2$, and the property
$\langle m,\grad\widehat\Lyap\rangle<0$ in $G_2\setminus G_1$.
We write $A\Subset B$ to indicate that $\Bar{A}\subset B$.
Let $\Tilde{G}\df\{x\in\Rd\,\colon \widehat\Lyap<\nicefrac{(c_1+2c)}{2}\}$,
and $c_2\df \sup_{\Tilde{G}} \Bar\Lyap$.
Then $G_1\Subset\Tilde{G}\Subset G_2$, and $c<c_2<2c$ by construction.

Let $\psi\colon\RR\to\RR$ be a smooth non-decreasing function such
that $\psi(t) = t$ for $t\le \frac{1}{2}(c_1+2c)$, 
$\psi(t) = 2c$ for $t\ge 2c$, and whose derivative
is strictly positive on the interval $\bigl[\frac{1}{2}(c_1+2c),2c\bigr]$.
Similarly, let $\Bar\psi\colon\RR\to\RR$
be a smooth non-decreasing function such
that $\Bar\psi(t) = 0$ for $t\le -c$
and $\Bar\psi(t) = t$ for $t\ge c_2-2c$.
Define $\Lyap\df \psi\circ\widehat\Lyap+\Bar\psi\circ(\Bar\Lyap-2c)$.
By construction $\Lyap$ agrees with $\widehat\Lyap$ on $G_1$
and with $\Bar\Lyap$ on $G_2^c$.
It can also be easily verified that
$\sup_{G_2\setminus G_1}\;\langle m,\grad\Lyap\rangle <0$.
Thus $\Lyap\in\Cc^\infty(\Rd)$ is an energy function,
and $m$ is adapted to $\Lyap$ according
to Definition~\ref{D2.1}.

Since $\bigl\langle m(x),\grad\Lyap(x)\bigr\rangle<0$
for all $x\notin\cS$, and $\Lyap$ agrees
with $\widehat\Lyap$ on $\cK$, Theorem~\ref{T2.2}\,(i)--(iv) clearly hold.
Also since \eqref{EL2.3} holds in some neighborhood of $\cS$ by
\eqref{ET2.2c}--\eqref{ET2.2d}, then, in view
of the linear growth of $\bigl\langle m(x),\grad\Bar\Lyap(x)\bigr\rangle\ne0$
in \eqref{Lyapunov}, and the assumptions on the growth of $\Bar\Lyap$ in
Hypothesis~\ref{H1.1}, \eqref{EL2.3} also holds
on $\Rd$.
\end{proof}

%%%%%%%%%%%%%%%%%%%%%%%%%%%%%%%%%%%%%%%%%%%%%%%%%%%%%%%%%%%%%%%%%%%%%%%%%%%%%%%%
\section{Minimal stochastically stable sets}\label{S3}

Recall that $\beta^\varepsilon_*$ denotes the optimal value of
\eqref{cost}, $\eta_{*}^{\varepsilon}$ denotes the stationary
distribution of the process $X$ under the optimal stationary Markov control
$v^\varepsilon_*$,
and $\varrho_{*}^{\varepsilon}$ denotes its density.
These definitions are fixed throughout the rest of the paper.
Also recall the definition of the extended generator
in \eqref{E-Lgv}, and the definition of $\sR$ in \eqref{E-cRnew}.
For a stationary Markov control $v$, we use the notation
\begin{equation}\label{E-sR}
\sR[v](x)\;\df\; \sR\bigl(x,v(x)\bigr)
\;=\;\ell(x) + \frac{1}{2}\, \abs{v(x)}^{2}\,.
\end{equation}%
\nomenclature[Ac]{$\sR[v](x)$}{running cost under control $v$, equation \eqref{E-sR}}%

Throughout the rest of the paper $\Lyap$ is
a smooth function that satisfies (i)--(iv) in Theorem~\ref{T2.2}
and agrees with $\Bar\Lyap$ in Hypothesis~\ref{H1.1} on the complement
of some open ball which contains $\cS$ (Lemma~\ref{L2.3}).
We refer to $\Lyap$ as the \emph{energy function}.

We start the analysis with the following lemma which
asserts that $\eta_{*}^\varepsilon$ concentrates on $\cS$ as
$\varepsilon\searrow 0$.

%%%%%%%%%%%%%%%%%%%%%%%%%%%%%%%%%%%%%%%%%%%%%%%%%%%%%%%%%%%%%%%%%%%%%%%%%%%%%%%%
\begin{lemma}\label{L3.1}
The family $\{\eta^\varepsilon_*,\, \varepsilon\in(0, 1)\}$ is tight,
and any sub-sequential limit as $\varepsilon\searrow0$ has support on $\cS$.
\end{lemma}

\begin{proof}
Recall that $\eta_{0}^{\varepsilon}$
denotes the invariant probability measure of \eqref{E-sde} under
the control $U= 0$.
Define
\begin{equation*}
\beta^{\varepsilon}_{0}\;\df\; \int_{\Rd}\ell(x)\, \eta_{0}^{\varepsilon}(\D{x})\,.
\end{equation*}
By \eqref{EE1.8} we have
\begin{equation}\label{EL3.1A}
\int_{\Rd}\ell(x)\, \eta^\varepsilon_*(\D{x})
\;\le\; \beta^\varepsilon_* \;\le\; \beta^{\varepsilon}_{0}
\;\le\; \Bar{c}_\ell\qquad\forall\,\varepsilon\in(0,1)\,.
\end{equation}
Since $\ell$ is inf-compact, \eqref{EL3.1A} implies that
$\{\eta^\varepsilon_*,\, \varepsilon\in(0, 1)\}$ is tight.
Let $\phi_t(x)$ denote the solution of \eqref{ode} starting at $x\in\Rd$
at $t=0$, i.e., $\phi_0(x)=x$.
If $C_{m}$ denotes a Lipschitz constant of $m$ and $X_{0}=x$, we
have
\begin{equation}\label{EL3.1B}
\abs{X_{t}-\phi_t(x)}\;\le\; C_{m} \int_{0}^{t}\abs{X_{s}-\phi_s(x)}\,\D{s}
+ \varepsilon \int_{0}^{t}\abs{v^\varepsilon_*(X_{s})}\,\D{s}
+ \varepsilon^\nu \abs{W_{t}}\,.
\end{equation}
Hence applying Gronwall's inequality we obtain from \eqref{EL3.1B} that
\begin{equation}\label{EL3.1C}
\sup_{s\in[0, t]}\, \abs{X_{s}-\phi_s(x)}\;\le\;
\E^{C_{m}t}\biggl(\varepsilon\int_{0}^{t}
\abs{v^\varepsilon_*(X_{s})}\,\D{s}
+\varepsilon^\nu \sup_{s\le t}\;\abs{W_{s}}\biggr)\,.
\end{equation}
In turn, for any $\delta>0$, \eqref{EL3.1C} implies that
\begin{equation*}
\Prob_{x}\Bigl(\abs{X_{t}-\phi_t(x)} \ge \delta\Bigr) 
\;\le\; \Prob_{x}\biggl(\int_{0}^{t}\abs{v^\varepsilon_*(X_{s})}\,\D{s}
\;\ge\; \frac{\delta \E^{-C_{m}t}}{2\varepsilon}\biggr)
+ \Prob_{x}\biggl(\sup_{s\le t}\; \abs{W_{s}}\;\ge\;
\frac{\delta \E^{-C_{m}t}}{2\varepsilon^\nu}\biggr)
\end{equation*}
for $t>0$\,.
By Jensen's inequality we obtain
\begin{align*}
\Prob_{x}\biggl(\int_{0}^{t} \abs{v^\varepsilon_*(X_{s})}\,\D{s}
\;\ge\; \frac{\delta \E^{-C_{m}t}}{2\varepsilon}\biggr)&\;\le\;
\Prob_{x}\biggl(\int_{0}^{t} \abs{v^\varepsilon_*(X_{s})}^{2} \,\D{s}
\;\ge\; \frac{\delta^{2} \E^{-2C_{m}t}}{4t\varepsilon^{2}}\biggr)\\[5pt]
&\;\le\; \frac{4t\varepsilon^{2}}{\delta^{2}}\,\E^{2C_{m}t}\,
\Exp_{x}\biggl[\int_{0}^{t} \abs{v^\varepsilon_*(X_{s})}^{2} \,
\D{s}\biggr]\,.
\end{align*}
Therefore for any compact set $K\subset\Rd$ we have
\begin{multline}\label{EL3.1D}
\int_{K}\Prob_{x}\bigl(\abs{X_{t}-\phi_t(x)}\ge \delta\bigr)\,
\eta^\varepsilon_*(\D{x})
\;\le\; \frac{4t^{2}\varepsilon^{2}}{\delta^{2}}\,\E^{2C_{m}t}\,
\int_{\Rd} \abs{v^\varepsilon_*(x)}^{2}\, \eta^\varepsilon_*(\D{x})\\
+ \sup_{x\in K}\;
\Prob_{x}\biggl(\sup_{s\le t}\; \abs{W_{s}}\;\ge\; \frac{\delta}{2\varepsilon^\nu}
\E^{-C_{m}t}\biggr)\,.
\end{multline}
It is clear that the right hand side of \eqref{EL3.1D} tends to $0$ as
$\varepsilon\searrow 0$.
Thus for any compact set $K\subset\Rd$,
and any Lipschitz function $f\in\Cc_{b}(\Rd)$ it holds that
\begin{equation}\label{EL3.1E}
\int_{K} \babs{\Exp^{v^\varepsilon_*}_{x}[f(X_{t})]-f\bigl(\phi_t(x)\bigr)}\,
\eta^\varepsilon_*(\D{x})
\;\xrightarrow[\varepsilon\searrow0]{}\;0\,.
\end{equation}
On the other hand, since $\eta^\varepsilon_*$ is an invariant probability
measure, we have
\begin{equation}\label{EL3.1F}
\int_{\Rd}\Exp^{v^\varepsilon_*}_{x}[f(X_{t})]\,\eta^\varepsilon_*(\D{x})
\;=\;\int_{\Rd}f(x)\,\eta^\varepsilon_*(\D{x})
\qquad\forall\,f\in\Cc_{b}(\Rd)\,,\ \forall\, t\ge0\,.
\end{equation}
Let $\Bar\eta\in\cP(\Rd)$ be any limit of
$\eta^\varepsilon_*$ along some sequence $\{\varepsilon_n\}$, with
$\varepsilon_n\searrow 0$ as $n\to\infty$.
By \eqref{EL3.1E}--\eqref{EL3.1F}, the tightness of
$\{\eta^\varepsilon_*,\, \varepsilon\in(0, 1)\}$, and a standard
triangle inequality, we obtain
\begin{equation}\label{EL3.1G}
\int_{\Rd}f\bigl(\phi_t(x)\bigr)\,\Bar\eta(\D{x})\;=\;\int_{\Rd}f(x)\,
\Bar\eta(\D{x})\qquad\forall\, t\ge0\,,
\end{equation}
for all Lipschitz functions $f\in\Cc_{b}(\Rd)$.
Since the $\omega$-limit set of any trajectory of \eqref{ode}
is contained in  $\cS$,
\eqref{EL3.1G} shows that $\Bar\eta$ has support on $\cS$.
This completes the proof.
\end{proof}

%%%%%%%%%%%%%%%%%%%%%%%%%%%%%%%%%%%%%%%%%%%%%%%%%%%%%%%%%%%%%%%%%%%%%%%%%%%%%%%%
\subsection{Two Lemmas concerning the case \texorpdfstring{$\nu \ge 1$}{n>>1}}
%\label{S3.1}

For $z\in\cS$, let $\Bar{v}_z^{\varepsilon}$, $\varepsilon\in(0,1)$, denote
the stationary Markov control defined by
\begin{equation}\label{E-barv}
\Bar{v}_z^\varepsilon(x)
\;\df\; \frac{(M_z-\widehat{Q}_z)(x-z)-m(x)}{\varepsilon}\,,
\quad t \ge 0\,,
\end{equation}
where $M_z$ and $\widehat{Q}_z$ are as in Definition~\ref{D1.9}.
The controlled process, is then governed by the diffusion
\begin{equation}\label{E-bardiff}
\D{X}_{t}\;=\; (M_z-\widehat{Q}_z) (X_{t}-z)\,\D{t} + \varepsilon^\nu\,\D{W}_{t}\,.
\end{equation}
Since $M_z-\widehat{Q}_z$ is Hurwitz by Theorem~\ref{T1.18},
the diffusion has a stationary probability distribution $\Bar\mu^\varepsilon_z$,
which is Gaussian with mean $z$ and covariance matrix
$\varepsilon^{2\nu}\,\widehat\Sigma_z$, where $\widehat\Sigma_z$ is as
in \eqref{ED1.9}.

We start with the following lemma.

%%%%%%%%%%%%%%%%%%%%%%%%%%%%%%%%%%%%%%%%%%%%%%%%%%%%%%%%%%%%%%%%%%%%%%%%%%%%%%%%
\begin{lemma}\label{L3.2}
Suppose that $\nu\ge 1$ and $z\in\cS$.
Let $\Bar{v}_z^\varepsilon$ be the stationary Markov control
in \eqref{E-barv}, and $\Bar\mu^\varepsilon_z$ the invariant probability measure
of the diffusion governed by \eqref{E-bardiff}.
Then
\begin{equation}\label{EL3.2A}
\begin{split}
\int_{\Rd} \tfrac{1}{2} \abs{\Bar{v}^\varepsilon(x)}^{2}\,\Bar\mu^\varepsilon_z(\D{x})
&\;=\;\varepsilon^{2\nu - 2} \varLambda^+ \bigl(Dm(z)\bigr)
+\order\bigl(\varepsilon^{4\nu-2}\bigr)\,,
\\[5pt]
\int_{\Rd}\ell(x)\,\Bar\mu^\varepsilon_z(\D{x})
&\;=\;\ell(z) + \order\bigl(\varepsilon^{2\nu}\bigr)\,.
\end{split}
\end{equation}
\end{lemma}

\begin{proof}
Without loss of generality assume that $z=0$, and simplifying the
notation we let $M=M_z$, $Q=\widehat{Q}_z$, $\Sigma=\widehat\Sigma_z$,
and $\Bar\mu^{\varepsilon}=\Bar\mu^{\varepsilon}_z$.

We have
\begin{equation}\label{EL3.2B}
\abs{(M-Q)x-m(x)}^{2}\;=\;\abs{Qx}^{2} + 2\bigl\langle Qx,Mx-m(x)\bigr\rangle
+ \abs{Mx-m(x)}^{2}\,.
\end{equation}
Since by Taylor's theorem it holds that
\begin{equation*}
\bigl\langle Qx,Mx-m(x)\bigr\rangle \;=\;
\bigl\langle Qx, F(x)\bigr\rangle + \order\bigl(\abs{x}^4\bigr)\,,
\end{equation*}
with
\begin{equation*}
F(x)\;\df\;\bigl(F_1(x),\dotsc, F_d(x)\bigr) \quad \text{and}
\quad F_i(x) \;\df\;\tfrac{1}{2}\bigl\langle x, \grad^2 m_i(0) x\bigr\rangle\,,
\end{equation*}
by \eqref{E-Cm2} and \eqref{EL3.2B} we obtain
\begin{equation}\label{EL3.2C}
\abs{(M-Q)x-m(x)}^{2}\;=\;\abs{Qx}^{2} + 2\bigl\langle Qx, F(x)\bigr\rangle
+ \order\bigl(\abs{x}^4\bigr)\,.
\end{equation}
As mentioned in the paragraph preceding the lemma, $\Bar\mu^{\varepsilon}$
is Gaussian, with zero mean, and
covariance matrix $\varepsilon^{2\nu}\,\Sigma$, where
$\Sigma$ is the solution of \eqref{ET1.18B}.
Since $\bigl\langle Qx, F(x)\bigr\rangle$ is a homogeneous polynomial
of degree $3$ it has zero mean under the Gaussian.
Also the fourth moments of $\Bar\mu^{\varepsilon}$
are of order $\varepsilon^{4\nu}$.
It then follows by the estimate in \eqref{EL3.2C} and Theorem~\ref{T1.18}\,(b)
that
\begin{align}\label{EL3.2D}
\frac{1}{2}\,\int_{\Rd}\abs{\Bar{v}^\varepsilon(x)}^{2}\,\Bar\mu^{\varepsilon}(\D{x})
&\;=\;
\int_{\Rd} \tfrac{1}{2\varepsilon^2} \abs{Qx}^2\,\Bar\mu^{\varepsilon}(\D{x})
+ \order\bigl(\varepsilon^{4\nu - 2}\bigr)\\[5pt]
&\;=\; 
\varepsilon^{2\nu-2}\,\varLambda^+(M)+\order\bigl(\varepsilon^{4\nu - 2}\bigr)\,.
\nonumber
\end{align}

To prove the second equation
in \eqref{EL3.2A}, we use the bound
\begin{equation}\label{E-Cl2}
\babs{\ell(x)-\ell(z)- D\ell(z)(x-z)}\;\le\; \Tilde{C}_\ell\,\abs{x-z}^2
\qquad\forall\, x\in\Rd\,,\quad\forall\,z\in\cS\,,
\end{equation}
for some constant $\Tilde{C}_\ell$,
and since $\Bar\mu^{\varepsilon}$ has zero mean we obtain
\begin{equation}\label{EL3.2E}
\babss{\int_{\Rd} \bigl(\ell(x)-\ell(0)\bigr)\,\Bar\mu^{\varepsilon}(\D{x})}
\;\le\; \varepsilon^{2\nu}\Tilde{C}_\ell\,\trace(\Sigma)\,.
\end{equation}
By combining \eqref{EL3.2D} and \eqref{EL3.2E}
we obtain \eqref{EL3.2A}.
The proof is complete.
\end{proof}

Recall the notation in Definition~\ref{D1.10}.
Lemma~\ref{L3.2} in conjunction with Lemma~\ref{L3.1} leads to
the following.

%%%%%%%%%%%%%%%%%%%%%%%%%%%%%%%%%%%%%%%%%%%%%%%%%%%%%%%%%%%%%%%%%%%%%%%%%%%%%%%%
\begin{lemma}\label{L3.3}
It holds that
\begin{equation}\label{EL3.3A}
\begin{aligned}
\beta^\varepsilon_* &\;\le\;
\fJ+\varepsilon^{2\nu - 2}\,\min_{z\in\cZ}\,\varLambda^+\bigl(Dm(z)\bigr)
+\order\bigl(\varepsilon^{2\nu}\bigr)\qquad\text{if\ \ } \nu>1\,,\\[5pt]
\beta^\varepsilon_* &\;\le\; \fJc+\order\bigl(\varepsilon^{2}\bigr)
\qquad\text{if\ \ } \nu=1\,.
\end{aligned}
\end{equation}
Moreover, if $\nu>1$, then
\begin{equation}\label{EL3.3B}
\lim_{\varepsilon\searrow0}\;\beta^\varepsilon_*\;=\; \fJ\,,
\end{equation}
and
$\sS\,\subset\,\cZ$.
\end{lemma}

\begin{proof}
Recall the function $\sR[v]$ defined in \eqref{E-sR}.
By Lemma~\ref{L3.2} we have
\begin{align}\label{EL3.3C}
\beta^\varepsilon_* &\;\le\; \int_{\Rd}\sR[\Bar{v}_z^\varepsilon](x)\,
\Bar\mu^{\varepsilon}_z(\D{x})\\[5pt]
&\;\le\;
\ell(z) + \varepsilon^{2\nu - 2} \varLambda^+ \bigl(Dm(z)\bigr)
+ \order\bigl(\varepsilon^{2\nu}\bigr) 
\qquad\forall\,z\in\cS\,,\ \nu\ge1\,.\nonumber
\end{align}
Since $\ell(z)=\fJ$ for all $z\in\widetilde\cZ\subset\cZ$,
the first inequality in \eqref{EL3.3A} follows by evaluating
\eqref{EL3.3C} at a point $z\in\widetilde\cZ$,
while the second inequality in \eqref{EL3.3A} follows by evaluating
\eqref{EL3.3C} at a point $z\in\cZc$.

Since
\begin{equation}\label{EL3.3D}
\lim_{\varepsilon\searrow0}\;\beta^\varepsilon_*\;\ge\;\fJ
\end{equation}
for all $\nu>0$ by Lemma~\ref{L3.1},
\eqref{EL3.3B} follows by \eqref{EL3.3A}
and \eqref{EL3.3D} when $\nu>1$, and clearly then, in this case we have
$\sS\,\subset\,\cZ$.
\end{proof}

%%%%%%%%%%%%%%%%%%%%%%%%%%%%%%%%%%%%%%%%%%%%%%%%%%%%%%%%%%%%%%%%%%%%%%%%%%%%%%%%
\begin{remark}\label{R3.4}
It is worth mentioning here that if $z\in\cSs$,
then a control that renders $\{z\}$ stochastically stable
can be synthesized from the energy function $\Lyap$.
Note that by Theorem~\ref{T2.2}\,(ii), $\Lyap$ can be selected
so that $\Lyap(z)=0$ and $\Lyap(z')>0$ for all $z'\in\cS\setminus\{z\}$.
Consider the control
\begin{equation*}
\Breve{v}^{\varepsilon}(x) \;\df\; 
- \frac{1}{\varepsilon}\bigl(m(x) +\grad\Lyap(x)\bigr)\,,
\quad t \ge 0\,.
\end{equation*}
Then $X$ is given by
\begin{equation*}
\D X_{t} \;=\; -\grad\Lyap(X_{t})\,\D{t}
+ \varepsilon^{\nu}\,\D{W}_{t}\,, \quad t \ge 0\,.
\end{equation*}
Let $\Breve\mu^{\varepsilon}$ denote its unique invariant probability
measure.
Recall the definition in \eqref{E-Lgv}.
Since
\begin{equation*}
\Lg^{\varepsilon}_{\Breve{v}^{\varepsilon}}\Lyap \;\le\;
\frac{\varepsilon^{2\nu}}{2} \norm{\Delta\Lyap}_{\infty}
- \abs{\grad\Lyap}^{2}\,,
\end{equation*}
it follows that
\begin{equation*}
2\int_{\Rd}\abs{\grad\Lyap}^{2}\,\D\Breve\mu^{\varepsilon}\;\le\;
\varepsilon^{2\nu} \norm{\Delta\Lyap}_{\infty}\,.
\end{equation*}
Note that $\Breve\mu^{\varepsilon}$ has density
$\varrho^{\varepsilon}(x)
= C(\varepsilon)\,\E^{-\frac{2\Lyap(x)}{\varepsilon^{2\nu}}}$,
where $C(\varepsilon)$ is a normalizing constant.
Therefore we have
\begin{align*}
\int_{\Rd}\abs{\Breve{v}^{\varepsilon}(x)}^{2}\,\Breve\mu^{\varepsilon}(\D{x})
&\;\le\; 2\int_{\Rd} \bigl(\abs{m(x)}^{2}
+ \abs{\grad\Lyap(x)}^{2}\bigr)\varepsilon^{-2}\,\Breve\mu^{\varepsilon}(\D{x})
\nonumber\\[5pt]
&\;\le\; 2\int_{\Rd}\varepsilon^{-2} \abs{m(x)}^{2} \,\Breve\mu^{\varepsilon}(\D{x})
+ \varepsilon^{2\nu - 2} \norm{\Delta\Lyap}_{\infty}
\nonumber\\[5pt]
&\;\le\; \order\bigl(\varepsilon^{2\nu - 2}\bigr)
+ \varepsilon^{2\nu - 2} \norm{\Delta\Lyap}_{\infty}\,. 
\end{align*}
For the last inequality we used the fact that $m$ is bounded, $m(z)=0$,
and that $\Lyap$ is locally quadratic around $z$.
\end{remark}

%%%%%%%%%%%%%%%%%%%%%%%%%%%%%%%%%%%%%%%%%%%%%%%%%%%%%%%%%%%%%%%%%%%%%%%%%%%%%%%%
\subsection{Results concerning stable equilibria}

Recall that $\cSs$ is the collection of stable equilibrium points,
and $\fJs=\min_{z\in\cSs}\,\bigl\{\ell(z)\bigr\}$.
The following lemma holds for any $\nu>0$.
It shows that if $z\in\cSs$ then there exists a Markov stationary
control $v^{\varepsilon}$ with invariant measure $\mu^{\varepsilon}$ satisfying
$\int_{\Rd}\abs{v^{\varepsilon}(x)}^2\mu^{\varepsilon}(\D{x})
\in\order\bigl(\varepsilon^{n}\bigr)$ for any $n\in\NN$,
under which $\{z\}$ is stochastically stable.

%%%%%%%%%%%%%%%%%%%%%%%%%%%%%%%%%%%%%%%%%%%%%%%%%%%%%%%%%%%%%%%%%%%%%%%%%%%%%%%%
\begin{lemma}\label{L3.5}
The following hold.
\begin{itemize}
\item[\textup{(}i\/\textup{)}]
For any $\nu>0$ and $z\in\cSs$ there exists a Markov control $\Check{v}^{\varepsilon}$,
and constants $\varepsilon_0=\varepsilon_0(\nu)>0$,
and $c_0>0$ independent of $\nu$,
with the following properties.
With $\Check\mu^{\varepsilon}$ denoting the invariant probability measure
of \eqref{E-sde} under the control $\Check{v}^{\varepsilon}$, it holds that
\begin{equation}\label{EL3.5A}
\begin{split}
\int_{\abs{x-z}\,\ge\,\varepsilon^{\nicefrac{\nu}{2}}}
\abs{x-z}^{2}\,\Check\mu^{\varepsilon}(\D{x}) &\;\le\;
\frac{\varepsilon^{2\nu}}{c_0(1-\varepsilon^\nu)}\,
\E^{-c_0\varepsilon^{-\nu}}\,,\\[5pt]
\int_{\Rd} \abs{\Check{v}^{\varepsilon}(x)}^{2}\, \Check\mu^{\varepsilon}(\D{x})
&\;\le\; \frac{\varepsilon^{2(\nu-1)}}{c_0(1-\varepsilon^\nu)}\,
\E^{-c_0\varepsilon^{-\nu}}
\end{split}
\end{equation}
for all $\varepsilon< \varepsilon_0$,
and
\begin{equation}\label{EL3.5B}
\varepsilon^{-\nu}\;\babss{\int_{\Rd} \ell(x)\,\Check\mu^{\varepsilon}(\D{x})-\ell(z)}
\;\xrightarrow[\varepsilon\searrow0]{}\;0\,.
\end{equation}
In particular, we have
\begin{equation*}
\limsup_{\varepsilon\searrow0}\;
\frac{1}{\varepsilon^{n}}\;\int_{\Rd} \abs{\Check{v}^{\varepsilon}(x)}^{2}\,
\Check\mu^{\varepsilon}(\D{x})\;=\;0\qquad\forall\,n\in\NN\,.
\end{equation*}
\item[\textup{(}ii\/\textup{)}]
It holds that
$\beta^\varepsilon_*\le\fJs+\sorder\bigl(\varepsilon^{\nu}\bigr)$
for $\nu\in(0, \nicefrac{2}{3})$, 
and $\beta^\varepsilon_*\le\fJs+\order\bigl(\varepsilon^{4\nu-2}\bigr)$ for
$\nu\in[\nicefrac{2}{3}, 1)$.
\end{itemize}
\end{lemma}

\begin{proof}
In order to simplify
the notation, we translate the origin so that $z=0$,
and we let $M\df Dm(0)$.
Let $R^{-1}$ be the symmetric positive definite solution to
the Lyapunov equation
$M R^{-1} + R^{-1}M\transp = -4I$.
Thus $M\transp R + R M = -4 R^{2}$.
Since scaling $R$ by multiplying it with a positive constant smaller than $1$
preserves the inequality
\begin{equation}\label{EL3.5C}
M\transp R + R M \le -4 R^2\,,
\end{equation}
we may assume that
$\trace(R)\le1$ and \eqref{EL3.5C} holds.
The sole purpose of this scaling is to simplify the calculations in
the proof.
We define the control $\Check{v}^{\varepsilon}$ by
\begin{equation*}
\Check{v}^{\varepsilon}(x) \;\df\;
\begin{cases}\varepsilon^{-1}
\bigl(Mx-m(x)\bigr)&\text{if\ \ }\abs{Rx}\ge\varepsilon^{\nicefrac{\nu}{2}}\,,
\\[5pt]
0&\text{otherwise.}
\end{cases}
\end{equation*}

We apply the function
$F(x)\df\varepsilon^{2\nu}\,
\exp\bigl(\varepsilon^{-2\nu}\,\langle x, R x\rangle\bigr)$ to
$\Lg^{\varepsilon}_{\Check{v}^{\varepsilon}}$, which is
defined in \eqref{E-Lgv}.
By \eqref{EL3.5C}, and since $\trace(R)\le1$, we obtain
\begin{align}\label{EL3.5D}
\Lg^{\varepsilon}_{\Check{v}^{\varepsilon}}\,F(x)
&\;=\; \bigl(\varepsilon^{2\nu}\trace(R) + 2\abs{Rx}^{2}
+ \bigl\langle x,(M\transp R + R M)x\bigr\rangle\bigr)\,
\E^{\frac{\langle x, R x\rangle}{\varepsilon^{2\nu}}} \\[5pt]
&\;\le\;
\bigl(\varepsilon^{2\nu} - 2\abs{Rx}^{2}\bigr)\,
\E^{\frac{\langle x, R x\rangle}{\varepsilon^{2\nu}}}
\qquad\text{if\ \ }\abs{Rx}\ge\varepsilon^{\nicefrac{\nu}{2}} \,.\nonumber
\end{align}
If $\abs{Rx}<\varepsilon^{\nicefrac{\nu}{2}}$, then $\Check{v}^{\varepsilon}=0$,
and we obtain
\begin{align}\label{EL3.5F}
\Lg^{\varepsilon}_{\Check{v}^{\varepsilon}}\,F(x)
&\;=\;\bigl(\varepsilon^{2\nu}\trace(R) + 2\abs{Rx}^{2}
+ 2\langle m(x), Rx\rangle\bigr)\,
\E^{\frac{\langle x, R x\rangle}{\varepsilon^{2\nu}}} \\[5pt]
&\;\le\;
\bigl(\varepsilon^{2\nu} - 2\abs{Rx}^{2}
+2\abs{Mx-m(x)}\abs{Rx}\bigr)\,
\E^{\frac{\langle x, R x\rangle}{\varepsilon^{2\nu}}} \nonumber\\[5pt]
&\;\le\;
\bigl(\varepsilon^{2\nu} - \abs{Rx}^{2}\bigr)\,
\E^{\frac{\langle x, R x\rangle}{\varepsilon^{2\nu}}}
\qquad\text{if\ \ } \abs{Rx}<
\varepsilon^{\nicefrac{\nu}{2}}\wedge \tfrac{1}{2}\norm{R}^2\,\Tilde{C}_m^{-1}\,,
\nonumber
\end{align}
where in the first inequality we use \eqref{EL3.5C},
and in the second we use \eqref{E-Cm2}.
Thus selecting $\varepsilon_0$ as
\begin{equation*}
\varepsilon_0 \;\df\; 1\wedge
\Bigl(\tfrac{1}{2}\norm{R}^2\,\Tilde{C}_m^{-1}\Bigr)^{\nicefrac{2}{\nu}}\,,
\end{equation*}
provided $\varepsilon<\varepsilon_0$,
\eqref{EL3.5F} holds for all $x$ such that
$\abs{Rx}< \varepsilon^{\nicefrac{\nu}{2}}$.
It follows by \eqref{EL3.5D} and \eqref{EL3.5F} that
$\Lg^{\varepsilon}_{\Check{v}^{\varepsilon}}\,F(x)\le0$
if $\abs{Rx}\ge \varepsilon^{\nu}$, and
\begin{equation}\label{EL3.5G}
\sup\;\bigl\{\Lg^{\varepsilon}_{\Check{v}^{\varepsilon}}\,F(x)\,\colon
\abs{Rx}\le \varepsilon^{\nu}\,,~\varepsilon<\varepsilon_0\bigr\}
\;\le\;\E^{\norm{R^{-1}}}\varepsilon^{2\nu}
\qquad\forall\,\varepsilon<\varepsilon_{0}\,.
\end{equation}

Thus, by \eqref{EL3.5D}, \eqref{EL3.5F}, and \eqref{EL3.5G}, 
we obtain
\begin{equation}\label{EL3.5H}
\Lg^{\varepsilon}_{\Check{v}^{\varepsilon}}\,F(x)
\;\le\; \E^{\norm{R^{-1}}}\varepsilon^{2\nu}\,
\Ind_{\{\abs{Rx}\,\le\,\varepsilon^{\nu}\}}
-\bigl(\abs{Rx}^{2}-\varepsilon^{2\nu}\bigr)\,
\E^{\frac{\langle x, R x\rangle}{\varepsilon^{2\nu}}}\,
\Ind_{\{\abs{Rx}\,\ge\,\varepsilon^{\nu}\}}
\end{equation}
for all $x\in\Rd$ and $\varepsilon<\varepsilon_{0}$.
Note that \eqref{EL3.5H} is a Foster--Lyapunov equation and $F$ is inf-compact.
Therefore $\Check{v}^{\varepsilon}$ is a stable Markov control with invariant measure 
$\Check\mu^{\varepsilon}$.
Thus, integrating \eqref{EL3.5H} with respect to the invariant probability measure
$\Check\mu^{\varepsilon}$, we obtain
\begin{equation}\label{EL3.5I}
\int_{\{\abs{Rx}\,\ge\,\varepsilon^{\nu}\}}
\bigl(\abs{Rx}^{2}-\varepsilon^{2\nu}\bigr)\,
\E^{\frac{\langle x, R x\rangle}{\varepsilon^{2\nu}}}\,
\Check\mu^{\varepsilon}(\D{x})\;\le\; \E^{\norm{R^{-1}}}\varepsilon^{2\nu}
\qquad\forall\,\varepsilon<\varepsilon_{0}\,.
\end{equation}

For any $a\in(0,1)$ we have
\begin{equation}\label{EL3.5J}
\abs{y}^2 \;\le\;
\frac{\abs{y}^{2}-a^4}{1-a^2} 
\qquad\text{if\ \ } \abs{y}\ge a\,.
\end{equation}
Thus using \eqref{EL3.5I}, and applying \eqref{EL3.5J}
with $a=\varepsilon^{\nicefrac{\nu}{2}}$, and the inequality
$\langle x, R x\rangle\ge \norm{R}^{-1}\abs{Rx}^2$, we obtain
\begin{align}\label{EL3.5K}
&\int_{\abs{Rx}\,\ge\,\varepsilon^{\nicefrac{\nu}{2}}}
\abs{Rx}^{2}\,\Check\mu^{\varepsilon}(\D{x})\\[5pt]
&\;\le\;
\int_{\abs{Rx}\,\ge\,\varepsilon^{\nicefrac{\nu}{2}}}
\frac{\abs{Rx}^{2}-\varepsilon^{2\nu}}{1-\varepsilon^{\nu}}\,
\E^{-\norm{R}^{-1}\varepsilon^{-\nu}}\,
\E^{\frac{\langle x, R x\rangle}{\varepsilon^{2\nu}}}\,
\Check\mu^{\varepsilon}(\D{x})\nonumber\\[5pt]
&\;\le\;
\frac{1}{1-\varepsilon^{\nu}}\,\E^{-\norm{R}^{-1}\varepsilon^{-\nu}}\,
\int_{\abs{Rx}\,\ge\,\varepsilon^{\nu}}
\bigl(\abs{Rx}^{2}-\varepsilon^{2\nu}\bigr)\,
\E^{\frac{\langle x, R x\rangle}{\varepsilon^{2\nu}}}\,
\Check\mu^{\varepsilon}(\D{x})\nonumber\\[5pt]
&\;\le\; \E^{\norm{R^{-1}}}\,\frac{\varepsilon^{2\nu}}{1-\varepsilon^{\nu}}\,
\E^{-\norm{R}^{-1}\varepsilon^{-\nu}}
\qquad \forall\,\varepsilon<\varepsilon_0\,.\nonumber
\end{align}

Similarly, by \eqref{EL3.5I}, and using the inequality
$(N^2-1)\abs{y}^2\le N^2(\abs{y}^2-\varepsilon^{2\nu})$ if
$\abs{y}\ge N\varepsilon^{\nu}$
for any $N\geq 2$, we obtain
\begin{equation}\label{EL3.5L}
\int_{\abs{Rx}\,\ge\, N\varepsilon^\nu}
\abs{Rx}^{2}\,\Check\mu^{\varepsilon}(\D{x})\;\le\;
 \E^{\norm{R^{-1}}}
\frac{N^2\varepsilon^{2\nu}}{N^2-1}\,
\E^{-N^{-2}\varepsilon^{-2\nu}\norm{R}^{-1}}
\end{equation}
for all $\varepsilon<\varepsilon_0$.

Also, since by definition
$\Check{v}^{\varepsilon}=0$ for $\abs{Rx}\le \varepsilon^{\nicefrac{\nu}{2}}$, and
$\abs{\Check{v}^{\varepsilon}(x)}\; \le\; \Tilde{C}_m\frac{\abs{x}}{\varepsilon}$
by \eqref{E-Cm2}, it follows
by \eqref{EL3.5K} that
\begin{equation}\label{EL3.5M}
\int_{\Rd}\abs{\Check{v}^{\varepsilon}(x)}^2 \Check\mu^{\varepsilon}(\D{x})
\;\le\;\norm{R^{-1}}^2\,
\frac{\Tilde{C}_m^2}{1-\varepsilon^{\nu}}\,\E^{\norm{R^{-1}}}\,
\varepsilon^{2\nu-2}\,
\E^{-\norm{R}^{-1}\varepsilon^{-\nu}}
\end{equation}
for all $\varepsilon<\varepsilon_0$.
Then \eqref{EL3.5A} follows from \eqref{EL3.5K} and \eqref{EL3.5M}, by
choosing a common constant $c_0$.

Consider the `scaled' diffusion
\begin{equation*}
\D \Hat{X}_{t} \;=\; \Hat{b}^{\varepsilon}(\Hat{X}_{t})\,\D{t}
+ \,\D{W}_{t}\,, \quad t \ge 0\,,
\end{equation*}
where
\begin{equation*}
\Hat{b}^{\varepsilon}\;\df\;
\frac{m(\varepsilon^{\nu}x)+\varepsilon\,\Check{v}^{\varepsilon}(\varepsilon^{\nu}x)}
{\varepsilon^{\nu}}\,.
\end{equation*}
and let $\Hat\mu^{\varepsilon}$ 
denote its invariant probability measure.
It $\Check\varrho^{\varepsilon}$ and $\Hat\varrho^{\varepsilon}$
denote the densities of $\Check\mu^{\varepsilon}$ and $\Hat\mu^{\varepsilon}$
respectively, then
$\varepsilon^{\nu d}\Check\varrho^{\varepsilon}(\varepsilon^{\nu}x)
=\Hat\varrho^{\varepsilon}(x)$
for all $x\in\Rd$.
Substituting $x=\varepsilon^{\nu}y$ in \eqref{EL3.5I} we deduce that
the family of probability measures
$\{\Hat{\mu}^{\varepsilon}\, \colon \varepsilon\in(0,1)\}$ is tight.
The (discontinuous) drift $\Hat{b}^{\varepsilon}$ converges to $Mx$
as $\varepsilon\searrow0$, uniformly on compact sets.
This implies that $\Hat\varrho^{\varepsilon}$ converges, as $\varepsilon\searrow0$,
to the Gaussian density $\rho^{~}_{\Sigma}$ with mean $0$ and covariance matrix
$\Sigma$, given by $M\Sigma+\Sigma M\transp = -I$, i.e, $\Sigma=\tfrac{1}{4}R^{-1}$,
uniformly on compact sets.
Indeed, 
since $\Hat{b}^{\varepsilon}$ is locally bounded uniformly in $\varepsilon\in(0,1)$,
and the family $\{\Hat\mu^\varepsilon\,,\;\varepsilon\in(0,1)\}$ is tight,
the densities $\Hat\varrho^\varepsilon$ of $\Hat\mu^\varepsilon$ are 
locally H\"older
equicontinuous (see Lemma~3.2.4 in \cite{book}).
Let $\Hat\varrho$ be any limit point of $\Hat\varrho^{\varepsilon_n}$
along some sequence $\varepsilon_n\searrow0$.
Since $\{\Hat{\mu}^{\varepsilon}\, \colon \varepsilon\in(0,1)\}$ is tight
it follows that $\Hat\varrho^{\varepsilon_n}$
also converges in $L^{1}(\Rd)$, as $n\to\infty$, and hence
$\int_{\Rd}\Hat\varrho(x)\,\D{x}=1$.
With $\Hat\Lg^\varepsilon \df \frac{1}{2}\Delta
+ \langle\Hat{b}^{\varepsilon},\grad\rangle$
and $\Hat\Lg^0 \df \frac{1}{2}\Delta + \langle Mx,\grad\rangle$,
and since $\int_{\Rd}\Hat\Lg^\varepsilon f(x)\,\Hat\varrho^\varepsilon(x)\,\D{x}=0$
for all $f\in\Ccc^\infty(\Rd)$,
we have
\begin{equation}\label{EL3.5N}
\int_{\Rd}\Hat\Lg^0 f(x)\,\Hat\varrho(x)\,\D{x} \;=\;
\int_{\Rd}\bigl(\Hat\Lg^0 f(x)-\Hat\Lg^\varepsilon f(x)\bigr)\,\Hat\varrho(x)\,\D{x}
+ \int_{\Rd}\Hat\Lg^\varepsilon f(x)\,
\bigl(\Hat\varrho(x)-\Hat\varrho^\varepsilon(x)\bigr)\,\D{x}
\end{equation}
for all $f\in\Ccc^\infty(\Rd)$.
It is clear that both terms on the right hand side of
\eqref{EL3.5N} converge to $0$ as $\varepsilon=\varepsilon_n\searrow0$.
This implies that $\Hat\varrho$ is the density of the invariant
probability measure of the diffusion $\D{X_t} = M X_t\,\D{t} + \D{W}_t$,
which is Gaussian as claimed.

Since the Gaussian density $\rho^{~}_{\Sigma}$ has zero mean, then by uniform
integrability implied by \eqref{EL3.5L} we have
\begin{equation}\label{EL3.5O}
\varepsilon^{-\nu}\;
\int_{\Rd}\bigl(D\ell(0)x\bigr)\,\Check\mu^{\varepsilon}(\D{x})
\;\xrightarrow[\varepsilon\searrow0]{}\;0\,.
\end{equation}
It follows by \eqref{EL3.5L} that for some constant $\kappa_1>0$
we have $\int_{\Rd} \abs{x}^2\Check\mu^{\varepsilon}(\D{x})<\kappa_1$
for all $\varepsilon<\varepsilon_0$.
Thus, using \eqref{E-Cl2}, we obtain
\begin{equation}\label{EL3.5P}
\varepsilon^{-\nu}\;
\int_{\Rd}\abs{\ell(x)-\ell(0)- D\ell(0)x}\,\Check\mu^{\varepsilon}(\D{x})
\;\le\;\kappa_1\Tilde{C}_\ell\,\varepsilon^\nu\,.
\end{equation}
Combining \eqref{EL3.5O}--\eqref{EL3.5P}, we obtain \eqref{EL3.5B}.

Next we turn to part~(ii).
Consider the control
$v^{\varepsilon}(x)=\varepsilon^{-1}
\bigl(Mx-m(x)\bigr)$ for $x\in\Rd$.
Then $m(x)+\varepsilon v^{\varepsilon}(x)=Mx$ and
the associated invariant measure $\mu^{\varepsilon}$ is Gaussian with mean $0$ and
covariance matrix $\varepsilon^{2\nu}\Sigma$.
Using the bound in \eqref{E-Cm2},
we obtain
\begin{equation*}
\int_{\Rd} \abs{v^{\varepsilon}}^{2}\,\D \mu^{\varepsilon}
\;\le\; \int_{\Rd} \Tilde{C}_m^{2}\,\varepsilon^{-2}\abs{x}^{4}\,
\mu^{\varepsilon}(\D{x})
\;\in\; \order\bigl(\varepsilon^{4\nu-2}\bigr)\,.
\end{equation*}
Since $\mu^{\varepsilon}$ has zero mean, using a triangle inequality,
and \eqref{E-Cl2}, as in the proof of Lemma~\ref{L3.2}, we obtain
\begin{equation*}
\babss{\int_{\Rd} \bigl(\ell(x)-\ell(0)\bigr)\,\mu^{\varepsilon}(\D{x})}
\;\le\; \varepsilon^{2\nu}\,\Tilde{C}_\ell\,\trace(\Sigma)\,.
\end{equation*}
Since $4\nu-2<2\nu$ for $\nu<1$, we obtain that $\beta^{\varepsilon}_{*}\leq \fJs +
\order\bigl(\varepsilon^{4\nu-2}\bigr)$.
On the other hand, by part (1) we already know that
$\beta^{\varepsilon}_{*}\leq \fJs + \sorder(\varepsilon^{\nu})$.
To complete
the proof we observe that $\nu\leq 4\nu-2$ for $\nu\ge \nicefrac{2}{3}$ and 
$\nu> 4\nu-2$ for $\nu< \nicefrac{2}{3}$.
\end{proof}

%%%%%%%%%%%%%%%%%%%%%%%%%%%%%%%%%%%%%%%%%%%%%%%%%%%%%%%%%%%%%%%%%%%%%%%%%%%%%%%%
\subsection{Results concerning the subcritical regime}

By Lemma~\ref{L3.5} we can always find a stable admissible
control such that the corresponding invariant probability
measure concentrates on a stable equilibrium point as $\varepsilon\searrow0$,
while keeping the ergodic cost in \eqref{cost} bounded, uniformly
in $\varepsilon\in(0,1)$.
Now we proceed to show that for $\nu<1$, $\eta_{*}^{\varepsilon}$
concentrates on $\cSs$.

%%%%%%%%%%%%%%%%%%%%%%%%%%%%%%%%%%%%%%%%%%%%%%%%%%%%%%%%%%%%%%%%%%%%%%%%%%%%%%%%
\begin{lemma}\label{L3.6}
Suppose $\nu<1$.
Then
\begin{equation*}
\eta_{*}^{\varepsilon}(\cS\setminus\cSs)\;\xrightarrow[\varepsilon\searrow0]{}\;0\,,
\quad\text{and}\quad
\lim_{\varepsilon\searrow0}\,\beta^\varepsilon_*\;=\; \fJs\,.
\end{equation*}
\end{lemma}

\begin{proof}
We argue by contradiction.
Suppose that
$$\limsup_{\varepsilon\searrow0}\,\eta_{*}^{\varepsilon}\bigl(B_{r}(z)\bigr)\;>\; 0$$
for some $r>0$ and $z\notin \cSs$.
In Theorem~\ref{T2.2} we may select $a_z$ such that
$a_z\ne a_z'$ for $z\ne z'$.
Thus by Theorem~\ref{T2.2}\,(ii), there exists
$\delta>0$ be such that the interval 
$(\Lyap(z)-3\delta,\Lyap(z)+3\delta)$ contains no other
critical values of $\Lyap$ other than $\Lyap(z)$.
Let $\varphi\in\Cc^2(\RR)$ be such that
\begin{itemize}
\item[(a)]
$\varphi(\Lyap(z)+y)=y$ for
$y\in(\Lyap(z)-\delta,\Lyap(z)+\delta)\,$;
\smallskip
\item[(b)]
$\varphi'\in[0,1]$ on $(\Lyap(z)-2\delta,\Lyap(z)+2\delta)\,$;
\smallskip
\item[(c)]
$\varphi'= 0$ on $(\Lyap(z)-2\delta,\Lyap(z)+2\delta)^{c}\,$.
\end{itemize}
Select $r>0$ such that
\begin{equation}\label{EL3.6A}
\sup_{x\in B_{r}(z)}\;\babs{\Delta\Lyap(x)-\Delta\Lyap(z)}
\;<\;\frac{1}{2}\babs{\Delta\Lyap(z)}\,.
\end{equation}
Note that by Theorem~\ref{T2.2} and Lemma~\ref{L2.3} the
function $\Lyap$ takes distinct values on $\cS$.  
Therefore we may also choose this $r$ small enough so that
\begin{equation*}
B_{r}(z)\;\subset\; \{x\,\colon \abs{\Lyap(x)-\Lyap(z)}\le \delta\}
\;\subset\; B^{c}_{r}(S\setminus\{z\})\,.
\end{equation*}

By the infinitesimal characterization of an invariant probability measure
we have
\begin{equation*}
\int_{\Rd} \Lg_{v^\varepsilon_*}^{\varepsilon}(\varphi\circ\Lyap)(x)\,
\eta_{*}^{\varepsilon}(\D{x}) \;=\;0\,,
\end{equation*}
which we write as
\begin{multline}\label{EL3.6B}
\frac{\varepsilon^{2\nu}}{2}\biggl(\int_{\Rd}\varphi'(\Lyap)\Delta\Lyap\,
\D\eta_{*}^{\varepsilon}
+\int_{\Rd}\varphi''(\Lyap)\,\abs{\grad\Lyap}^2\,
\D\eta_{*}^{\varepsilon}\biggr)\\[5pt]
+\varepsilon\int_{\Rd}\varphi'(\Lyap)\,
\langle v^\varepsilon_*,\grad\Lyap\rangle\,\D\eta_{*}^{\varepsilon}
+\int_{\Rd}\varphi'(\Lyap)\langle m,\grad\Lyap\rangle\,
\D\eta_{*}^{\varepsilon}\;=\;0\,.
\end{multline}
Recall the definition of the optimal control effort $\sG^\varepsilon_*$
in \eqref{E-sG}, and also define
\begin{equation}\label{EL3.6C}
\begin{gathered}
\zeta^{\varepsilon}\;\df\;\biggl(\int_{\Rd}
\varphi'(\Lyap)\,\abs{\grad \Lyap}^{2}\,
\D\eta_{*}^{\varepsilon}\biggr)^{\nicefrac{1}{2}}\,,\\[5pt]
\xi^\varepsilon_1\;\df\;\frac{1}{2}\,\int_{\Rd}\varphi'(\Lyap)\Delta\Lyap\,
\D\eta_{*}^{\varepsilon}\,,\qquad
\xi^\varepsilon_2\;\df\;\frac{1}{2}\,
\int_{\Rd}\varphi''(\Lyap)\,\abs{\grad\Lyap}^2\,\D\eta_{*}^{\varepsilon}\,,
\end{gathered}
\end{equation}%
\nomenclature[Ck]{$\zeta^{\varepsilon}$, $\xi^\varepsilon_1$,
$\xi^\varepsilon_2$}{constants, equation \eqref{EL3.6C}}%
and $\xi^\varepsilon\df\xi^\varepsilon_1+\xi^\varepsilon_2$.
By the Cauchy--Schwarz inequality we have
\begin{equation}\label{EL3.6D}
\babss{\int_{\Rd} \varphi'(\Lyap)
\langle v^\varepsilon_*,\grad\Lyap\rangle\,
\D\eta_{*}^{\varepsilon}}
\;\le\;\norm{\sqrt{\varphi'}}_\infty\,
\sqrt{2\sG^\varepsilon_*}\,\zeta^{\varepsilon}
\;\le\;
\sqrt{2\sG^\varepsilon_*}\,\zeta^{\varepsilon}\,.
\end{equation}

By Theorem~\ref{T2.2}\,(iv)
we have
$C_0\,(\zeta^{\varepsilon})^{2}\le
-\int_{\Rd}\varphi'(\Lyap)\langle m,\grad\Lyap\rangle\,
\D\eta_{*}^{\varepsilon}$.
Therefore, by \eqref{EL3.6B} and \eqref{EL3.6D} we obtain
\begin{equation}\label{EL3.6E}
C_0\,(\zeta^{\varepsilon})^{2}
- \varepsilon\,\sqrt{2\sG^\varepsilon_*}\,\zeta^{\varepsilon}
- \varepsilon^{2\nu}\,\xi^\varepsilon\;\le\;0\,.
\end{equation}

We write
\begin{equation}\label{EL3.6F}
\xi^\varepsilon_1 \;=\;
\int_{B_{r}(z)}\varphi'(\Lyap)\Delta\Lyap\,\D\eta_{*}^{\varepsilon}
+ \int_{B^{c}_{r}(z)}\varphi'(\Lyap)\Delta\Lyap\,\D\eta_{*}^{\varepsilon}\,.
\end{equation}
Since $\Lyap$ is inf-compact, it follows that $\varphi\circ\Lyap$
is constant outside a compact set.
Therefore, the support of $\varphi'(\Lyap(\cdot))$ is compact,
and as a result $\Delta\Lyap$ is bounded on this set.
By \eqref{EL3.6A}, \eqref{EL3.6F}, Theorem~\ref{T2.2}\,(iii), and since
$\eta_{*}^{\varepsilon}\bigl(B^{c}_{r}(\cS)\bigr)\searrow0$
as $\varepsilon\searrow0$ (by Lemma~\ref{L3.1}),
we obtain
\begin{equation}\label{EL3.6G}
\limsup_{\varepsilon\searrow0}\; (-\xi^\varepsilon_1)
\;\ge\; -
\tfrac{1}{2}
\Delta\Lyap(z)\,
\limsup_{\varepsilon\searrow0}\;\eta_{*}^{\varepsilon}\bigl(B_{r}(z)\bigr)\;>\; 0\,.
\end{equation}

On the other hand, since $\varphi''(\Lyap)=0$ on some open neighborhood of
$\cS$, it follows that $\xi^\varepsilon_2\to0$ as $\varepsilon\searrow0$.
Therefore, we have $\limsup_{\varepsilon\searrow0}\, (-\xi^\varepsilon)>0$.
However, since the discriminant of \eqref{EL3.6E} must be nonnegative, we obtain
\begin{equation}\label{E-disc}
\varepsilon^2\sG^\varepsilon_* \;\ge\; -2\,
C_0\,\varepsilon^{2\nu}\xi^\varepsilon\,,
\end{equation}
which leads to a contradiction.
Hence, $\eta_{*}^{\varepsilon}(\cS\setminus\cSs)\,
\xrightarrow[\varepsilon\searrow0]{}\,0$.
This implies that
$\liminf_{\varepsilon\searrow0}\,\beta^\varepsilon_*\ge \fJs$,
which combined with Lemma~\ref{L3.5}\,(ii),
results in equality for the limit as claimed.
\end{proof}

We revisit the subcritical regime in Corollary~\ref{C4.2} to obtain
a lower bound for $\beta^{\varepsilon}_{*}$.

It is worthwhile at this point to present the following one-dimensional
example, which shows how the value of $\beta^\varepsilon_*$
for small $\varepsilon$ bifurcates as we cross the critical regime.

%%%%%%%%%%%%%%%%%%%%%%%%%%%%%%%%%%%%%%%%%%%%%%%%%%%%%%%%%%%%%%%%%%%%%%%%%%%%%%%%
\begin{example}
Let $d=1$,  $m(x)= Mx$, and $\ell(x) = \frac{1}{2}Lx^{2}$, with $M>0$ and $L>0$.
Then the solution to \eqref{E-HJB2} takes the form
\begin{align*}
V^{\varepsilon}&\;=\;
\frac{M+\sqrt{M^{2} + L\varepsilon^{2}}}{2\varepsilon^{2}}\,x^{2}\,,\\[5pt]
\beta^\varepsilon_*&\;=\;\frac{\varepsilon^{2\nu-2}}{2}\,
\bigl(M+\sqrt{M^{2} + L\varepsilon^{2}}\bigr)\,.
\end{align*}
Note that $\beta^\varepsilon_*\to\ell(0)=0$,
$\beta^\varepsilon_*\to M$, and $\beta^\varepsilon_*\to\infty$,
as $\varepsilon\searrow0$,
when $\nu>1$, $\nu=1$, and $\nu<1$, respectively.
\end{example}

%%%%%%%%%%%%%%%%%%%%%%%%%%%%%%%%%%%%%%%%%%%%%%%%%%%%%%%%%%%%%%%%%%%%%%%%%%%%%%%%
\section{Concentration bounds for the optimal
stationary distribution}\label{S4}

We start with the following lemma, which is valid for all $\nu$.

%%%%%%%%%%%%%%%%%%%%%%%%%%%%%%%%%%%%%%%%%%%%%%%%%%%%%%%%%%%%%%%%%%%%%%%%%%%%%%%%
\begin{lemma}\label{L4.1}
For any bounded domain $G$ there exists a constant
$\Hat{\kappa}_{0}=\Hat{\kappa}_{0}(G,\nu)$ such that
\begin{equation}\label{EL4.1A}
\int_{G} \bigl(\dist(x,\cS)\bigr)^{2}\,\eta_{*}^{\varepsilon}(\D{x}) \;\le\;
\Hat{\kappa}_{0}\,\varepsilon^{2(\nu\wedge2)}\qquad
\forall\,\nu>0\,,\qquad\forall\,\varepsilon\in(0,1)\,,
\end{equation}
where $\dist(x,\cS)$ denotes the Euclidean distance of $x$ from
the set $\cS$.
\end{lemma}

\begin{proof}
We fix some bounded domain $G$ which, without loss of
generality contains $\cS$, and choose
some number $\delta$ such that
$\delta\ge \sup_{x\in G}\,\Lyap(x)$. Without loss of generality assume that
$\ell(x)>\fJ$ for all $x\in G^{c}$, otherwise we enlarge $G$.
Let $\Tilde\varphi\colon\RR\to\RR$ be a smooth function such that
\begin{itemize}
\item[(a)]
$\Tilde\varphi(y)=y$ for
$y\in(-\infty,\delta)\,$;
\smallskip
\item[(b)]
$\Tilde\varphi'\in(0,1)$ on $(\delta,2\delta)\,$;
\smallskip
\item[(c)]
$\Tilde\varphi'= 0$ on $[2\delta,\infty)\,$;
\smallskip
\item[(d)]
$\Tilde\varphi''\le 0$.
\end{itemize}
Define $\Tilde\zeta^\varepsilon$, $\Tilde\xi^\varepsilon_1$,
and $\Tilde\xi^\varepsilon_2$,  as in \eqref{EL3.6C}
by replacing $\varphi$ with $\Tilde\varphi$, and let
$\Tilde\xi^\varepsilon\df \Tilde\xi^\varepsilon_1+\Tilde\xi^\varepsilon_2$.
As in \eqref{EL3.6E} we obtain
\begin{equation}\label{EL4.1B}
C_0\,(\Tilde\zeta^{\varepsilon})^{2}
- \varepsilon\,\sqrt{2\sG^\varepsilon_*}\,\Tilde\zeta^{\varepsilon}
- \varepsilon^{2\nu}\,\Tilde\xi^\varepsilon\;\le\;0\,.
\end{equation}
By Theorem~\ref{T2.2}\,(iv) we have
\begin{equation}\label{EL4.1C}
\int_{\{x\,\colon \Lyap(x)\,\le\,\delta\}}
\bigl(\dist(x,\cS)\bigr)^{2}\,
\eta_{*}^{\varepsilon}(\D{x}) \;\le\; C_0^{-1}\,(\Tilde\zeta^{\varepsilon})^{2}\,.
\end{equation}
By an application of Young's inequality
to \eqref{EL4.1B}, we obtain
$$\frac{C_0}{2}\,(\Tilde\zeta^{\varepsilon})^{2}
- \frac{1}{C_0} \varepsilon^2 \sG^\varepsilon_* - \varepsilon^{2\nu}
\Tilde\xi^{\varepsilon}\;\le\;0\,,$$
and hence we have
$\Tilde\zeta^{\varepsilon}\in\order\bigl(\varepsilon^{\nu\wedge 2})$.
Thus \eqref{EL4.1A} follows by \eqref{EL4.1C}.
\end{proof}

%%%%%%%%%%%%%%%%%%%%%%%%%%%%%%%%%%%%%%%%%%%%%%%%%%%%%%%%%%%%%%%%%%%%%%%%%%%%%%%%
\begin{corollary}\label{C4.2}
Suppose $\nu\ge1$.
Then following hold.
\begin{enumerate}
\item[\textup{(}a\/\textup{)}]
The optimal control effort $\sG^\varepsilon_*$ satisfies
\begin{equation}\label{EC4.2A}
\begin{aligned}
&\sG^\varepsilon_*\;\in\;\order\bigl(\varepsilon^{\nu\wedge2}\bigr)
&&\text{if\ } \fJ=\fJs,\ \nu>1\,,
\text{\ or if\ \ }\fJc=\fJs,\ \nu=1\,,\\[5pt]
&\sG^\varepsilon_*\;\in\;\order\bigl(\varepsilon^{(2\nu-2)\wedge2}\bigr)
&&\text{if\ } \fJ<\fJs\text{\ and \ }\nu>1\,,
\end{aligned}
\end{equation}
and
\begin{equation}\label{EC4.2B}
\liminf_{\varepsilon\searrow0}\,
\frac{1}{\varepsilon^{2\nu-2}}\,\sG^\varepsilon_* \;>\;0
\qquad\text{if\ \ } \fJ<\fJs \text{\ and\ \ } \nu>1\,.
\end{equation}
\item[\textup{(}b\/\textup{)}]
$\beta^{\varepsilon}_{*}-\fJ\ge\order\bigl(\varepsilon^{\nu\wedge2}\bigr)$
for $\nu>1$.
\end{enumerate}
\end{corollary}

\begin{proof}
Select a domain $G$ as in the proof of Lemma~\ref{L4.1}.
Define $\Tilde\zeta^\varepsilon$, $\Tilde\xi^\varepsilon_1$,
and $\Tilde\xi^\varepsilon_2$  as in \eqref{EL3.6C}
by replacing $\varphi$ with $\Tilde\varphi$, and let
$\Tilde\xi^\varepsilon\df \Tilde\xi^\varepsilon_1+\Tilde\xi^\varepsilon_2$.
Then \eqref{EL4.1B} holds,
and thus $\Tilde\zeta^{\varepsilon}\in\order\bigl(\varepsilon^{\nu\wedge 2})$.
Recall the notation in Definition~\ref{D1.10}.
With $C_\ell$ a Lipschitz constant
for $\ell$, and some fixed $\Bar{z}\in\cZ$, we have
\begin{equation*}
\ell(x)-\fJ\;=\;\bigl(\ell(x) - \ell(z)\bigr)
+ \bigl(\ell(z) - \ell(\Bar{z})\bigr)\;\ge\; -C_\ell \abs{x-z}\qquad
\forall\,z\in\cS\,,\ \forall\,x\in\Rd\,,
\end{equation*}
since $\ell(z) - \ell(\Bar{z})\ge0$ for all $z\in\cS$.
Therefore, we obtain
\begin{equation}\label{EC4.2C}
\ell(x)-\fJ\;\ge\; -C_{\ell} \dist(x,\cS)\qquad\forall\,x\in\Rd\,,
\end{equation}
and using the Cauchy--Schwarz inequality, and the
assumption that $\ell(x)>\fJ$ on $G^c$, we deduce from \eqref{EC4.2C}
and Theorem~\ref{T2.2}\,(iv) that
\begin{equation}\label{EC4.2D}
\int_{\Rd}\,\ell\,\D\eta_{*}^{\varepsilon}-
\fJ\;\ge\; \int_{G}(\ell(x)-\fJ)\D\eta_{*}^{\varepsilon}
\;\ge\;-\frac{C_\ell}{\sqrt{C_0}}\,\Tilde\zeta^{\varepsilon}\,.
\end{equation}
Thus by \eqref{EC4.2D} and non-negativity of $\sG^\varepsilon$ we have
\begin{equation}\label{EC4.2E}
-\frac{C_\ell}{\sqrt{C_0}}\,\Tilde\zeta^{\varepsilon}\;\le\;
\beta^\varepsilon_*-\fJ\,.
\end{equation}
By \eqref{EC4.2D}--\eqref{EC4.2E}, we obtain
\begin{align}\label{EC4.2F}
\sG^\varepsilon_*
&\;\le\; \beta^{\varepsilon}_{*}-\int_{\Rd}\ell\, \D\eta^{\varepsilon}_{*}
\\[5pt]
&\;\le\; \beta^{\varepsilon}_{*}-\fJ
+\fJ-\int_{\Rd}\ell\, \D\eta^{\varepsilon}_{*}\nonumber\\[5pt]
&\;\le\; \beta^{\varepsilon}_{*}-\fJ
+ \frac{C_\ell}{\sqrt{C_0}}\,\Tilde\zeta^{\varepsilon}\,.\nonumber
\end{align}
By an application of Young's inequality
to \eqref{EL4.1B}, we obtain
$$\frac{C_0}{2}\,(\Tilde\zeta^{\varepsilon})^{2}
- \frac{1}{C_0} \varepsilon^2 \sG^\varepsilon_* - \varepsilon^{2\nu}
\Tilde\xi^{\varepsilon}\;\le\;0\,,$$
and thus
\begin{equation}\label{EC4.2G}
\Tilde\zeta^{\varepsilon}\;\le\;
\frac{\sqrt 2}{C_0} \varepsilon \sqrt{\sG^\varepsilon_*}
+\frac{\sqrt 2}{\sqrt{C_0}} \varepsilon^{\nu}
\sqrt{\abs{\Tilde\xi^{\varepsilon}}}\,.
\end{equation}
Combining \eqref{EC4.2F}--\eqref{EC4.2G}, and using
again Young's inequality in the form
$\frac{C_\ell}{\sqrt{C_0}}\frac{\sqrt 2}{C_0} \varepsilon \sqrt{\sG^\varepsilon_*}
\le \frac{C_\ell^2}{C_0^3}\,\varepsilon^2+\tfrac{1}{2}\,\sG^\varepsilon_*$,
and rearranging terms, we have
\begin{equation}\label{EC4.2H}
\tfrac{1}{2}\,\sG^\varepsilon_*
\;\le\; \beta^{\varepsilon}_{*}-\fJ
+ \frac{C_\ell^2}{C_0^3}\,\varepsilon^2
+ \frac{\sqrt{2}C_\ell}{C_0}\varepsilon^{\nu}\,
\sqrt{\abs{\Tilde\xi^{\varepsilon}}}\,.
\end{equation}

By Lemma~\ref{L3.3} and \eqref{EC4.2H} we obtain
$\sG^\varepsilon_*\in\order(\varepsilon^{\nu\wedge2}\bigr)$
if $\fJ=\fJs$ for $\nu>1$, or if $\fJc=\fJs$ and $\nu=1$.
We also obtain
$\sG^\varepsilon_*\in\order\bigl(\varepsilon^{2\wedge{(2\nu-2)}}\bigr)$
if $\fJ<\fJ_s$ and $\nu>1$.
Thus \eqref{EC4.2A} holds.

If $\fJ<\fJs$ and $\nu>1$, then $\cZ\subset\cS\setminus\cSs$,
and $\sS\subset\cZ$ by Lemma~\ref{L3.3}.
Fix some $z\in\sS$.
Then $\liminf_{\varepsilon\searrow0}\,\eta_{*}^{\varepsilon}\bigl(B_{r}(z)\bigr)> 0$
for any $r>0$.
Also $\Delta\Lyap(z)<0$ by Theorem~\ref{T2.2}.
Therefore \eqref{EL3.6G} holds, with `$\liminf$' replacing the `$\limsup$'.
Expanding $\Tilde\xi^\varepsilon_1$ as in \eqref{EL3.6F},
and arguing as in Lemma~\ref{L3.6} it follows that
\eqref{EL3.6G} with `$\liminf$' also holds for $\Tilde\xi^\varepsilon$.
In fact, it easily follows that for some constant $\kappa_{1}$, we have
\begin{equation}\label{EC4.2I}
\liminf_{\varepsilon\searrow0}\,\bigl(-\Tilde\xi^\varepsilon\bigr) \;\ge\;
\min_{z\in\cZ}\; \kappa_{1}\bigl(-\tfrac{1}{2}\Delta\Lyap(z)\bigr)\,.
\end{equation}
The discriminant of the quadratic polynomial in \eqref{EL4.1B} is nonnegative
and this implies that
\begin{equation}\label{E-disc2}
\varepsilon^2\sG^\varepsilon_* \;\ge\; -2\, C_0\,\varepsilon^{2\nu}
\Tilde\xi^\varepsilon\,,
\end{equation}
in direct analogy with \eqref{E-disc}.
Thus, \eqref{EC4.2B} follows by \eqref{EC4.2I} and \eqref{E-disc2}.
This completes the proof of part~(a).

Since $\Tilde\zeta^\varepsilon\in\order\bigl(\varepsilon^{\nu\wedge2}\bigr)$, we
obtain $\beta^{\varepsilon}_{*}-\fJ\ge\order\bigl(\varepsilon^{\nu\wedge2}\bigr)$
by \eqref{EC4.2E}.
This proves part~(b),
and completes the proof.
\end{proof}

We define the following scaled quantities.

%%%%%%%%%%%%%%%%%%%%%%%%%%%%%%%%%%%%%%%%%%%%%%%%%%%%%%%%%%%%%%%%%%%%%%%%%%%%%%%%
\begin{definition}\label{D4.3}
For $z\in\cS$, and $V^\varepsilon$ as in Theorem~\ref{T1.4}, we define
\begin{equation*}
\widehat{V}_{z}^{\varepsilon}(x)\;\df\;V^{\varepsilon}(\varepsilon^{\nu} x+z)\,,
\qquad x\in\Rd\,.
\end{equation*}
and
\begin{equation*}
\widetilde{V}^{\varepsilon}\;\df\; {\varepsilon}^{2}V^{\varepsilon}\,,
\qquad\Breve{V}_{z}^{\varepsilon}\;\df\;\varepsilon^{2(1-\nu)}
\widehat{V}_{z}^{\varepsilon}\,.
\end{equation*}
We also define the `scaled' vector field and penalty by
\begin{equation*}
\widehat{m}_{z}^{\varepsilon}(x)\;\df\;\frac{m(\varepsilon^{\nu} x+z)}
{\varepsilon^{\nu}}\,,
\qquad 
\widehat{\ell}_{z}^{\varepsilon}(x)\;\df\;\ell(\varepsilon^{\nu} x+z)\,.
\end{equation*}
\end{definition}
\nomenclature[Cb]{$\widehat{V}_{z}^{\varepsilon}$, $\widetilde{V}^{\varepsilon}$,
$\Breve{V}_{z}^{\varepsilon}$}{scaled solutions of the HJB, Definition~\ref{D4.3}}
\nomenclature[Ch]{$\widehat{m}_{z}^{\varepsilon}$,
$\widehat{\ell}_{z}^{\varepsilon}$}{scaled vector field and potential,
Definition~\ref{D4.3}}

The next lemma shows provides estimates for the growth of
$\grad \widehat{V}_{z}^{\varepsilon}$,
and $\grad \widetilde{V}^{\varepsilon}$.

%%%%%%%%%%%%%%%%%%%%%%%%%%%%%%%%%%%%%%%%%%%%%%%%%%%%%%%%%%%%%%%%%%%%%%%%%%%%%%%%
\begin{lemma}\label{L4.4}
Assume $\nu\in(0,2]$, and let $\widehat{V}_{z}^{\varepsilon}$,
$\widetilde{V}^{\varepsilon}$, and $\Breve{V}_{z}^{\varepsilon}$,
be as in Definition~\ref{D4.3}.
Then
\begin{itemize}
\item[\textup{(}a\/\textup{)}]
Under the restriction that $z\in\cZ$ when $\nu\in(1,2]$,
there exists a constant $\Breve{c}_0$ such that
\begin{equation}\label{EL4.4A}
\abs{\grad \Breve{V}_{z}^{\varepsilon}(x)}\;\le\; \Breve{c}_0\,(1+\abs{x})
\qquad\forall\,\varepsilon\in(0,1)\,,\quad\forall\,x\in\Rd\,.
\end{equation}
\item[\textup{(}b\/\textup{)}]
The bound in \eqref{EL4.4A} also holds for $\widetilde{V}^{\varepsilon}$
for all $\nu\in(0,2]$, with no restrictions on $z$.
\end{itemize}
\end{lemma}

\begin{proof}
By \eqref{E-HJB2}, the function $\Breve{V}_{z}^{\varepsilon}$ satisfies
\begin{equation}\label{EL4.4B}
\frac{1}{2}\Delta \Breve{V}_{z}^{\varepsilon}(x)
+ \bigl\langle\widehat{m}_{z}^{\varepsilon}(x),
\grad \Breve{V}_{z}^{\varepsilon}(x)\bigr\rangle
-\frac{1}{2}\,\abs{\grad \Breve{V}_{z}^{\varepsilon}(x)}^{2}
\;=\; \varepsilon^{2(1-\nu)}\bigl(\beta^\varepsilon_*
-\widehat{\ell}_{z}^{\varepsilon}(x)\bigr)\,.
\end{equation}
Since $\ell$ is Lipschitz, the gradient of the map
$x\mapsto\varepsilon^{2(1-\nu)}
\bigl(\widehat{\ell}_{z}^{\varepsilon}(x)-\ell(z)\bigr)$ is bounded
in $\Rd$, uniformly in $\varepsilon\in(0,1)$, and $\nu\in(0,2]$.
Similarly, $\abs{\widehat{m}_{z}^{\varepsilon}(x)}$,
$\norm{D\widehat{m}_{z}^{\varepsilon}(x)}$ and
$\norm{D^2\widehat{m}_{z}^{\varepsilon}(x)}$,
are bounded in $\Rd$, uniformly in $\varepsilon\in(0,1)$, and $\nu\in(0,2]$.
By Theorem~\ref{T1.11}\,(i), which is established in
Corollary~\ref{C4.2}, the constants
$\varepsilon^{2(1-\nu)}\bigl(\beta^\varepsilon_*-\ell(z)\bigr)$
are  bounded uniformly in $\varepsilon\in(0,1)$, and $\nu\in(1,2]$
for $z\in\cZ$.
Applying \cite[Lemma~5.1]{Met} to \eqref{EL4.4B} it follows
that $\Breve{V}_{z}^{\varepsilon}$ satisfies \eqref{EL4.4A} if $\nu\in(1,2]$
and $z\in\cZ$.
On the other hand, if $\nu\in(0,1]$, then the gradient of the right hand
side of \eqref{EL4.4B} is bounded in $\Rd$, uniformly in
$\varepsilon\in(0,1)$, and the restriction $z\in\cZ$ is not needed.
This completes the proof of part~(a).

Next show that \eqref{EL4.4A} holds for $\widetilde{V}^{\varepsilon}$.
Fix an arbitrary $z\in\cZ$.
We have
\begin{align*}
\grad_{x}V^{\varepsilon}(x+z)&\;=\;
\varepsilon^{-\nu}\,\grad_{y}\widehat{V}_{z}^{\varepsilon}(y)
\bigr|_{y=\varepsilon^{-\nu}x}\\[5pt]
&\;\le\;\frac{\varepsilon^{-\nu}}{\varepsilon^{2(1-\nu)}}\,
\Breve{c}_{0}\bigl(1+ \abs{\varepsilon^{-\nu}x}\bigr)\\[5pt]
&\;=\;
\frac{\Breve{c}_{0}}{\varepsilon^{2}}\,
\bigl(\varepsilon^{\nu}+ \abs{x}\bigr)\,,
\end{align*}
where in the inequality we use the identity
$\widehat{V}_{z}^{\varepsilon} =\varepsilon^{2(\nu-1)}\Breve{V}_{z}^{\varepsilon}$
and \eqref{EL4.4A}.
Since $\widetilde{V}^{\varepsilon}\;=\; {\varepsilon}^{2}V^{\varepsilon}$,
this proves the property for $\widetilde{V}^{\varepsilon}$.
This completes the proof.
\end{proof}

We continue with a version of Lemma~\ref{L4.1} for unbounded domains.

%%%%%%%%%%%%%%%%%%%%%%%%%%%%%%%%%%%%%%%%%%%%%%%%%%%%%%%%%%%%%%%%%%%%%%%%%%%%%%%%
\begin{proposition}\label{P4.5}
Let $\nu\in(0,2]$.  Then for any $k\in\NN$ and $r>0$,
there exist constants
and $\Hat{\kappa}_{1}=\Hat{\kappa}_{1}(k)$ and
$\Hat{\kappa}_{2}=\Hat{\kappa}_{1}(k)$ such that
with $\Hat{r}(\varepsilon) \df \Hat{\kappa}_{2}\varepsilon^{\nu\wedge 1}$
we have
\begin{equation*}
\int_{B_{\Hat{r}(\varepsilon)}^{c}(\cS)}
\bigl(\dist(x,\cS)\bigr)^{2k}\, \eta_{*}^{\varepsilon}(\D{x})
\;\le\; \Hat{\kappa}_{1}\, \varepsilon^{2(\nu\wedge 1)}
\qquad\forall\,\varepsilon\in(0,1)\,.
\end{equation*}
\end{proposition}

\begin{proof}
Let $\widetilde{V}^{\varepsilon}\df {\varepsilon}^{2}V^{\varepsilon}$.
Since $V^{\varepsilon}(0)=0$, by Lemma~\ref{L4.4}
the function $\widetilde{V}^{\varepsilon}={\varepsilon}^{2}V^{\varepsilon}$
is locally bounded, uniformly in $\varepsilon>0$.
Applying the operator
\begin{equation*}
\Lg^{\varepsilon}_{*}\;\df\;\frac{\varepsilon^{2\nu}}{2}\Delta
+ \langle m-\varepsilon^{2}\grad V^{\varepsilon},\nabla\rangle
\end{equation*}
to the function $\Lyap^{2k}\,\E^{\widetilde{V}^{\varepsilon}}$
and using the identities
$\Lg^{\varepsilon}_{*}\bigl[\widetilde{V}^{\varepsilon}\bigr]
= \varepsilon^{2} \bigl(\beta^{\varepsilon}_{*}-\ell\bigr)
- \frac{1}{2}\,\abs{\grad\widetilde{V}^{\varepsilon}}^{2}$,
and rearranging terms we obtain
\begin{align}\label{EP4.5A}
\Lg^{\varepsilon}_{*} \bigl[\Lyap^{2k}\,\E^{\widetilde{V}^{\varepsilon}}\bigr]
&\;=\; \Lyap^{2k}\,\Lg^{\varepsilon}_{*} \bigl[\E^{\widetilde{V}^{\varepsilon}}\bigr]
+ \E^{\widetilde{V}^{\varepsilon}}\Lg^{\varepsilon}_{*}
\bigl[\Lyap^{2k}\bigr]
+2k\,\varepsilon^{2\nu}\,\Lyap^{(2k-1)}\,\E^{\widetilde{V}^{\varepsilon}}
\langle\grad\widetilde V^\varepsilon, \grad\Lyap\rangle
\\[5pt]
&\;=\; \Lyap^{2k} \E^{\widetilde{V}^{\varepsilon}}
\bigl( \varepsilon^2(\beta^{\varepsilon}_{*}-\ell)
+ \frac{\varepsilon^{2\nu}}{2}\abs{\grad\widetilde{V}^\varepsilon}^2\bigr)
+2k\,\varepsilon^{2\nu}\,\Lyap^{(2k-1)}\,\E^{\widetilde{V}^{\varepsilon}}
\langle\grad\widetilde V^\varepsilon, \grad\Lyap\rangle\nonumber\\[5pt]
&\mspace{50mu}
+ \E^{\widetilde{V}^{\varepsilon}}\biggl( 2k\varepsilon^{2\nu}\Lyap^{2k-1} \Delta \Lyap 
+
k(2k-1)\varepsilon^{2\nu}\Lyap^{2k-2}\abs{\grad\Lyap}^2 + \langle m-\varepsilon^2\grad V^\varepsilon, 
\grad\Lyap\rangle\biggr)\nonumber\\[5pt]
&\;=\; \Lyap^{2k}\,\E^{\widetilde{V}^{\varepsilon}}
\biggl[\varepsilon^{2} \bigl(\beta^{\varepsilon}_{*}-\ell\bigr)
+k\varepsilon^{2\nu}\frac{\Delta\Lyap}{\Lyap}
-\frac{1-\varepsilon^{2\nu}}{2}\abs{\grad \widetilde V^\varepsilon}^2- 
 \frac{2k(1-\varepsilon^{2\nu})}{\Lyap}
\langle\grad\widetilde V^\varepsilon, \grad\Lyap\rangle\nonumber\\[6pt]
&\mspace{330mu}+ 2k\,\frac{\langle m, \grad\Lyap\rangle}{\Lyap}
+k(2k-1)\varepsilon^{2\nu}
\,\frac{\abs{\grad\Lyap}^{2}}{\Lyap^{2}}\biggr]\,\nonumber\\[5pt]
&\;=\;
\Lyap^{2k}\,\E^{\widetilde{V}^{\varepsilon}}
\biggl[\varepsilon^{2} \bigl(\beta^{\varepsilon}_{*}-\ell\bigr)
+k\varepsilon^{2\nu}\frac{\Delta\Lyap}{\Lyap}
-\frac{1-\varepsilon^{2\nu}}{2}\biggl(\grad\widetilde{V}^{\varepsilon}+
2k\,
\frac{\grad\Lyap}{\Lyap}\biggr)^{2}\nonumber\\[6pt]
&\mspace{330mu}+ 2k\,\frac{\langle m, \grad\Lyap\rangle}{\Lyap}
+k\bigl(2k-\varepsilon^{2\nu}\bigr)
\,\frac{\abs{\grad\Lyap}^{2}}{\Lyap^{2}}\biggr]\,.\nonumber
\end{align}
By \eqref{EL2.3}, and since $\Bar\Lyap$ has
strict quadratic growth and $\grad\Bar\Lyap$ is Lipschitz by Hypothesis~\ref{H1.1},
and $\Lyap$ agrees with $\Bar\Lyap$ outside a compact set,
it follows that
$\frac{\abs{\grad\Lyap}^2}{\Lyap}$ is bounded on $\Rd$.
Therefore, in view of the bounds in \eqref{ET2.2} and \eqref{EL2.3},
 we can add a positive constant to $\Lyap$ so that
\begin{equation}\label{EP4.5B}
2\,\frac{\langle m, \grad\Lyap\rangle}{\Lyap}
+ \bigl(2k-\varepsilon^{2\nu}\bigr)\, \frac{\abs{\grad\Lyap}^{2}}{\Lyap^{2}}
\;\le\; \frac{\langle m, \grad\Lyap\rangle}{\Lyap}\qquad\text{on\ }\Rd\,,
\quad\forall\,\varepsilon>0\,.
\end{equation}
The constant is selected so that $\Lyap\ge1$ on $\Rd$.
Define
\begin{equation*}
G^{\varepsilon}_{0}\;\df\; \varepsilon^{2-2\wedge2\nu}(\beta^\varepsilon_*-\ell)
-\frac{1-\varepsilon^{2\nu}}{2\varepsilon^{2\wedge2\nu}}
\babss{\grad\widetilde{V}^{\varepsilon}
+2k\,\frac{\grad\Lyap}{\Lyap}}^{2}\,.
\end{equation*}
Since $\ell$ is inf-compact, there exists $r_{0}>0$ such that
$G^{\varepsilon}_{0}\le0$ on $B_{r_{0}}^{c}$. We may choose $r_0$ large enough so that
$\cS\subset B_{r_{0}}$.
Let $\kappa_0$ be a bound of $\beta^\varepsilon_* -\ell$
on $B_{r_{0}}$.
Using this bound and \eqref{EP4.5A}--\eqref{EP4.5B}, we obtain
\begin{equation}\label{EP4.5C}
\frac{1}{\varepsilon^{2\wedge 2\nu}}\,
\Lg^{\varepsilon}_{*}\bigl[\Lyap^{2k}\,\E^{\widetilde{V}^{\varepsilon}}\bigr](x)
\le\;\Lyap^{2k}(x)\,\E^{\widetilde{V}^{\varepsilon}(x)}
\biggl[\kappa_0\Ind_{B_{r_{0}}}(x) +\frac{k}{\varepsilon^{2\wedge2\nu}}\,
\frac{\varepsilon^{2\nu}\Delta\Lyap(x)
+ \bigl\langle m(x),\grad\Lyap(x)\bigr\rangle}
{\Lyap(x)}\biggr]
\end{equation}
for all $x\in\Rd$, and all $\varepsilon\in(0,1)$.
By \eqref{ET2.2} we have
\begin{equation}\label{EP4.5D}
\varepsilon^{2\nu}\Delta\Lyap(x)+ \bigl\langle m(x),\grad\Lyap(x)\bigr\rangle
\;\le\; \frac{1}{2} \bigl\langle m(x),\grad\Lyap(x)\bigr\rangle
\end{equation}
for all $x\in\Rd$ such that
$\dist(x,\cS)\ge \kappa_1\varepsilon^\nu$,
with $\kappa_1\df\sqrt{2C_0^{-1}\norm{\Delta\Lyap}_\infty}$ .
Using \eqref{ET2.2} once more, if we define
$\kappa_2\df \bigl(4 k^{-1} C_0^{-1} \kappa_0\,
\sup_{B_{r_0}}\Lyap\bigr)^{\nicefrac{1}{2}}$,
then we have
\begin{equation}\label{EP4.5E}
\varepsilon^{2\wedge2\nu}\kappa_0
+ \frac{k \bigl\langle m(x),\grad\Lyap(x)\bigr\rangle}
{4\Lyap(x)} \;\le\; 0
\end{equation}
in $\bigl\{x\in B_{r_0}\,\colon
\dist(x,\cS)\ge \kappa_2 \varepsilon^{1\wedge \nu}\bigr\}$.
Combining \eqref{EP4.5C}, \eqref{EP4.5D}, and \eqref{EP4.5E}, we obtain
\begin{equation}\label{EP4.5F}
\frac{1}{\varepsilon^{2\wedge 2\nu}}\,
\Lg^{\varepsilon}_{*}\bigl[\Lyap^{2k}\,\E^{\widetilde{V}^{\varepsilon}}\bigr](x)
\;\le\; \frac{k}{4\varepsilon^{2\wedge 2\nu}}\,
\Lyap^{2k-1}(x)\,\E^{\widetilde{V}^{\varepsilon}(x)}
\bigl\langle m(x),\grad\Lyap(x)\bigr\rangle
\end{equation}
for all $x\in\Rd$ such that
$\dist(x,\cS)\ge \Hat{r}(\varepsilon)\df
(\kappa_1\vee\kappa_2) \varepsilon^{\nu\wedge 1}$.
Let $\kappa_3$ be a bound of the right hand side of
\eqref{EP4.5C} on $B_{\Hat{r}(\varepsilon)}(\cS)$.
This bound does not depend on $\varepsilon$, since $\widetilde{V}^{\varepsilon}$
is locally bounded, uniformly in $\varepsilon\in(0,1)$.
Then, by \eqref{EP4.5C} and \eqref{EP4.5F} we obtain
\begin{equation}\label{EP4.5G}
\frac{1}{\varepsilon^{2\wedge 2\nu}}\,
\Lg^{\varepsilon}_{*}\bigl[\Lyap^{2k}\,\E^{\widetilde{V}^{\varepsilon}}\bigr](x)
\;\le\;\kappa_3 
+\frac{k}{4\varepsilon^{2\wedge 2\nu}}\,
\bigl\langle m(x),\grad\Lyap(x)\bigr\rangle
\Lyap^{2k-1}(x)\,\E^{\widetilde{V}^{\varepsilon}(x)}\,
\Ind_{B_{\Hat{r}(\varepsilon)}^c(\cS)}(x)
\end{equation}
for all $x\in\Rd$, and $\varepsilon\in(0,1)$.

By the strong maximum principle, $V^{\varepsilon}$ attains
its infimum in $\Rd$ in the set
$\{x\in\Rd\,\colon \ell(x)\le\beta^{\varepsilon}_{*}\}$.
Therefore, $\widetilde{V}^{\varepsilon}$ is bounded below in $\Rd$,
uniformly in $\varepsilon$,
by Lemma~\ref{L4.4}.
Thus, from \eqref{EP4.5G} we obtain
\begin{equation}\label{EP4.5H}
\int_{B^{c}_{\Hat{r}(\varepsilon)}(\cS)}
\frac{\babs{\bigl\langle m(x),\grad\Lyap(x)\bigr\rangle}}{\varepsilon^{2\wedge 2\nu}}\,
\Lyap^{2k-1}(x)\,\eta_{*}^{\varepsilon}(\D{x})
\;\le\; \frac{4\kappa_{3}}{k(\inf_{\Rd}\E^{\widetilde{V}^{\varepsilon}})}
\end{equation}
for all $\varepsilon<\varepsilon_{0}$.
By the strict quadratic growth of $\Lyap$ mentioned earlier, together
with \eqref{EL2.3} and \eqref{EP4.5H}, there exists
a constant $\kappa_{4}$, such that
\begin{equation*}
\int_{B^{c}_{\Hat{r}(\varepsilon)}(\cS)}
\frac{1}{\varepsilon^{2\wedge 2\nu}}\,\bigl(\dist(x,\cS)\bigr)^{4k-1}\,
\eta_{*}^{\varepsilon}(\D{x})
\;\le\;\kappa_{4}\qquad\forall\varepsilon\in(0,1)\,.
\end{equation*}
This finishes the proof.
\end{proof}

%%%%%%%%%%%%%%%%%%%%%%%%%%%%%%%%%%%%%%%%%%%%%%%%%%%%%%%%%%%%%%%%%%%%%%%%%%%%%%%%
\begin{corollary}\label{C4.6}
Let $D$ be any open set such that $\cSs\subset D$. The following hold.
\begin{itemize}
\item[\textup{(}a\/\textup{)}]
If $\fJ=\fJs$, then
$\eta^{\varepsilon}_*(D^c)\in\order\bigl(\varepsilon^{2-\nu}\bigr)$
for all $\nu\in(1,2)$.

\smallskip
\item[\textup{(}b\/\textup{)}]
If $\nu\in(0,1)$ then
\begin{equation}\label{EC4.6A}
\sG^\varepsilon_*\;\in\;\order\bigl(\varepsilon^{\nu}\bigr)\,,~
\beta^{\varepsilon}_{*}-\fJ\;\ge\;\order\bigl(\varepsilon^{\nu}\bigr)\,,~
\text{and\ }
\eta^{\varepsilon}_*(D^c)\;\in\;
\order\bigl(\varepsilon^{2\nu\wedge(2-\nu)}\bigr)\, .
\end{equation}
\end{itemize}
\end{corollary}

\begin{proof}
Since $2-\nu< 2(\nu\wedge1)$ for $\nu\in[1,2)$,
then, in view of Proposition~\ref{P4.5}, it suffices to prove that
$\eta^{\varepsilon}_*(\cN)\in
\order\bigl(\varepsilon^{2-\nu}\bigr)$
for  a bounded open neighborhood $\cN$ of $z\in\cS\setminus\cSs$.
Let $\varphi$ be as in the proof of Lemma~\ref{L3.6}.
By Proposition~\ref{P4.5}, we have
\begin{equation*}
\xi^\varepsilon_2\in\order\bigl(\varepsilon^{2(\nu\wedge 1)}\bigr)\,,\quad
\text{and}\quad\int_{B^{c}_{r}(\cS)}\varphi'(\Lyap)\Delta\Lyap\,\D\eta_{*}^{\varepsilon}
\in\order\bigl(\varepsilon^{2(\nu\wedge 1)}\bigr)\,.
\end{equation*}
Thus
\begin{equation}\label{EC4.6B}
\xi^\varepsilon \;\le\; \frac{1}{2}\Delta\Lyap(z)\,
\eta^{\varepsilon}_*\bigl(B_r(z)\bigr)
+\order\bigl(\varepsilon^{2(\nu\wedge 1)}\bigr)
\end{equation}
by \eqref{EL3.6A} and \eqref{EL3.6F}.
In addition, we have $\sG^\varepsilon_*\in\order\bigl(\varepsilon^{\nu\wedge2}\bigr)$ by
Corollary~\ref{C4.2}\,(a),
and
$-C_0\,\xi^\varepsilon\le\frac{1}{2}\varepsilon^{2-2\nu}\sG^\varepsilon_*$
by \eqref{E-disc}.
We combine these with \eqref{EC4.6B} for $\nu\in(1,2)$ to obtain
\begin{equation*}%\label{E-discbis}
-C_0\Delta\Lyap(z)\,
\eta^{\varepsilon}_*\bigl(B_r(z)\bigr)
+\order\bigl(\varepsilon^{2}\bigr)
\;\le\;  \varepsilon^{2-2\nu}\,\sG^\varepsilon_*
 \;\in\; \order\bigl(\varepsilon^{2-\nu}\bigr)\,.
\end{equation*}
Thus
$\eta^{\varepsilon}_*\bigl(B_r(z)\bigr)\in\order\bigl(\varepsilon^{2-\nu}\bigr)$
for $\nu\in(1,2)$.
This completes the proof of part~(a).

The proof of part~(b) is divided in two steps.
\noindent
\paragraph{Step 1}  Suppose $\fJ=\fJs$.
Then \eqref{EC4.2E}--\eqref{EC4.2H} hold with $\fJ$
replaced by $\fJs$.
By Lemma~\ref{L3.5}\,(ii) we have
$\beta^\varepsilon_*-\fJs\le\order\bigl(\varepsilon^{\nu\vee(4\nu-2)}\bigr)$.
Therefore $\sG^\varepsilon_*\in\order\bigl(\varepsilon^{\nu}\bigr)$ by
\eqref{EC4.2H},
and thus $\Tilde\zeta^\varepsilon\in\order\bigl(\varepsilon^\nu\bigr)$
by \eqref{EC4.2G}.
Hence, $\beta^\varepsilon_*-\fJs\ge\order\bigl(\varepsilon^{\nu}\bigr)$
by \eqref{EC4.2E}.
The estimate $\eta^{\varepsilon}_*(D^c)\in
\order\bigl(\varepsilon^{2\nu\wedge(2-\nu)}\bigr)$ is obtained
exactly as in Corollary~\ref{C4.6}\,(a).

\noindent
\paragraph{Step 2}
Suppose $\fJ<\fJs$.
By Theorem~\ref{T2.2}\,(ii), we may construct $\Lyap$ such that
$\Lyap(z)> 5\,\max_{\cSs}\,\Lyap$ for all $z\in\cS\setminus\cSs$.
Let $G=\bigl\{x\in\Rd\,\colon \Lyap(x)<2\,\max_{\cSs}\,\Lyap\bigr\}$
and $\Tilde\varphi$ be as in the proof of Lemma~\ref{L4.1}, with
$\delta= 2\,\max_{\cSs}\,\Lyap$.
We have
\begin{equation*}
\fJs-\ell(x)\;\le\; \ell(z) - \ell(x) \;\le\; C_\ell\,\abs{x-z}\qquad
\forall\,z\in\cSs\,,\text{\ \ and\ \ } x\in\Rd\,.
\end{equation*}
Thus
\begin{equation*}
\ell(x) -\fJs\;\ge\; \max_{z\in\cSs}\; \bigl\{- C_\ell\,\abs{x-z}\bigr\}
\;=\; - C_\ell\,\dist(x,\cSs)\qquad\forall\,x\in\Rd\,.
\end{equation*}
Also by Proposition~\ref{P4.5}, for some positive constants
$r$ and $\kappa_1$ we obtain
\begin{equation*}
\int_{G^c}(\ell(x)-\fJs)\D\eta_{*}^{\varepsilon}
\;\ge\; -\kappa_1 \sum_{z\in\cS\setminus\cSs} \eta^\varepsilon_*\bigl(B_r(z)\bigr)
+\order\bigl(\varepsilon^{2\nu}\bigr)\,.
\end{equation*}
Therefore, splitting the integral over $G$ and $G^c$, we obtain
as in \eqref{EC4.2D} that
\begin{equation*}
\int_{\Rd}\,\ell\,\D\eta_{*}^{\varepsilon}-\fJs
\;\ge\; -\kappa_1 \sum_{z\in\cS\setminus\cSs} \eta^\varepsilon_*\bigl(B_r(z)\bigr)
+\order\bigl(\varepsilon^{2\nu}\bigr)
-\frac{C_\ell}{\sqrt{C_0}}\,\Tilde\zeta^{\varepsilon}\,,
\end{equation*}
and since $\Tilde\zeta^{\varepsilon}\in\order\bigl(\varepsilon^{\nu}\bigr)$,
following the steps in \eqref{EC4.2E}--\eqref{EC4.2H} we have
\begin{equation}\label{EC4.6C}
-\kappa_1 \sum_{z\in\cS\setminus\cSs} \eta^\varepsilon_*\bigl(B_r(z)\bigr)
-\order\bigl(\varepsilon^{2\nu}\bigr)
-\frac{C_\ell}{\sqrt{C_0}}\,\Tilde\zeta^{\varepsilon}\;\le\;
\beta^\varepsilon_*-\fJs\,,
\end{equation}
and
\begin{equation}\label{EC4.6D}
\tfrac{1}{2}\,\sG^\varepsilon_*
\;\le\; \beta^{\varepsilon}_{*}-\fJs
+ \frac{C_\ell}{C_0^3}\,\varepsilon^2
+ \frac{\sqrt{2}\,C_\ell}{C_0}\varepsilon^{\nu}\,\sqrt{\abs{\Tilde\xi^{\varepsilon}}}
+\kappa_1 \sum_{z\in\cS\setminus\cSs} \eta^\varepsilon_*\bigl(B_r(z)\bigr)
+\order\bigl(\varepsilon^{2\nu}\bigr)\,.
\end{equation}
In view of \eqref{E-disc2} and \eqref{EC4.6B} we have
\begin{equation}\label{EC4.6E}
\sum_{z\in\cS\setminus\cSs} \eta^\varepsilon_*\bigl(B_r(z)\bigr)
\;\le\; \kappa_2(\varepsilon^{2-2\nu}\sG^{\varepsilon}_{*} + \varepsilon^{2\nu})
\end{equation}
for some positive constant $\kappa_2$.
Since $\beta^\varepsilon_*-\fJs\le\order\bigl(\varepsilon^{\nu\vee(4\nu-2)}\bigr)$
by Lemma~\ref{L3.6}, and $\nu<1$,
combining \eqref{EC4.6D} and \eqref{EC4.6E} we obtain
$\sG^{\varepsilon}_{*}\in\order(\varepsilon^{\nu})$.
Therefore by \eqref{EC4.6E}, we obtain
$\eta^{\varepsilon}_*\bigl(B_r(z)\bigr)\in
\order\bigl(\varepsilon^{2\nu\wedge(2-\nu)}\bigr)$
for all $z\in\cS\setminus\cSs$.
In turn, $\beta^{\varepsilon}_{*}-\fJ\ge\order\bigl(\varepsilon^{\nu}\bigr)$
by \eqref{EC4.6C}.
This completes the proof.
\end{proof}

%%%%%%%%%%%%%%%%%%%%%%%%%%%%%%%%%%%%%%%%%%%%%%%%%%%%%%%%%%%%%%%%%%%%%%%%%%%%%%%%
\begin{remark}\label{R4.7}
If $\nu=1$ and $\fJc=\fJs$, then following the argument in
Step~2 of the proof of Corollary~\ref{C4.6} we obtain the
same estimates as in \eqref{EC4.6A}.
In this case we don't estimate $\sG^\varepsilon_*$ from \eqref{EC4.6D},
but rather use Corollary~\ref{C4.2}\,(a) which asserts
that $\sG^\varepsilon_*\in\order(\varepsilon)$.
Thus $\eta^\varepsilon_*\bigl(B_r(z)\bigr)\in\order(\varepsilon)$
by \eqref{EC4.6E}, which, in turn, implies that
$\beta^{\varepsilon}_{*}-\fJs\ge\order(\varepsilon)$
by \eqref{EC4.6C}.
\end{remark}

%%%%%%%%%%%%%%%%%%%%%%%%%%%%%%%%%%%%%%%%%%%%%%%%%%%%%%%%%%%%%%%%%%%%%%%%%%%%%%%%
\section{Convergence of the scaled optimal stationary distributions}\label{S5}

We need the following definition.

%%%%%%%%%%%%%%%%%%%%%%%%%%%%%%%%%%%%%%%%%%%%%%%%%%%%%%%%%%%%%%%%%%%%%%%%%%%%%%%%
\begin{definition}\label{D5.1}
For the rest of the paper $\{\sB_{z}\,\colon z\in\cS\}$ is
some collection of nonempty, disjoint
balls, with each $\sB_{z}$ centered around $z$, and we define
$\sB_{\cS}\,\df\,\cup_{z\in\cS}\,\sB_{z}$.

Recall $\widehat{V}_{z}^{\varepsilon}$ from Definition~\ref{D4.3}.
For $z\in\cS$, we define the `scaled' density $\Hat{\varrho}_{z}^{\varepsilon}(x)\df
\varepsilon^{\nu d}\varrho_{*}^{\varepsilon}(\varepsilon^{\nu} x+z)$,
and denote by $\Hat{\eta}_{z}^{\varepsilon}$ the corresponding probability
measure in $\Rd$.
We also define the `normalized'
probability density $\mathring\varrho^\varepsilon_{z}$ supported
on $\eta_{*}^{\varepsilon}(\sB_{z})$ by
\begin{equation*}
\mathring\varrho^\varepsilon_{z}(x)\;\df\;
\begin{cases}
\frac{\Hat{\varrho}_{z}^{\varepsilon}(x)}{\eta_{*}^{\varepsilon}(\sB_{z})}
&\text{if\ }\varepsilon^{\nu}x+z\in
\sB_{z}\,,\\[5pt]
0&\text{otherwise,}
\end{cases}
\end{equation*}
and let
$\mathring{\eta}_{z}^{\varepsilon}(\D{x})
=\mathring\varrho^\varepsilon_{z}(x)\,\D{x}$.%
\nomenclature[Cf]{$\Hat{\varrho}_{z}^{\varepsilon}$,
$\mathring\varrho^\varepsilon_{z}$}{scaled optimal densities, Definition~\ref{D5.1}}%
\nomenclature[Cf]{$\Hat{\eta}_{z}^{\varepsilon}$,
$\mathring{\eta}_{z}^{\varepsilon}$}{scaled optimal stationary
distributions, Definition~\ref{D5.1}}%
\end{definition}

Section~\ref{S5.1} which follows concerns the critical regime.
The subcritical and supercritical regimes are treated
in Section~\ref{S5.2}.

%%%%%%%%%%%%%%%%%%%%%%%%%%%%%%%%%%%%%%%%%%%%%%%%%%%%%%%%%%%%%%%%%%%%%%%%%%%%%%%%
\subsection{Convergence to a Gaussian in the critical regime}\label{S5.1}

Recall the notation in Definitions~\ref{D1.9} and \ref{D1.10}.
Also the scaled quantities in Definition~\ref{D4.3}.
We start with the following lemma.

%%%%%%%%%%%%%%%%%%%%%%%%%%%%%%%%%%%%%%%%%%%%%%%%%%%%%%%%%%%%%%%%%%%%%%%%%%%%%%%%
\begin{lemma}\label{L5.2}
Assume $\nu=1$. Fix any $z\in\cS$.
Then every sequence $\varepsilon_n\searrow0$ has a subsequence along which
$\widehat{V}_{z}^{\varepsilon}(x)-\widehat{V}_{z}^{\varepsilon}(z)$
converges to some $\Bar{V}_z\in\Cc^{2}(\Rd)$ uniformly on compact
subsets of $\Rd$, and $\beta^\varepsilon_*$ converges to some constant
$\Bar{\beta}_*$, and these satisfy
\begin{equation}\label{EL5.2A}
\tfrac{1}{2}\Delta \Bar{V}_z(x)
+ \bigl\langle M_z\,x, \grad \Bar{V}_z(x)\bigr\rangle
-\tfrac{1}{2}\abs{\grad \Bar{V}_z(x)}^{2} \;=\; \Bar{\beta}_*-\ell(z)\,.
\end{equation}
Moreover, for some constant $\Hat{c}_0$ we have
\begin{equation}\label{EL5.2B}
\abs{\grad \Bar{V}_z(x)}\;\le\; \Hat{c}_{0}\,(1+\abs{x})
\qquad\forall\,\varepsilon\in(0,1)\,,\quad\forall\,x\in\Rd\,,
\end{equation}
and
\begin{equation}\label{EL5.2C}
\Bar{\beta}_*\;\le\;\varLambda^+(M_z)+\ell(z)\,.
\end{equation}
\end{lemma}

\begin{proof}
If $\nu=1$, then by \eqref{EL4.4B} we obtain
\begin{equation}\label{EL5.2D}
\frac{1}{2}\Delta \widehat{V}_{z}^{\varepsilon}
+\bigl\langle\widehat{m}_{z}^{\varepsilon}, \grad \widehat{V}_{z}^{\varepsilon}
\bigr\rangle
-\frac{1}{2}\abs{\grad \widehat{V}_{z}^{\varepsilon}}^{2}
+\widehat{\ell}_{z}^{\varepsilon} \;=\; \beta^\varepsilon_*\,.
\end{equation}
By applying
\cite[Lemma~5.1]{Met} to \eqref{EL5.2D}
and using the assumptions on the growth of $m$ and $\ell$,
it follows that
there exists a constant $\Hat{c}_{0}$ such that
\begin{equation}\label{EL5.2E}
\abs{\grad \widehat{V}_{z}^{\varepsilon}(x)}\;\le\; \Hat{c}_{0}\,(1+\abs{x})
\qquad\forall\,\varepsilon\in(0,1)\,,\quad\forall\,x\in\Rd\,.
\end{equation}
It follows by \eqref{EL5.2D} and the bound in \eqref{EL5.2E}
that $\widehat{V}_{z}^{\varepsilon}$ is locally bounded in
$\Cc^{2,\alpha}(\Rd)$, for any $\alpha\in(0,1)$.
It is also clear that $\widehat{m}_{z}^{\varepsilon}(x)\to M_z\,x$
and $\widehat{\ell}_{z}^{\varepsilon}(x)\to \ell(z)$, as $\varepsilon\searrow0$,
uniformly over compact sets.
Thus, taking limits in \eqref{EL5.2D} along some sequence
$\varepsilon_{n}\searrow0$ we obtain a function $\Bar{V}_z\in\Cc^{2}(\Rd)$
and a constant $\Bar{\beta}_*$ which satisfy \eqref{EL5.2A}.
The bound in \eqref{EL5.2B} follows by \eqref{EL5.2E}, while
the bound in \eqref{EL5.2C} follows by applying Theorem~\ref{T1.18}\,(c)
to \eqref{EL5.2A} with $\Bar\beta= \Bar{\beta}_*-\ell(z)$.
\end{proof}

We fix some notation.
The function $\Bar{V}_z$ for $z\in\cS$ denotes the limit obtained
in Lemma~\ref{L5.2}.
The associated  `diffusion limit', takes the form
\begin{equation}\label{ES5.1A}
\D \Bar{X}_{t}\;=\;\bigl(M_z\,\Bar{X}_{t}-\grad \Bar{V}_z(\Bar{X}_{t})\bigr)\,\D{t}
+\D\Bar{W}_{t}\,,
\end{equation}
and its extended generator is denoted by
\begin{equation}\label{ES5.1B}
\Bar{\Lg}_z\;\df\;\frac{1}{2}\Delta+
\bigl\langle M_z\,x-\grad \Bar{V}_z(x), \grad\bigr\rangle\,.
\end{equation}

Since \eqref{EL5.2C} holds for all $z\in\cS$, then we must have
$\Bar{\beta}_*\le\fJc$, and Lemma~\ref{L5.2} provides
an alternate proof of the upper bound
$\limsup_{\varepsilon\searrow0}\,\beta^\varepsilon_*\le\fJc$, which was
already shown in Lemma~\ref{L3.3}.
In the next theorem we show that if
$\liminf_{\varepsilon_n\searrow0}\,\eta_{*}^{\varepsilon_{n}}(\sB_{z})>0$,
over some sequence $\{\varepsilon_n\}$, then
the diffusion in \eqref{ES5.1A} is positive recurrent.

%%%%%%%%%%%%%%%%%%%%%%%%%%%%%%%%%%%%%%%%%%%%%%%%%%%%%%%%%%%%%%%%%%%%%%%%%%%%%%%%
\begin{theorem}\label{T5.3}
Assume $\nu=1$, and let $\{\sB_{z}\,\colon z\in\cS\}$ be as in
Definition~\ref{D5.1}.
Let $\varepsilon_{n}\searrow0$ be any sequence
satisfying  $\liminf_{n\to\infty}\,\eta_{*}^{\varepsilon_{n}}(\sB_{z})= \theta_z>0$
for some $z\in\cS$,
and $(\Bar{V}_z,\Bar{\beta}_*)\in\Cc^2(\Rd)\times\RR$ be any limit point of
$\bigl(\widehat{V}_{z}^{\varepsilon}(x)-\widehat{V}_{z}^{\varepsilon}(z),
\beta^\varepsilon_*\bigr)$ along some subsequence of $\{\varepsilon_{n}\}$
\textup{(}see Lemma~\ref{L5.2}\,\textup{)}.
Recall Definition~\ref{D1.9}.
Then
\begin{itemize}
\item[\textup{(}a\/\textup{)}]
The diffusion in \eqref{ES5.1A}
is positive recurrent with invariant probability measure $\Bar\eta_z$,
and the density
$\mathring\varrho^\varepsilon_{z}$ in Definition~\ref{D5.1} converges to
the density $\Bar\varrho_z$ of $\Bar\eta_z$, uniformly on compact subsets of $\Rd$.
\smallskip
\item[\textup{(}b\/\textup{)}]
The invariant probability measure $\Bar\eta_z$ has finite second moments.
\smallskip
\item[\textup{(}c\/\textup{)}]
It holds that $\Bar{\beta}_*=\ell(z)+\varLambda^+(M_z)$.
\smallskip
\item[\textup{(}d\/\textup{)}] We have
\begin{equation}\label{ET5.3A}
\widehat{V}_{z}(x)\;=\;\tfrac{1}{2}\, \bigl\langle x,\widehat{Q}_z x
\bigr\rangle\,,
\end{equation} 
and that $\Bar\varrho_z$ is the density of a Gaussian with mean $0$
and covariance matrix $\widehat\Sigma_z$.
Here $(\widehat{Q}_{z},\widehat{\Sigma}_{z})$ are the pair
of matrices which solve \eqref{ED1.9}.
\item[\textup{(}e\/\textup{)}]
It holds that
\begin{equation*}
\liminf_{\varepsilon_{n}\searrow0}\; \int_{\sB_z}
\Bigl(\ell(x) +\tfrac{1}{2}\abs{v_{*}^{\varepsilon_{n}}(x)}^{2}\Bigr)\,
\eta_{*}^{\varepsilon_{n}}(\D{x}) \;\ge\; \theta_{z}\,
\bigl(\ell(z)+\varLambda^+(M_z)\bigr)\,.
\end{equation*}
\end{itemize}
\end{theorem}

\begin{proof}
In order to show that the diffusion in \eqref{ES5.1A} is positive recurrent,
we examine the scaled diffusion
\begin{equation}\label{ET5.3E}
\D{X}_{t} = \bigl(\widehat{m}_{z}^{\varepsilon}(X_{t})
-\grad \widehat{V}_{z}^{\varepsilon}(X_{t})\bigr)\,\D{t} + \D{W}_{t}\,.
\end{equation}
Recall from Definition~\ref{D5.1} that $\Hat{\eta}_{z}^{\varepsilon}$
and $\Hat{\varrho}_{z}^{\varepsilon}$ denote the
invariant probability measure of \eqref{ET5.3E} and its density, respectively.
Let 
\begin{equation*}
\widehat{\Lg}_{z}^{\varepsilon}\;\df\;\tfrac{1}{2}\Delta+
\bigl\langle\widehat{m}_{z}^{\varepsilon}
-\grad \widehat{V}_{z}^{\varepsilon}, \grad\bigr\rangle
\end{equation*}
denote the extended generator of \eqref{ET5.3E}.
It follows by Lemma~\ref{L4.1} and the Markov inequality that
$\eta_{*}^{\varepsilon}\bigl(\sB_{z}\setminus B_{n\varepsilon}(z)\bigr)
\;\le\; \frac{\Hat{\kappa}_{0}}{n^{2}}$ for all $n\in\NN$.
Hence,
$\{\mathring{\eta}_{z}^{\varepsilon_{n}}\,\colon n\in\NN\}$
is a tight family of measures.
By the Harnack inequality the family
$\{\Hat{\varrho}_{z}^{\varepsilon_{n}}\,\colon n\in\NN\}$ is locally
bounded, and locally H\"older equicontinuous,
and the same of course applies to
$\{\mathring{\varrho}_{z}^{\varepsilon_{n}}\,\colon n\in\NN\}$.
Moreover,
the tightness of $\{\mathring{\eta}_{z}^{\varepsilon_{n}}\,\colon n\in\NN\}$
implies the uniform integrability of
$\{\mathring{\varrho}_{z}^{\varepsilon_{n}}\,\colon n\in\NN\}$.
Select any subsequence, also denoted by $\{\varepsilon_{n}\}$ along which
$\mathring{\varrho}_{z}^{\varepsilon_{n}}$ converges locally uniformly,
and denote the limit by $\Bar\varrho_z$.
By uniform integrability, $\mathring{\varrho}_{z}^{\varepsilon_{n}}$
also converges in $L^{1}(\Rd)$, as $n\to\infty$, and hence
$\int_{\Rd}\Bar\varrho_z(x)\,\D{x}=1$.
Therefore $\Bar\eta_z(\D{x})\df\Bar\varrho_z(x)\,\D{x}$
is a probability measure.
Let $f$ be a smooth function with compact support,
and $\Bar\Lg_z$ be as in \eqref{ES5.1B}.
Then
\begin{align}\label{ET5.3G}
\babss{\int_{\Rd}\widehat{\Lg}_{z}^{\varepsilon_{n}}f(x)
\mathring{\varrho}_{z}^{\varepsilon_{n}}(x)\,
\D{x} - \int_{\Rd}\Bar\Lg_{z}f(x)\Bar\varrho_z(x)\,\D{x}}
&\;\le\;
\babss{\int_{\Rd}\widehat{\Lg}_{z}^{\varepsilon_{n}}f(x)
\bigl(\mathring{\varrho}_{z}^{\varepsilon_{n}}(x)-
\Bar\varrho_z(x)\bigr)\,\D{x}}\\[5pt]
&\mspace{50mu}+ \babss{\int_{\Rd}\bigl(\widehat{\Lg}_{z}^{\varepsilon_{n}}f(x)
-\Bar\Lg_{z}f(x)\bigr)
\Bar\varrho_z(x)\,\D{x}}\,.\nonumber
\end{align}
Since $\mathring{\varrho}_{z}^{\varepsilon_{n}}\to\Bar\varrho_z$ in $L^{1}(\Rd)$,
the first term on the right hand side of \eqref{ET5.3G} converges to $0$
as $n\to\infty$.
Similarly, since $\widehat{m}_{z}^{\varepsilon_{n}}(x)\to M_z\,x$ and
$\grad\widehat{V}_{z}^{\varepsilon_{n}}\to\grad\Bar{V}_z$ uniformly on compact
subsets of $\Rd$,
the second term also converges to $0$.
Since $\Hat{\eta}_{z}^{\varepsilon}$ is an invariant probability measure
of \eqref{ET5.3E}, by the definition of $\mathring{\varrho}_{z}^{\varepsilon_{n}}$
we have $\int_{\Rd}\widehat{\Lg}_{z}^{\varepsilon_{n}}f(x)
\mathring{\varrho}_{z}^{\varepsilon_{n}}(x)\,\D{x}=0$, for all large
enough ${n}$, which implies
that $\int_{\Rd}\Bar\Lg_{z}f(x)\Bar\varrho_z(x)\,\D{x}=0$.
Hence, $\Bar\eta_z$ is an infinitesimal
invariant probability measure of \eqref{ES5.1A}, and since the diffusion
is regular, it is also an invariant probability measure.
This proves part~(a).

Since the diffusion in \eqref{ES5.1A} has an invariant probability measure,
it follows that it is positive recurrent.
By Lemma~\ref{L4.1} we have
\begin{equation*}
\sup_{\varepsilon\in(0,1)}\;
\int_{\{\varepsilon^{\nu}x+z\,\in\,\sB_{z}\}}\abs{x}^{2}\,
\Hat{\eta}_{z}^{\varepsilon}(\D{x})<\infty\,,
\end{equation*}
which implies by Fatou's lemma that $\int_{\Rd}\abs{x}^{2}\,\Bar\eta_z(\D{x})<\infty$.
Also by Theorem~\ref{T1.4} and Theorem~\ref{T1.18}\,(c)
we must have $\Bar{\beta}_*-\ell(z)= \varLambda^+(M_z)$.
This completes the proof of parts~(b) and (c).

By part~(c) and Theorem~\ref{T1.18}\,(c) the solution of $\eqref{EL5.2A}$
is unique and is given by \eqref{ET5.3A}.
That $\Bar\varrho_z$ is Gaussian with
covariance matrix $\widehat\Sigma_z$ follows by the second equation in \eqref{ED1.9}.
This proves part~(d).

Since $\Bar{V}_z$ has at most quadratic growth by \eqref{EL5.2E}, we have
$$\int_{\Rd}\abs{\Bar{V}_z(x)}\,\Bar\eta_z(\D{x})\;<\;\infty\,.$$
Therefore, with $\overline{\Exp}_{x}$ denoting
the expectation operator for the process governed by \eqref{ES5.1A},
it is the case that
$\overline{\Exp}_{x}\bigl[\Bar{V}_z(X_{t})\bigr]$ converges as $t\to\infty$
\cite[Theorem~4.12]{Ichihara-12}.
Integrating both sides of \eqref{EL5.2A} with respect to $\Bar\eta_z$,
we deduce that
\begin{equation}\label{ET5.3H}
\int_{\Rd}
\tfrac{1}{2}\,\abs{\grad \Bar{V}_z(x)}^{2}\,\Bar\eta_z(\D{x})\;=\;
\Bar{\beta}_* - \ell(z)\,.
\end{equation}
Using Fatou's lemma, we obtain by part~(d) that
\begin{align*}
\liminf_{\varepsilon_{n}\searrow0}\; \int_{\sB_z}
\sR[v_{*}^{\varepsilon_{n}}](x)\, \eta_{*}^{\varepsilon_{n}}(\D{x})
&\;=\;
\liminf_{\varepsilon_{n}\searrow0}\;
\int_{\{\varepsilon^{\nu}x+z\,\in\,\sB_z\}}
\Bigl(\widehat{\ell}_{z}^{\varepsilon_{n}}(x)
+\tfrac{1}{2}\abs{\grad \widehat{V}_{z}^{\varepsilon_{n}}(x)}^{2}\Bigr)\,
\Hat{\eta}_{z}^{\varepsilon_{n}}(\D{x})\nonumber\\[5pt]
&\;\ge\;
\lim_{R\to\infty}\liminf_{\varepsilon_{n}\searrow0}\;
\int_{\{\abs{x}\le R\}}
\Bigl(\widehat{\ell}_{z}^{\varepsilon_{n}}(x)
+\tfrac{1}{2}\abs{\grad \widehat{V}_{z}^{\varepsilon_{n}}(x)}^{2}\Bigr)\,
\eta_{*}^{\varepsilon_n}(\sB_z)\,
\mathring{\eta}_{z}^{\varepsilon_{n}}(\D{x})\nonumber\\[5pt]
&\;\ge\;\theta_z\,\bigl(\varLambda^+(M_z)+\ell(z)\bigr)\,,
\end{align*}
where in the second inequality we use \eqref{ET5.3H},
along with the hypothesis that $\eta_{*}^{\varepsilon_{n}}(\sB_{z})\to \theta_z>0$.
This proves part~(e) and thus completes the proof.
\end{proof}

Part of the statement in Theorem~\ref{T1.11}\,(iii)
follows from the following result.

%%%%%%%%%%%%%%%%%%%%%%%%%%%%%%%%%%%%%%%%%%%%%%%%%%%%%%%%%%%%%%%%%%%%%%%%%%%%%%%%
\begin{theorem}\label{T5.4}
Recall the definition of $\fJc$ from Theorem~\ref{T1.11}.
We assume $\nu=1$.
Then, it holds that $\lim_{\varepsilon\searrow0}\,\beta^\varepsilon_*= \fJc$.
Also $\Bar{\beta}_*$ in \eqref{EL5.2A} equals $\fJc$.
Moreover, for any $r>0$ we have
\begin{equation}\label{ET5.4A}
\lim_{\varepsilon\searrow0}\;\eta_{*}^\varepsilon\bigl(B^c_r(\cZ_c)\bigr)\;=\; 0\,,
\quad\text{and}\quad
\lim_{\varepsilon\searrow0}\;
\int_{B^{c}_r(\cZc)} \abs{v^\varepsilon_*(x)}^{2}\,
\eta_{*}^{\varepsilon}(\D{x}) \;=\;0\,.
\end{equation}
\end{theorem}

\begin{proof}
Since the collection $\{\sB_z\}$ used in Theorem~\ref{T5.3} was
arbitrary, without loss of generality, we may let $\sB_z=B_r(z)$.
Let $\varepsilon_{n}\searrow0$ be any sequence such that
$\eta_{*}^{\varepsilon_{n}}\bigl(B_r(z)\bigr)\to \theta_z$ for all $z\in\cS$,
and define $\cSo\df\{z\in\cS\,\colon \theta_z>0\}$.
Since $\cS$ is stochastically stable as shown in Theorem~\ref{T1.11},
we have $\sum_{z\in\cSo}\theta_z=1$.
By Theorem~\ref{T5.3}\,(e) we have
\begin{align}\label{ET5.4B}
\liminf_{n\to\infty}\; \beta_*^{\varepsilon_{n}}&\;\ge\;
\sum_{z\in\cSo}\,\int_{B_r(z)}
\Bigl(\ell(x) +\tfrac{1}{2}\abs{v_{*}^{\varepsilon_{n}}(x)}^{2}\Bigr)\,
\eta_{*}^{\varepsilon_{n}}(\D{x}) \\[5pt]
&\;\ge\;
\sum_{z\in\cSo}\,\theta_z\,\Bigl(\ell(z)+\varLambda^+\bigl(Dm(z)\bigr)\Bigr)
\;\ge\;\fJc\,.\nonumber
\end{align}
Since
$\limsup_{\varepsilon\searrow0}\,\beta^\varepsilon_*\,\le\,\fJc$
by Lemma~\ref{L3.3}, \eqref{ET5.4B} implies that
$\lim_{\varepsilon\searrow0}\,\beta^\varepsilon_*= \fJc$.
By Lemma~\ref{L5.2} we have
$\liminf_{\varepsilon\searrow0}\,\beta^\varepsilon_*\,\le\,\Bar{\beta}_*$,
and $\Bar{\beta}_*\le\fJc$ by \eqref{EL5.2C}.
Therefore $\Bar{\beta}_*=\fJc$.

Given any sequence $\varepsilon_{n}\searrow0$, we can extract a
subsequence also denoted by $\{\varepsilon_{n}\}$ along which
$\lim_{n\to\infty}\,\eta_{*}^{\varepsilon_{n}}\bigl(B_r(z)\bigr)\to\theta_z$
for all $z\in\cS$.
Then \eqref{ET5.4B} holds.
Also, by Proposition~\ref{P4.5} we have
$\int_{B^c_r(z)} \ell(x)\,\eta_{*}^{\varepsilon}(\D{x})\to0$
as $\varepsilon\searrow0$.
It is then clear that
both assertions in \eqref{ET5.4A} follow by \eqref{ET5.4B}.
\end{proof}

It is interesting to note that, even if
$\lim_{\varepsilon\to0}\,\eta_{*}^{\varepsilon_{n}}(\sB_{z})=0$,
equation \eqref{ET5.3A} still holds for any $z\in\cZc$.
This is part of the corollary that follows.

%%%%%%%%%%%%%%%%%%%%%%%%%%%%%%%%%%%%%%%%%%%%%%%%%%%%%%%%%%%%%%%%%%%%%%%%%%%%%%%%
\begin{corollary}
Suppose $\nu=1$.
Then for any $z\in\cZc$, we have
\begin{equation*}
\widehat{V}_{z}^{\varepsilon}(x)
-\widehat{V}_{z}^{\varepsilon}(z)\;\xrightarrow[\varepsilon\searrow0]{}\;
\frac{1}{2}\, \bigl\langle x, \widehat{Q}_{z}\, x\bigr\rangle\,,
\end{equation*}
uniformly on compact sets.
Also, unless $z\in\cZc$, then the family
$\{\mathring{\eta}_{z}^{\varepsilon}\,\colon \varepsilon\in(0,1)\}$
is not tight.
\end{corollary}

\begin{proof}
Since $\Bar{\beta}_*$ in \eqref{EL5.2A} equals $\fJc$ by Theorem~\ref{T5.4},
then, provided $z\in\cZc$, the right hand side of \eqref{EL5.2A}
equals $\varLambda^+(M_z)$.
The first assertion then follows by Theorem~\ref{T1.18}\,(c).

If the family $\{\mathring{\eta}_{z}^{\varepsilon}\,\colon \varepsilon\in(0,1)\}$
is tight, then it follows from the proof of Theorem~\ref{T5.3} that
the diffusion limit in \eqref{ES5.1A} is positive recurrent.
However, if $z\notin\cZc$, then $\Bar{\beta}_*-\ell(z)=\fJ_c-\ell(z)<\varLambda^+(M_z)$,
and by the results of Theorem~\ref{T1.4} and Theorem~\ref{T1.18}\,(c),
the diffusion in \eqref{ES5.1A} has to be transient.
Therefore, $\{\mathring{\eta}_{z}^{\varepsilon}\}$ cannot be tight.
\end{proof}

%%%%%%%%%%%%%%%%%%%%%%%%%%%%%%%%%%%%%%%%%%%%%%%%%%%%%%%%%%%%%%%%%%%%%%%%%%%%%%%%
\begin{remark}
It is worth examining the diffusion in \eqref{ES5.1A} in the context
of Example~\ref{E1.14}.
Consider the example with the first set of data, and let $c=5$.
Then $\sS=\{0\}$ and $\fJc=2$.
Thus, for $z=0$, we have $\Bar{V}_z= \Bar{V}_0 = 2x^2$,
and the drift in \eqref{ES5.1A} equals $-2\Bar{X}_t$.
For $z=-1$, we have $\ell(-1)=5$, $Dm(-1)=-3$, and direct substitution shows
that $\Bar{V}_{-1} = -3x^2$ solves \eqref{EL5.2A}.
The associated diffusion in \eqref{ES5.1A} has drift $3\Bar{X}_t$,
and thus it is transient. 
\end{remark}

%%%%%%%%%%%%%%%%%%%%%%%%%%%%%%%%%%%%%%%%%%%%%%%%%%%%%%%%%%%%%%%%%%%%%%%%%%%%%%%%
\subsection{Convergence to a Gaussian in the subcritical/supercritical
regime}\label{S5.2}

We return to the analysis of the subcritical and supercritical regimes
in order to determine the asymptotic behavior of the density 
of the optimal stationary distribution in the vicinity of the
stochastically stable set.
In these regimes there are two scales.
If we center the coordinates around a point in $\sS$,
then we have
$V^{\varepsilon}(x)\in\order\bigl(\varepsilon^{-2}\abs{x}^2\bigr)$,
and
$-\log\varrho_{*}^{\varepsilon}(x)\in\order\bigl(\varepsilon^{-2\nu}\abs{x}^2
\bigr)$.
To avoid this incompatibility we use the function
$\Breve{V}_z(x) = \varepsilon^{2(1-\nu)}V^{\varepsilon}(\varepsilon^{\nu}x)$
in the analysis,
which scales correctly in space for all $\nu$.
We have the following  result.

%%%%%%%%%%%%%%%%%%%%%%%%%%%%%%%%%%%%%%%%%%%%%%%%%%%%%%%%%%%%%%%%%%%%%%%%%%%%%%%%
\begin{theorem}\label{T5.7}
Assume $\nu\in(0,2)$ and
let $\{\sB_{z}\,\colon z\in\cS\}$ be as in Definition~\ref{D5.1}.
The following hold.
\begin{itemize}
\item[\textup{(}a\/\textup{)}]
Suppose that for some $z\in\cS$ and a sequence $\varepsilon_{n}\searrow0$
it holds that
$\liminf_{n\to\infty}\,\eta_{*}^{\varepsilon_{n}}(\sB_{z})>0$.
Then the density $\mathring{\varrho}_{z}^{\varepsilon_{n}}$
in Definition~\ref{D5.1} converges
as $n\to\infty$ (uniformly on compact sets) to the density
of a Gaussian with mean $0$ and covariance matrix $\widehat{\Sigma}_{z}$
given in \eqref{ED1.9}.
\item[\textup{(}b\/\textup{)}]
If $\nu\in(1,2)$ and $z\in\cS\setminus\widetilde\cZ$, then
$\lim_{\varepsilon\searrow0}\,\eta_{*}^{\varepsilon}(\sB_{z})=0$.
\end{itemize}
\end{theorem}

\begin{proof}
The proof closely follows those of Lemma~\ref{L5.2} and Theorem~\ref{T5.3}.
Only the scaling differs.
We summarize the essential steps.

First, suppose $\nu<1$.
Since
$\liminf_{n\to\infty}\,\eta_{*}^{\varepsilon_{n}}(\sB_{z})>0$
then necessarily $z\in\cSs$ by Lemma~\ref{L3.6}.
We scale the space as $\nicefrac{1}{\varepsilon^{\nu}}$,
and use \eqref{EL4.4B} which we write again here as
\begin{equation}\label{ET5.7A}
\frac{1}{2}\Delta \Breve{V}_{z}^{\varepsilon}(x)
+ \bigl\langle\widehat{m}_{z}^{\varepsilon}(x),
\grad \Breve{V}_{z}^{\varepsilon}(x)\bigr\rangle
-\frac{1}{2}\,\abs{\grad \Breve{V}_{z}^{\varepsilon}(x)}^{2}
\;=\; \varepsilon^{2(1-\nu)}\bigl(\beta^\varepsilon_*
-\widehat{\ell}_{z}^{\varepsilon}(x)\bigr)\,.
\end{equation}
By Lemma~\ref{L4.4},
$\grad\Breve{V}_{z}^{\varepsilon}=\varepsilon^{2(1-\nu)}
\grad\widehat{V}_{z}^{\varepsilon}$
is locally bounded and has at most linear growth.
We write \eqref{ET5.7A} as a HJB equation
\begin{equation}\label{ET5.7B}
\tfrac{1}{2}\Delta \Breve{V}_{z}^{\varepsilon}(x)
+ \min_{\Breve{u}\in\Rd}\;\Bigl[\bigl\langle\widehat{m}_{z}^{\varepsilon}(x)
+\Breve{u},\grad\Breve{V}_{z}^{\varepsilon}(x)\bigr\rangle
+\tfrac{1}{2}\,\abs{\Breve{u}}^{2}\Bigr]
\;=\; \varepsilon^{2(1-\nu)}\bigl(\beta^\varepsilon_*
-\widehat{\ell}_{z}^{\varepsilon}(x)\bigr)\,.
\end{equation}
The associated scaled controlled diffusion is
\begin{equation}\label{ET5.7C}
\D\widehat{X}_{t}\;=\; \bigl(\widehat{m}_{z}^{\varepsilon}(\widehat{X}_{t})
-\Breve{U}_t\bigr)\,
\D{t} + \D\widehat{W}_{t}\,.
\end{equation}
Taking limits in \eqref{ET5.7B} along some subsequence
$\varepsilon_{n}\searrow0$, 
we obtain a function $\Bar{V}_z\in\Cc^{2}(\Rd)$ of at most quadratic
growth satisfying
\begin{equation}\label{ET5.7D}
\tfrac{1}{2}\Delta \Bar{V}_z(x)
+ \min_{\Bar{u}\in\Rd}\;\Bigl[\bigl\langle M_z\, x+\Bar{u},
\grad\Bar{V}_z(x)\bigr\rangle
+ \tfrac{1}{2}\abs{\Bar{u}}^{2}\Bigr] \;=\; 0\,.
\end{equation}
The associated diffusion limit is
\begin{equation}\label{ET5.7E}
\D \Bar{X}_{t}\;=\;\bigl(M_z\,\Bar{X}_{t}-\grad \Bar{V}_z(\Bar{X}_{t})\bigr)\,\D{t}
+\D\Bar{W}_{t}\,.
\end{equation}

As in Section~\ref{S5.1}, $\Hat\eta^\varepsilon_*$ denotes
the invariant probability measure of \eqref{ET5.7C}
under the control $\Breve{U}_t= - \grad\Breve{V}_{z}^{\varepsilon}(X_t)$,
and $\Hat\varrho^\varepsilon_*$ its density.
Following the proof of Theorem~\ref{T5.3}, 
and using Lemma~\ref{L4.1}, we deduce that the density
$\mathring\varrho^\varepsilon_{z}$ in Definition~\ref{D5.1}
converges as $\varepsilon_{n}\searrow0$
to the density $\Bar\varrho_z$ of the invariant probability
measure of \eqref{ET5.7E}.
However since $M_z$ is Hurwitz, then $\Lambda^+(M_z)=0$,
and by Theorem~\ref{T1.18} we obtain $\Bar{V}_z\equiv 0$.
So in this case \eqref{ET5.7D} is trivial,
and the covariance matrix $\widehat{\Sigma}_{z}$
of the Gaussian is the solution of \eqref{ED1.9}
with  $\widehat{Q}_{z}=0$.

Next we assume $\nu\in(1,2)$, and we use the same scaling and definitions
as for the subcritical regime,
except that $z\in\cZ$.
It is clear that
\begin{equation*}
\varepsilon^{2(1-\nu)}\bigl(\widehat{\ell}_{z}^{\varepsilon}(x)-\ell(z)\bigr)
\;\le\; C_\ell\, \varepsilon^{2(1-\nu)} \varepsilon^{\nu}\abs{x}
\;\xrightarrow[\varepsilon\searrow0]{}\;0\,,
\end{equation*}
where $C_\ell$
denotes a Lipschitz constant of $\ell$.
By Corollary~\ref{C4.2} the constants
$\varepsilon^{2(1-\nu)}\bigl(\beta^\varepsilon_*-\ell(z)\bigr)$
are bounded, uniformly in $\varepsilon\in(0,1)$.
Therefore, as argued in the proof of Theorem~\ref{T5.3},
for every sequence $\varepsilon_{n}\searrow0$,
there exists a subsequence, also denoted
as $\{\varepsilon_{n}\}$ along which
$\varepsilon_{n}^{2(1-\nu)}\bigl(\beta^\varepsilon_*-\ell(x)\bigr)$
converges to a constant $\Hat{\beta}$, and
$\Breve{V}_{z}^{\varepsilon}(\cdot)-\Breve{V}_{z}^{\varepsilon}(z)$
converges to some
$\Bar{V}_z\in\Cc^{2}(\Rd)$, uniformly on compact sets.
Taking limits in \eqref{ET5.7B} along this subsequence, we obtain
\begin{equation}\label{ET5.7F}
\frac{1}{2}\Delta \Bar{V}_z(x)
+ \min_{\Bar{u}\in\Rd}\;\Bigl[\bigl\langle M_z\,x+\Bar{u},
\grad\Bar{V}_z(x)\bigr\rangle
+ \tfrac{1}{2}\abs{\Bar{u}}^{2}\Bigr] \;=\; \Hat{\beta}\,.
\end{equation}
Recall the notation $\widetilde\cZ$ and $\widetilde\fJ$ in Definition~\ref{D1.10}.
By Lemma~\ref{L3.3} we have
\begin{equation}\label{ET5.7G}
\Hat\beta\;\le\; \widetilde\fJ\;=\;\min_{z\in\cZ}\,\varLambda^+\bigl(Dm(z)\bigr)\,.
\end{equation}

Following exactly the same steps as in the proof of Theorem~\ref{T5.3},
we deduce that the diffusion in \eqref{ET5.7E}
is positive recurrent, with an invariant probability
measure $\Bar\eta_z$ that has finite second moments,
and that the density
$\mathring\varrho^\varepsilon_{z}$ in Definition~\ref{D5.1}
converges as $\varepsilon_{n}\searrow0$
to the density $\Bar\varrho_z$ of $\Bar\eta_z$.
Therefore,
\begin{equation}\label{ET5.7H}
\varLambda^+\bigl(Dm(z)\bigr)\;=\;\Hat{\beta}
\end{equation}
by Theorem~\ref{T1.18}\,(c).
Thus
$\Hat\beta= \widetilde\fJ=\varLambda^+\bigl(Dm(z)\bigr)$
by \eqref{ET5.7G}--\eqref{ET5.7H}.
This shows that unless $z\in\widetilde\cZ$, the hypothesis
$\liminf_{n\to\infty}\,\eta_{*}^{\varepsilon_{n}}(\sB_{z})>0$ cannot hold,
thus establishing part~(b) of the theorem.

With $z\in\widetilde\cZ$, and $\Hat\beta=\widetilde\fJ$, equation
\eqref{ET5.7F} has a unique solution by Theorem~\ref{T1.18}\,(c),
and we obtain
$\Bar{V}_{z}(x)=\tfrac{1}{2}\, \bigl\langle x,\widehat{Q}_z\, x\bigr\rangle$,
and that $\Bar\varrho_z$ is the density of a Gaussian with mean $0$
and covariance matrix $\widehat\Sigma_z$,
with $(\widehat{Q}_{z},\widehat{\Sigma}_{z})$ the pair
of matrices which solve \eqref{ED1.9}.
This completes the proof.
\end{proof}

%%%%%%%%%%%%%%%%%%%%%%%%%%%%%%%%%%%%%%%%%%%%%%%%%%%%%%%%%%%%%%%%%%%%%%%%%%%%%%%%
\section{Concluding remarks}
In general, Morse--Smale flows may contain hyperbolic closed orbits,
and it would be desirable to extend the results of the paper accordingly.
An energy function $\Lyap$ as in Theorem~\ref{T2.2} may be constructed
to account for critical elements that are closed orbits
\cite{Smale,Meyer}.
Note that under the control used in Remark~\ref{R3.4}
the optimal stationary distribution concentrates on
the minimum of $\Lyap$.
In the case that $z\in\Rd$ belongs to a stable periodic orbit
with period $T_{0}$, we can construct $\Lyap$
so that it attains its minimum on this closed orbit.
In this manner, if $\phi_{t}$ denotes the flow of the
vector field $m$, then it follows by \eqref{EL3.1G}
that under the control used in Remark~\ref{R3.4}, we obtain
\begin{equation*}
\int_{\Rd}\ell(x)\,\mu^{\varepsilon}(\D{x})
\;\xrightarrow[\varepsilon\searrow0]{}\;
\frac{1}{T_{0}}\int_{0}^{T_{0}}\ell\bigl(\phi_{t}(z)\bigr)\,\D{t}\,.
\end{equation*}
The same can be done in the subcritical regime, by modifying
the proof of Lemma~\ref{L3.5}, and using instead the approach in
Remark~\ref{R3.4}.
We leave it up to the reader to verify that Lemma~\ref{L3.1} still holds
if the set of critical elements $\cS$ contains
hyperbolic closed orbits.
Let us define
\begin{equation*}
\mathring{\ell}(z) \;\df\;
\frac{1}{T_{0}}\int_{0}^{T_{0}}\ell\bigl(\phi_{t}(z)\bigr)\,\D{t}\,,
\end{equation*}
when $z$ belongs to a closed orbit,
and $\mathring{\ell}(z)=\ell(z)$, when $m(z)=0$.
Then, provided $\Argmin_{z\in\cS}\,\mathring{\ell}(z)$
contains only stable critical elements, then the support of
the limit of the optimal stationary distribution lies in
$\cSs$, and this is true in any of the three regimes.
However, the full analysis when unstable
closed orbits are involved seems to be more difficult.

%%%%%%%%%%%%%%%%%%%%%%%%%%%%%%%%%%%%%%%%%%%%%%%%%%%%%%%%%%%%%%%%%%%%%%%%%%%%%%%%
\appendix
\section{Proofs of the results in Section~\ref{S1.3}}\label{AppA}

We start with the proof of Lemma~\ref{L1.3}.

\begin{proof}[Proof of Lemma~\ref{L1.3}]
The proof is standard.
Let $U$ be given and define
$M_t\,\df\,\Exp\bigl[\int_{0}^{t}\,\abs{U_{s}}^{2}\,\D{s}\bigr]$, $t\in\RR_+$.
For $T>0$, let $\cH^2_T$ denote the space of $\{\sF_{t}\}$-adapted processes $Y$
defined on $[0,T]$, having continuous sample paths, and satisfying
$\Exp\bigl[\sup_{0\le t\le T}\, \abs{Y_t}^2\bigr]<\infty$.
The space $\cH^2_T$ (more precisely the set of equivalence classes
in $\cH^2_T$) is a Banach space under the norm
\begin{equation*}
\norm{Y}_{\cH^2_T}\;\df\;
\biggl(\Exp\biggl[\sup_{0\le t\le T}\, \abs{Y_t}^2\biggr]\biggr)^{\nicefrac{1}{2}}\,.
\end{equation*}
It is standard to show, for example following the proof
of \cite[Theorem~2.2.2]{book}
that any solution $X$ of \eqref{E-sde} satisfies
\begin{equation}\label{EL1.3-bound}
\norm{X-X_0}_{\cH^2_t}^2\;\le\; \kappa_0 t(1+t)
\bigl(1 + M_t +\Exp\bigl[\abs{X_0}^2\bigr]\bigr)\,\E^{\kappa_1 t}
\qquad \forall\,t\ge0\,,
\end{equation}
for some constants $\kappa_0$ and $\kappa_1$ that depend only on $m$.
The existence of a pathwise unique solution then follows by applying
the contraction mapping theorem as in \cite[Theorem~2.2.4]{book}.
\end{proof}

The rest of this section is devoted to the proof of Theorem~\ref{T1.4}.
Without loss of generality we fix $\varepsilon=1$, and suppress
the dependence on $\varepsilon$ in all the variables.
%and denote the optimal value $\beta^\varepsilon_*$ as $\beta_*$.
Also throughout the rest of this section, without loss of generality
we assume that $\ell\ge0$.

We proceed by establishing two key lemmas, followed by
the proof of Theorem~\ref{T1.4}.
Recall the definition of $\sR$ in \eqref{E-cRnew}.
For $x\in\Rd$, and $\alpha>0$,
we define the subset $\Uadm_x^\alpha$ of admissible controls by
\begin{equation}\label{E-Uadmalph}
\Uadm_x^\alpha\;\df\;
\biggl\{U\in\Uadm\,\colon
\Exp^U_{x}\biggl[\int_{0}^\infty \E^{-\alpha s}\,\sR(X_{s},U_{s})\,\D{s}\biggr]
\;<\;\infty\biggr\}\,,
\end{equation}
where $\Exp_x^{U}$ denotes the expectation under the law of $(X,U)$,
with $X_0=x$ for the solution of
\begin{equation}\label{ELA.1A}
X_{t}\;=\;x+\int_{0}^{t} m(X_{s})\,\D{s}
+\int_{0}^{t} U_{s}\,\D{s} + W_{t}\,, \quad t\ge 0\,.
\end{equation}

%%%%%%%%%%%%%%%%%%%%%%%%%%%%%%%%%%%%%%%%%%%%%%%%%%%%%%%%%%%%%%%%%%%%%%%%%%%%%%%%
\begin{lemma}\label{LA.1}
The equation
\begin{equation}\label{ELA.1B}
\tfrac{1}{2}\,\Delta V_{\alpha}+\langle m, \nabla V_{\alpha}\rangle
 - \tfrac{1}{2}\,|\nabla V_{\alpha}|^{2}+ \ell\;=\; \alpha V_{\alpha}
\end{equation}
has a solution in $\Cc^2(\Rd)$ for all $\alpha\in(0,1)$.
Moreover, for all $\alpha\in(0,1)$, we have the following.
\begin{itemize}
\item[\textup{(}i\/\textup{)}]
For some constant $c_0>0$, not depending on $\alpha$, it holds that
\begin{equation}\label{ELA.1C}
\abs{\grad V_\alpha(x)} \;\le\; c_0 \sqrt{1+\abs{x}}\,,\quad\text{and\ \ }
\abs{\alpha V_\alpha(x)} \;\le\;  \ell(x)+\tfrac{c_0}{\alpha} 
\end{equation}
for all $x\in\Rd$.

\item[\textup{(}ii\/\textup{)}]
The function $V_\alpha$ satisfies
\begin{equation}\label{ELA.1D}
V_{\alpha}(x)\;\le\; \inf_{U\in\,\Uadm_x^\alpha}\;\Exp^{U}_{x}
\biggl[\int_{0}^{\infty} \E^{-\alpha s}\,\sR(X_{s},U_{s})\,\D{s}\biggr]\,,
\qquad \forall\,x\in\Rd\,.
\end{equation}

\item[\textup{(}iii\/\textup{)}]
With $\Bar{c}_\ell$ the constant in \eqref{EE1.8}, we have
\begin{equation*}
\inf_{\{x\,\colon \ell(x)\,\le\,\Bar{c}_\ell\}}\; \alpha V_\alpha
\;=\; \inf_{\Rd}\; \alpha V_\alpha
\;\le\; \Bar{c}_\ell\,.
\end{equation*}
\end{itemize}
\end{lemma}

\begin{proof}
In \cite[Theorem~4.18, p.~177]{Bens-Fre-book} it is proved that \eqref{ELA.1B} has a
solution in $\Cc^2(\Rd)$, and it also shown in the proof of this theorem
that  there exists a
constant $\kappa_{0}>0$ which does not depend on $\alpha$ such that
\begin{equation}\label{ELA.1E}
\alpha\,V_\alpha(x)\;\ge\; -\kappa_{0}\qquad \forall\,x\in \Rd\,.
\end{equation}
By \cite[Theorem~B.1]{Ichihara-12} there exists a constant $C$ not depending
on $R>0$ such that
\begin{equation}\label{ELA.1EX}
\sup_{B_R}\; \abs{\grad V_\alpha}
\;\le\; C\Bigl( 1+ \sup_{B_{R+1}}\; \sqrt{(\alpha V_\alpha)^-} +
\sup_{B_{R+1}}\;\sqrt{\ell^+}
+\sup_{B_{R+1}}\; \abs{\grad \ell}^{\nicefrac{1}{3}}\Bigr)\,.
\end{equation}
from which gradient estimate in \eqref{ELA.1C} follows.
The structural assumption on the Hamiltonian $h(x,p)$ in
\cite[Theorem~B.1]{Ichihara-12} is $p\mapsto h(x,p)$ is strictly convex
for all $x\in\Rd$, and there exists some constant $k_0>0$ such that
\begin{equation}\label{E-IchiharaH1}
k_0\,\abs{p}^2\;\le\; h(x,p)\;\le\; k_0^{-1}\,\abs{p}^2\,,
\qquad \abs{\grad_x h(x,p)}\;\le\; k_0^{-1} \bigl(1+\abs{p}^2\bigr)\,,
\end{equation}
for $(x,p)\in\RR^{2d}$.
This Hamiltonian corresponds to $h(x,p)=\frac{1}{2}\abs{p}^2-\langle m,p\rangle$
for the equation in \eqref{ELA.1B}, and the first bound in
\eqref{E-IchiharaH1} is not satisfied.
However, replacing this bound with
\begin{equation*}
k_0\,\bigl(\abs{p}^2-k_1\bigr)\;\le\; h(x,p)\;\le\;
k_0^{-1}\,\bigl(\abs{p}^2+k_1\bigr)\,,
\end{equation*}
for some constant $k_1\ge0$, the proof of \cite[Theorem~B.1]{Ichihara-12}
goes through unmodified.

Recall the definition of $\widehat\Uadm$ in \eqref{E-HUadm}.
Writing \eqref{ELA.1B} in HJB form, and applying It\^o's formula we obtain
\begin{equation}\label{ELA.1F}
V_{\alpha}(x)-\E^{-\alpha t}\Exp^{U}_{x}\bigl[V_{\alpha}(X_{t})\bigr]\;\le\;
\Exp^{U}_{x}\biggl[\int_{0}^{t}
\E^{-\alpha s}\,\sR(X_{s},U_{s})\,\D{s}\biggr]\qquad\forall\,t>0\,,
\end{equation}
and all $U\in\widehat\Uadm$.
Since $m$ is bounded, then it is standard to show using \eqref{ELA.1A} that
\begin{equation}\label{ELA.1G}
\Exp^{U}_{x}\biggl[\sup_{0\le s\le t}\;\abs{X(s)-x}\biggr]
\;\le\;  \norm{m}_\infty\, t + \sqrt{t}
+ \Exp^{U}_{x}\biggl[\int_{0}^{t} \abs{U_{s}}\,\D{s}\biggr]
\;<\;\infty
\end{equation}
for all $U\in\widehat\Uadm$ and $t>0$.
Also, if $\Exp^0_{x}$ denotes the expectation
$\Exp^U_{x}$ with $U=0$, then by \eqref{ELA.1A} we have the estimate
\begin{equation}\label{ELA.1H}
\Exp^{0}_{x}\bigl[\abs{X_t}^2\bigr]
\;\le\;  \kappa_2 \bigl(1+t^2+\abs{x}^2\bigr)
\;<\;\infty\qquad\forall\, t>0\,,
\end{equation}
for some constant $\kappa_2$.
As shown in the proof of \cite[Theorem~4.18, p.~177]{Bens-Fre-book},
$\alpha\mapsto\alpha V_{\alpha}(0)$ is bounded on $(0,1)$, which
together with the gradient estimate in \eqref{ELA.1C} we have already
proved, provides us with a liberal bound of $V_\alpha$ of the
form $\abs{V_\alpha(x)}\le C \bigl(1+\abs{x}^2)$ for some constant $C$.
This combined with \eqref{ELA.1H} implies that
$\E^{-\alpha t}\Exp^{0}_{x}\bigl[V_{\alpha}(X_{t})\bigr]\to0$
as $t\to\infty$.
Therefore, using \eqref{ELA.1G}, and the Lipschitz constant $C_\ell$ of $\ell$,
we obtain by \eqref{ELA.1F} that
\begin{align*}
\alpha V_\alpha(x) &\;\le\;
\Exp^0_{x}\biggl[\int_0^\infty \alpha\,\E^{-\alpha s}\,\ell(X_s)\,\D{s}\biggr]\\[5pt]
&\;\le\;\ell(x)
+ C_\ell \int_0^\infty \alpha\,\E^{-\alpha s}
\bigl(\norm{m}_\infty\, s + 2 \sqrt{s}\bigr)\,\D{s}
\qquad\forall\,x\in\Rd\,,
\end{align*}
which results in the estimate given in \eqref{ELA.1C},
where without loss of
generality we use a common constant $c_0$.
This completes the proof of part~(i).

Let $g(x,t)\df\abs{x}+\norm{m}_\infty\, t + 2 \sqrt{t}$.
Multiplying both sides  of \eqref{ELA.1G} by $\E^{-\alpha t}$,
strengthening the inequality, and applying the H\"older inequality, we obtain
\begin{align}\label{ELA.1I}
\E^{-\alpha t}\Exp^{U}_{x}\bigl[\abs{X_{t}}\bigr] &\;\le\; 
 g(x,t)\,\E^{-\alpha t} + \E^{-\frac{\alpha}{2}t}
\Exp^{U}_{x}\biggl[\int_{0}^{t} \E^{-\frac{\alpha}{2}s}\,
\abs{U_{s}}\,\D{s}\biggr]\\[5pt]
&\;\le\; g(x,t)\,\E^{-\alpha t}
+ \sqrt{t}\E^{-\frac{\alpha}{4}t}
\biggl(\Exp^{U}_{x}\biggl[\int_{0}^{t}
\E^{-\alpha s}\,\abs{U_{s}}^2\,\D{s}\biggr]\biggr)^{\nicefrac{1}{2}}\nonumber\\[5pt]
&
\;\xrightarrow[t\to\infty]{}\; 0\qquad\forall\,U\in\Uadm_x^\alpha\,,\nonumber
\end{align}
with $\Uadm_x^\alpha$ as defined in \eqref{E-Uadmalph}.
Taking limits as $t\to\infty$ in \eqref{ELA.1F}, and using \eqref{ELA.1I},
and the bound of $V_\alpha$ in  \eqref{ELA.1C} together with
$\abs{\ell(x)}\le C_{l}\abs{x}+\abs{\ell(0)}$,
we obtain \eqref{ELA.1D}.

We now turn to part~(iii).
Let
\begin{equation*}
\chi(x)\;\df\; \frac{1}{\sqrt{3}}\,
\biggl(\min_{y\,\in\,B_1(x)}\,\Bigl[\ell(y)-(d+1+2\sqrt{d}\,\norm{m}_\infty)^2\Bigr]
\biggr)^{\nicefrac{1}{2}}\,,
\end{equation*}
and
\begin{equation*}
\psi(x) \;\df\; V_\alpha(x) +\tfrac{2\kappa_0}{\alpha}
-\chi(x_0)\bigl(1-\abs{x-x_{0}}^{2}\bigr)\,,\qquad x\in B_1(x_0)\,,
\end{equation*}
where $\kappa_0>0$ is the constant in \eqref{ELA.1E}.
With $\phi(x)\df \abs{x-x_{0}}^{2}$, we have 
\begin{align*}
-\tfrac{1}{2}\,\Delta \psi -\langle m-\grad V_\alpha, \grad \psi\rangle
+\alpha \psi 
&\;=\;
\Bigl(-\tfrac{1}{2}\,\Delta V_{\alpha}-\langle m, \nabla V_{\alpha}\rangle
+\tfrac{1}{2}\,\abs{\nabla V_{\alpha}}^{2}+ \alpha V_{\alpha}\Bigr)
\\[5pt]
&\mspace{50mu}
+\tfrac{1}{2}\,\babs{\grad V_\alpha-\chi(x_0)\grad\phi}^2 -2\chi^2(x_0)\phi
+ 2\kappa_0
\\[5pt]
&\mspace{100mu}
-\chi(x_0)\Bigl(\tfrac{1}{2}\,\Delta\phi +\langle m,\grad\phi\rangle
+\alpha(1-\phi)\Bigr) 
\\[5pt]
&\;\ge\;\ell - 2\chi^2(x_0) + 2\kappa_0
- \bigl(d+2\sqrt{d}\,\norm{m}_\infty+1\bigr)\chi(x_0) 
 \\[5pt]
&\;\ge\;  \ell - 3\chi^2(x_0) - \bigl(d+1+2\sqrt{d}\,\norm{m}_\infty\bigr)^2 
\quad \text{in\ \ }B_1(x_0)\,,
\end{align*}
for all $\alpha\in(0,1)$,
where we use \eqref{ELA.1B} and the fact that $\kappa_0\ge0$.
Since $\psi>0$ on $\partial B_1(x_0)$ by \eqref{ELA.1E},
an application of the strong maximum principle shows that
$\psi\ge0$ in $B_1(x_0)$, which implies that
\begin{equation*}
\alpha V_\alpha(x)\;\ge\; \alpha \chi(x) + \kappa_0\qquad \forall\; x\in \Rd\,.
\end{equation*}
Since $\ell$ is inf-compact, and therefore the same is true for $\chi$
by its definition,
this shows that $\alpha V_\alpha$ is inf-compact.
In particular, it attains its infimum in $\Rd$.
With $\eta_0$ denoting the invariant probability measure of the diffusion
in \eqref{ELA.1A} under the control $U=0$, using \eqref{ELA.1D},
we obtain
\begin{equation}\label{ELA.1J}
\inf_{\Rd}\;V_{\alpha} \;\le\; \int V_{\alpha} \,\D\eta_{0}
\;\le\; \int_{\Rd} \Exp_{x}
\biggl[\int_{0}^{\infty} \E^{-\alpha s}\ell(X_{s})\,\D{s}\biggr]\,\eta_{0}(\D{x})
\;\le\; \frac{\Bar{c}_\ell}{\alpha}\,,
\end{equation}
where the last inequality follows by \eqref{EE1.8}.
One more application of the maximum principle implies that
if $V_\alpha$ attains its infimum at $\Hat{x}\in\Rd$ then
$\ell(\Hat{x})\le \alpha V_\alpha(\Hat{x})$.
This together with \eqref{ELA.1J} implies part~(iii).
\end{proof}

%%%%%%%%%%%%%%%%%%%%%%%%%%%%%%%%%%%%%%%%%%%%%%%%%%%%%%%%%%%%%%%%%%%%%%%%%%%%%%%%
\begin{remark}%\label{RA.2}
We should mention, even though we don't need it for the proof of the main
theorem, that \eqref{ELA.1D} holds with equality, and thus
$V_\alpha$ is indeed the value of the infinite horizon
discounted control problem.
The proof of this assertion goes as follows.
Since $\grad V_\alpha$ has at most linear growth, the diffusion
in  \eqref{ELA.1A} under the Markov control $v_\alpha = -\grad V_\alpha$
has a unique strong solution.
It is also clear by \eqref{ELA.1C} that for any $\alpha>0$ we can select
a constant $\kappa_1(\alpha)$ such
that $\abs{\grad V_\alpha(x)}\le \kappa_1(\alpha) + \frac{\alpha}{16}{x}$.
Thus using a standard estimate \cite[Theorem~2.2.2]{book}
we obtain
\begin{equation}\label{ERA.2A}
\Exp^{v_\alpha}_{x}\biggl[\sup_{0\le s\le t}\;\abs{X(s)}^2\biggr]
\;\le\; \kappa_2(\alpha) (1 + t^2) (1+ \abs{x}^2)\E^{\frac{\alpha}{2}t}
\end{equation}
for some constant $\kappa_2(\alpha)>0$.
With $\uptau_R$ denoting the first exit time from $B_R$, applying Dynkin's
formula we obtain
\begin{equation*}
V_\alpha(x) \;=\; \Exp^{v_\alpha}_x\biggl[\int_0^{t\wedge\uptau_R}
\E^{-\alpha s}\,\sR(X_{s},v_\alpha(X_{s}))\,\D{s}\biggr]
+ \Exp^{v_\alpha}_x\bigl[\E^{-\alpha(t\,\wedge\,\uptau_R)}\,
V_\alpha(X_{t\wedge\uptau_R})\bigr]\,.
\end{equation*}
We write
\begin{equation*}
\Exp^{v_\alpha}_x\bigl[\E^{-\alpha(t\wedge\uptau_R)}\,
V_\alpha(X_{t\wedge\uptau_R})\bigr]\;=\;
A_1(t,R)+ A_2(t,R)\,.
\end{equation*}
with
\begin{align*}
A_1(t,R)&\;\df\;\Exp^{v_\alpha}_x\bigl[\E^{-\alpha t}\,
V_\alpha(X_{t\wedge\uptau_R})\,\Ind_{\{t\le\uptau_R\}}\bigr]\,,\\[5pt]
A_2(t,R)&\;\df\;\Exp^{v_\alpha}_x\bigl[\E^{-\alpha \uptau_R}\,
V_\alpha(X_{t\wedge\uptau_R})\,\Ind_{\{\uptau_R\,<\,t\}}\bigr]\,.
\end{align*}
Since $V_\alpha$ has at most linear growth
in $x$ by \eqref{ELA.1C}, it follows by \eqref{ERA.2A}
that
\begin{equation*}
\lim_{t\to\infty}\;\limsup_{R\to\infty}\;\abs{A_1(t,R)} \;=\; 0\,.
\end{equation*}
We also have $\limsup_{R\to\infty}\,\abs{A_2(t,R)} = 0$
 by dominated convergence,
since $\Prob^{v_\alpha}_x \bigl(\uptau_R<t\bigr)\to0$ as $R\to\infty$.
Thus, taking limits first as $R\to\infty$,
and then as $t\to\infty$ in \eqref{ERA.2A}, we obtain
\begin{equation*}
V_\alpha(x) \;\ge\; \Exp^{v_\alpha}_x\biggl[\int_0^{\infty}
\E^{-\alpha s}\,\sR(X_{s},v_\alpha(X_{s}))\,\D{s}\biggr]\,.
\end{equation*}
Thus the converse inequality to \eqref{ELA.1D} also holds.
\end{remark}

Define the class of controls $\overline\Uadm_x$ by
\begin{equation*}
\overline\Uadm_x\;\df\;
\biggl\{U\in\Uadm\,\colon\limsup_{T\to\infty}\;\frac{1}{T}\,
\Exp^U_{x}\biggl[\int_{0}^T \sR(X_{s},U_{s})\,\D{s}\biggr]
\;<\;\infty\biggr\}\,.
\end{equation*}

%%%%%%%%%%%%%%%%%%%%%%%%%%%%%%%%%%%%%%%%%%%%%%%%%%%%%%%%%%%%%%%%%%%%%%%%%%%%%%%%
\begin{lemma}\label{LA.3}
There exists an inf-compact  $V\in\Cc^2(\Rd)$ which satisfies
\begin{equation}\label{ELA.3A}
\cA[V](x)\;\df\; \tfrac{1}{2}\,\Delta V + \langle m,\nabla V\rangle - \tfrac{1}{2}\,
\abs{\nabla V}^{2} + \ell\;=\;\beta\,,
\end{equation}
with
\begin{equation*}
\beta\;=\;\beta_* \;\df\; \inf_{U\in\overline\Uadm_x}\;
\limsup_{T\to\infty}\;\frac{1}{T}\,
\Exp^U_{x}\biggl[\int_{0}^T \sR(X_{s},U_{s})\,\D{s}\biggr]\,,
\end{equation*}
and For some constant $c_0>0$, it holds that
$\abs{\grad V(x)} \;\le\; c_0 \sqrt{1+\abs{x}}$ for all $x\in\Rd$.
In addition, under the Markov control $U_t=v_*(X_t)$, with $v_{*}=-\grad V$, the
diffusion in \eqref{ELA.1A} is positive recurrent, and
$\beta_*=\int_{\Rd}\sR[v_*](x)\,\D\eta_{*}$,
where $\eta_{*}$ is the invariant probability measure
corresponding to the control $v_{*}$.
\end{lemma}

\begin{proof}
The existence of a solution to \eqref{ELA.3A} is established as a limit of
$V_{\alpha}(\cdot)-V_\alpha(0)$, with
$V_\alpha$ the solution of \eqref{ELA.1B} in Lemma~\ref{LA.1}  along
some sequence $\alpha_n\searrow0$ \cite[p.~175]{Bens-Fre-book}.
That $V$ is inf-compact follows by  \cite[Theorem~4.21]{Bens-Fre-book}.
It also follows from the proof from this
convergence result that $\beta\le\limsup_{\alpha\searrow0}\,\alpha V_\alpha(x)$ for
all $x\in\Rd$.

We first show that $\beta\le\beta_*$.
For this, we employ the following assertion which is a special case of the
Hardy--Littlewood theorem \cite{Sz-Fil}.
For any sequence $\{a_{n}\}$ of non-negative real numbers, it holds that
\begin{equation}\label{E-Tauber}
\limsup_{\theta\nearrow 1}\; (1-\theta)\sum_{n=1}^{\infty} \theta^n a_{n} 
\;\le\; \limsup_{N\to\infty}\;\frac{1}{N}\sum_{n=1}^N a_{n}\,.
\end{equation}
Concerning this assertion,
note that if the right hand side of the above display is finite then the set
$\{\frac{a_{n}}{n}\}$ is bounded. Therefore 
$\sum_{n=1}^{\infty} \theta^n a_{n}$ in finite for every $\theta<1$.
Hence we can apply \cite[Theorem~2.2]{Sz-Fil} to obtain \eqref{E-Tauber}.

Fix $x\in \Rd$, and $U\in\overline\Uadm_x$.
Define
\begin{equation*}
a_{n}\;\df\;
\Exp_{x}^{U}\biggl[\int_{n-1}^n\sR(X_{s},U_s)\,\D{s}\biggr]\,,\quad n\ge 1\,.
\end{equation*}
and let $\theta=\E^{-\alpha}$.
Applying \eqref{E-Tauber}, and with $N$ running over the set of natural numbers,
we obtain
\begin{align}\label{ELA.3C}
\limsup_{N\to\infty}\;  \frac{1}{N}\, 
\Exp_{x}^{U}\biggl[\int_{0}^N\sR(X_{s},U_s)\,\D{s}\biggr]
&\;\ge\; \limsup_{\theta\nearrow 1}\; (1-\theta)\sum_{n=1}^{\infty}\theta^n a_{n}\\[5pt]
&\;\ge\; \limsup_{\alpha\searrow 0}\; (1-\E^{-\alpha})\sum_{n=1}^{\infty}
\Exp_{x}^{U}\biggl[\int_{n-1}^n \E^{-\alpha s}\,
\sR(X_{s},U_s)\Bigr)\,\D{s}\biggr]
\nonumber\\[5pt]
&\;\ge\; \limsup_{\alpha\searrow 0}\;
(1-\E^{-\alpha})\Exp_{x}^{U}\biggl[\int_{0}^{\infty}
\E^{-\alpha s}\,\sR(X_{s},U_s)\,\D{s}\biggr]
\;\ge\;\beta\,.\nonumber
\end{align}
where we use the property that
$\limsup_{\alpha\searrow0}\,\alpha V_{\alpha}(x)\ge \beta$.
Since $U\in\overline\Uadm_x$ is arbitrary, \eqref{ELA.3C} together
with the definition of $\beta_*$ imply that $\beta\le \beta_*$.
Note also that \eqref{ELA.3C} implies that $\Uadm_x^\alpha\subset\overline\Uadm_x$
for all $\alpha\in(0,1)$.

Next, we prove the converse inequality.
By \eqref{ELA.1C} we have
$\abs{\nabla V(x)}\le c_0\,\sqrt{1+\abs{x}}$.
Therefore, since the Markov control $v_*\df -\grad V(x)$ has
at most linear growth,
there exists a unique strong solution to \eqref{ELA.1C} under the control $v_{*}$.
Applying It\^{o}'s formula to \eqref{ELA.3A}, and using \eqref{E-sR},
we obtain
\begin{equation*}
\Exp_{x}^{v_*}\bigl[V(X_{T\wedge\uptau_{R}})\bigr]-V(x)
+ \Exp_{x}\biggl[\int_{0}^{T\wedge\uptau_{R}}\sR[v_*](X_s)\,\D{s}\biggr]
\;=\;\beta\,\Exp_{x}\bigl[T\wedge\uptau_{R}\bigr]\,,
\end{equation*}
where $\uptau_{R}$ denotes the exit time from the ball of radius $R>0$ around $0$.
Since $V$ is bounded from below and $\uptau_{R}\to\infty$ a.s.,
as $R\to\infty$, using Fatou's lemma for the integral
on the left hand side, and then dividing by $T$ and taking limits
as $T\to\infty$, results in
\begin{equation*}
\limsup_{T\to\infty}\;\frac{1}{T}\,
\Exp_{x}^{v_*}\biggl[\int_{0}^T\sR[v_*](X_s)\,\D{s}\biggr]\;\le\; \beta\,.
\end{equation*}
Thus $\beta=\beta_*$.
Since $\ell$ is inf-compact this also implies that the diffusion
under the control $v_*$ is positive recurrent,
and by Birkhoff's ergodic theorem we obtain
$\beta_*=\int_{\Rd}\sR[v_*](x)\,\D\eta_{*}$,
and this completes the proof.
\end{proof}

Let $\Lg$ and $\Lg_v$ denote the operators defined in \eqref{E-Lg} and
\eqref{E-Lgv}, respectively, with $\varepsilon=1$.
Also $\fP$ denote the set of infinitesimal ergodic occupation measures
as defined in \eqref{E-eom}, again with $\varepsilon=1$.
%Let $\fP$ denote the set of infinitesimal ergodic occupation measures, that
%is the set of measures
%$\uppi=\eta_v\circledast v\in\cP(\Rd\times\Rd)$ satisfying
%\begin{equation}\label{E-eomX}
%\int_{\Rd\times\Rd} \Lg[f](x,u)\,\uppi(\D{x},\D{u})
%\;=\;0 \qquad \forall\,f\in\Ccc^{\infty}(\Rd)\,,
%\end{equation}
%where
%\begin{equation*}
%\Lg[f](x,u) \;\df\;  \frac{1}{2}\,\Delta f(x)
%+ \bigl\langle m(x)+u, \nabla f(x)\bigr\rangle
%\qquad\forall\,x\in\Rd\,, \quad f \in \Cc^{2}(\Rd)\,,
%\end{equation*}
%We also let $\Lg_v$ be defined as in \eqref{E-Lgv} with $\varepsilon=1$.
Note that if $\uppi=\eta_v\circledast v\in\fP$ then
\eqref{E-eom} can be written as $\int_\Rd \Lg_v f(x)\,\eta_v(\D{x})=0$.

Let $\Hat{v}(x)=\int v(\D{u}\,|\, x)$.
Since
\begin{equation}\label{EA.19}
\int_{\Rd\times\Rd}\abs{u}^2\eta_v(\D{x})\,v(\D{u}\,|\, x)
\;\ge\; \int_\Rd\abs{\Hat{v}(x)}^2\eta_v(\D{x})\,,
\end{equation}
and since $\eta_v\circledast\Hat{v}$ is also an infinitesimal ergodic
occupation measure, it is evident that as far as the proof of strong duality
is concerned we may restrict our attention to the subset of $\fP$ that
corresponds to precise controls, and which we denote
as $\fP_\circ$.
%$\eta_v\circledast v\in\fP$ is optimal,
%then $v$ is a precise control $\eta_v$-a.e.
%We denote this subset of $\fP$ by $\fP_\circ$.

We have the following lemma.

%%%%%%%%%%%%%%%%%%%%%%%%%%%%%%%%%%%%%%%%%%%%%%%%%%%%%%%%%%%%%%%%%%%%%%%%%%%%%%%%
\begin{lemma}\label{LA.4}
If $\uppi=\eta_v\circledast v\in\fP_\circ$ is such that
\begin{equation}\label{ELA.4A}
\int_{\Rd\times\Rd}\sR[v](x)\,\eta_v(\D{x})\;<\;\infty\,,
\end{equation}
then
\begin{equation}\label{ELA.4B}
\int_{\Rd} \sR[v](x)\,\eta_v(\D{x}) \;=\; \beta_*
+ \frac{1}{2}\,\int_\Rd \babs{v(x)+\grad V(x)}^2\,\eta_v(\D{x})\,.
\end{equation}
In addition the measure $\eta_v$
has a density $\varrho_v\in L^{\nicefrac{d}{(d-1)}}(\Rd)$.
\end{lemma}

\begin{proof}
Let $\chi$ be a concave $\Cc^2(\Rd)$ function such that
$\chi(x)= x$ for $x\le0$, and $\chi(x) = 1$ for $x\ge 1$.
Then $\chi'$ and $-\chi''$ are nonnegative on $(0,1)$.
Define $\chi_R(x) \df R + \chi(x-R)$, $R>0$.
By \eqref{ELA.3A}, completing the square, we have
\begin{equation*}
\Lg_v V - \tfrac{1}{2} \abs{v+\grad V}^2 + \sR[v] -\beta_* \;=\;0\,.
\end{equation*}
Therefore, we obtain
\begin{equation}\label{ELA.4C}
\Lg_v\chi_R(V) - \tfrac{1}{2}\chi''_R(V)\,\abs{\grad V}^2
- \tfrac{1}{2}\chi'_R(V)\, \abs{v+\grad V}^2
 + \chi'_R(V)\sR[v] - \chi'_R(V)\beta_* \;=\; 0\,.
\end{equation}

Observe that $\chi_R(V)-R-1$ is compactly supported by construction.
Thus $\int_{\Rd}\Lg_v\chi_R(V(x))\eta_v(\D{x})=0$ for all $R>0$.
Since $\int_{\Rd}\ell(x)\eta_v(\D{x})<\infty$ by \eqref{ELA.4A}, the bound in
\eqref{ELA.1EX} shows that
\begin{equation}\label{ELA.4D}
\int_{\Rd}\abs{\grad V(x)}^2\,\eta_v(\D{x})\;<\;\infty\,.
\end{equation}
Integrating \eqref{ELA.4C} with respect to $\eta_v$, using \eqref{ELA.4D}, and passing to
the limit as $R\to\infty$, we obtain \eqref{ELA.4B}.
We have thus shown that
\begin{equation*}%\label{ELA.4E}
\int_{\Rd} \abs{m(x)+v(x)}^2\,\eta_v(\D{x}) \;<\;\infty\,.
\end{equation*}
By Theorem~1.1 in \cite{Bogachev-96}, this implies that the measure $\eta_v$
has density in $L^{\nicefrac{d}{(d-1)}}(\Rd)$.
This completes the proof.
\end{proof}

%%%%%%%%%%%%%%%%%%%%%%%%%%%%%%%%%%%%%%%%%%%%%%%%%%%%%%%%%%%%%%%%%%%%%%%%%%%%%%%%
\begin{proof}[Proof of Theorem~\ref{T1.4}]
Without loss of generality we assume $\varepsilon=1$, and we suppress the
explicit dependence on $\varepsilon$ in the notation used in the theorem.
The statement concerning existence of solutions and
the behavior above and below a critical value for
$\beta$ follows by the results in \cite{Ichihara-11}.
For this, we need to first verify
a Foster--Lyapunov type hypothesis, which is part of the assumptions.
Note that the operator $F$ in \cite{Ichihara-11}
has a negative sign in the Laplacian
so that $\cA[\varphi]=-F[\varphi]$, where $\cA$ is the
operator defined in \eqref{ELA.3A}.
So, given that $\ell$ is inf-compact, $\varphi_0=0$ is
an obvious choice to satisfy (A4) in \cite{Ichihara-11}.
Then of course $-\cA[\varphi_0]\to-\infty$ as $\abs{x}\to\infty$.
Note that Theorem~2.2 in \cite{Ichihara-11} then
asserts that $V$ is bounded below in $\Rd$.

Next, consider $\varphi_1 = - a_1 \sqrt\Lyap$ 
with $a_1 \df \inf_{\cK^c}\, \frac{\abs{\langle m,\grad\Lyap\rangle}\sqrt\Lyap}
{\abs{\grad\Lyap}^2}$, where $\cK$ is as in Hypothesis~\ref{H1.1}\,(3),
and $\Lyap$ is as in
Lemma~\ref{L2.3}.
Since $\Lyap$ agrees with $\Bar\Lyap$ outside some compact set
by Lemma~\ref{L2.3}, it follows by Hypothesis~\ref{H1.1}\,(3)
that $a_1>0$.
Then we obtain
\begin{align*}
\tfrac{1}{2}\Delta\varphi_1+ \langle m,\grad\varphi_1\rangle
- \tfrac{1}{2}\abs{\grad\varphi_1}^2
&\;=\; \tfrac{a_1}{4\sqrt\Lyap}\,\Delta\Lyap - \tfrac{a_1}{2\sqrt\Lyap}
\Bigl(\langle m,\grad\Lyap\rangle
+ \tfrac{a_1\sqrt\Lyap-1}{4\Lyap}\,\abs{\grad\Lyap}^2\Bigr)
\\[5pt]
&\;\ge\; \tfrac{a_1}{4\sqrt\Lyap}\,
\bigl(\Delta\Lyap-\langle m,\grad\Lyap\rangle\bigr)
\qquad\text{on\ \ }\cK^c\,.
\end{align*}
Thus, since $\Delta\Bar\Lyap$ is bounded by Hypothesis~\ref{H1.1}\,(3b),
we obtain $-\cA[\varphi_1]\to-\infty$ as $\abs{x}\to\infty$.
It is also clear that
$\phi_0(x)-\phi_1(x)\to \infty$ as $\abs{x}\to\infty$.
Thus, Hypothesis (A.4)$^\prime$ in \cite{Ichihara-11} is also
satisfied.
Therefore, as shown in \cite[Theorem~2.1]{Ichihara-11}, there exists some
critical value $\lambda_*$ such  that
\eqref{HJBerg} has no solution for $\beta>\lambda^*$.
Also by Theorem~2.2 and Corollary~2.3 in \cite{Ichihara-11},
if $V$ is a solution
for $\beta<\lambda^*$, then under the control $v=-\grad V$, the diffusion
is transient.
For $\beta=\lambda^*$ there exists a unique solution $V=V_*$ (up to an additive)
constant, and under the control $v_*=-\grad V_*$ the diffusion
\begin{equation*}
X_{t} \;=\; X_{0} + \int_{0}^{t}\bigl(m(X_{s})+v(X_s)\bigr)\,\D{s}
+ W_{t}\,, \quad t \;\ge\; 0\,,
\end{equation*}
is positive recurrent.
It is clear then that Lemma~\ref{LA.3} implies that $\lambda^*=\beta_*$.

We next turn to the proof of items (a)--(e).
Part~(a) follows directly by \cite[Lemma~5.1]{Met},
and a sharper estimate
was established in the proof of Lemma~\ref{LA.3}
when $\beta=\beta_*$. 
The uniqueness of the solution for $\beta=\beta_*$
follows by the results in \cite{Ichihara-13} discussed above,
while the rest of the assertions in part~(b) follow
by Lemma~\ref{LA.3}.
Part (c) follows by Lemma~\ref{LA.4}.

We now turn to part~(d).
It is enough to show
that for any sequence $\{U^{n}\}\subset\widehat\Uadm$ and
a sequence of times $\{t_n\}$ diverging to $\infty$ then
\begin{equation}\label{PT1.4d1}
\liminf_{n\to\infty}\;\frac{1}{t_n}\,
\Exp_x\biggl[\int_{0}^{t_n} \sR(X^{n}_{s},U^{n}_{s})\,\D{s}\biggr]
\;\ge\;\beta^\varepsilon_*\,,
\end{equation}
where $X^{n}$ denotes the process controlled by $U^n$.
All the terms in this displayed equation are finite,
since
$\int_0^T\Exp_x^U[\sR(X_s,U_s]\,\D{s}<\infty$ for any $U\in\widehat\Uadm$.
This clearly follows by \eqref{EL1.3-bound}.
We include the dependence on the initial condition $X^n_0=x$ in the
notation and denote the corresponding sequence of
mean empirical measures defined in \eqref{E-emp}  by $\Phi^{U^n}_{x,t_n}$.
Extract a subsequence of $\{t_n\}$ over which the terms
on the left hand side of \eqref{PT1.4d1} converge to the `$\liminf$'
and suppose without loss of generality that this limit is finite.
Then the corresponding subsequence of mean empirical measures is tight.
Let $\uppi\in\fP$ be any limit point of this subsequence.
It follows that the left hand side of \eqref{PT1.4d1} is lower bounded
by $\uppi(\sR)$. However, $\uppi(\sR)\ge\beta_*$ by \eqref{EA.19}
and Lemma~\ref{LA.4}. This completes the proof of part (d).

It remains to prove part (e).
Let $\uppi=\eta_v\circledast v\in\fP_\circ$ be any optimal ergodic occupation
measure, and
$\uppi_*\df \eta_*\circledast v_*$, with $v_*=-\grad V$.
By Lemma~\ref{LA.4}, $\eta_v$ has density, which we denote by $\rho_v$.
Let $\xi_v\df\frac{\rho_v}{\rho_v+\rho_*}$ and $\xi_*\df\frac{\rho_*}{\rho_v+\rho_*}$,
and also define $\Bar{v} \df \xi_v v + \xi_* v_*$ and
$\Bar\eta\df\frac{1}{2}(\eta_v+\eta_*)$.
Using the property that the drift of \eqref{E-sde} is an affine function of
the control, it is straightforward to verify that
$\Bar\eta\circledast\Bar{v}\in\fP_\circ$.

By optimality, we have
\begin{align}\label{PT1.4e1}
0 &\;\le\; 2\int_{\Rd} \sR[\Bar{v}]\,\D\Bar\eta
-\int_{\Rd} \sR[v]\,\D\eta_v -\int_{\Rd} \sR[v_*]\,\D\eta_*\\[5pt]
&\;=\; \int_\Rd \abs{\xi_v\, v + \xi_*\, v_*}^2\,\D\Bar\eta
- \tfrac{1}{2}\,\int_\Rd \abs{v}^2\, \D\eta_v
- \tfrac{1}{2}\,\int_\Rd \abs{v_*}^2\, \D\eta_* \nonumber\\[5pt]
&\;=\; \int_{\Rd}\Bigl(\abs{\xi_v\, v + \xi_*\, v_*}^2
-\xi_v\, \abs{v}^2 -\xi_*\, \abs{v_*}^2\Bigr)\, \D\Bar\eta\nonumber\\[5pt]
&\;=\;-\frac{1}{2}\int_{\Rd}\frac{\rho_v(x)\,\rho_*(x)}{\rho_v(x)+\rho_*(x)}\,
\abs{v(x)-v_*(x)}^2\, \D{x}\,.\nonumber
\end{align}
Since $\rho_*$ is strictly positive, \eqref{PT1.4e1} implies
that $\rho_v\,\abs{v-v_*}=0$ a.e.\ in $\Rd$, and thus
$v=v_*$ on the support of $\eta_v$.
It is clear that if $v$ is modified outside the support of $\eta_v$ then
the modified $\eta_v\circledast v$ is also an infinitesimal ergodic occupation
measure.
Therefore $\eta_v\circledast v_*\in\fP_\circ$.
The uniqueness of the invariant measure of the diffusion with
generator $\Lg_{v_*}$ then implies that $\eta_v=\eta_*$,
which in turn implies (since $v=v_*$ on the support of $\eta_v$)
that $v=-\grad V$ a.e.\ in $\Rd$.
This completes the proof of part (e), and also of the theorem.
\end{proof}

%%%%%%%%%%%%%%%%%%%%%%%%%%%%%%%%%%%%%%%%%%%%%%%%%%%%%%%%%%%%%%%%%%%%%%%%%%%%%%%%
\section{Proofs of the results in Section~\ref{S1.5}}\label{AppB}

We start with the proof of Lemma~\ref{L1.16}.

\begin{proof}[Proof of Lemma~\ref{L1.16}]
Suppose that $M$ has a number $q$ of eigenvalues on the open right half complex plane.
Using a similarity transformation we can transform $M$ to a matrix of
the form $\diag(M_1, -M_2)$ where $M_1\in\RR^{(d-q)\times(d-q)}$
and $M_2\in\RR^{q\times q}$ are Hurwitz matrices.
So without loss of generality, we assume $M$ has this form.
Let $S_1$ and $S_2$ be the unique symmetric positive definite matrices
solving the Lyapunov equations $S_1 M_1 + M_1\transp S_1 = -I$ and
$S_2 M_2 + M_2\transp S_2 = -I$, respectively.
Extend these to symmetric matrices in $\RR^{d\times d}$
by defining $\Tilde{S}_1 = \diag(S_1,0)$ and $\Tilde{S}_2 = \diag(0,S_2)$,
and also define, for $\alpha>0$,
\begin{equation*}
\varphi_1(x)\;\df\; \E^{-\alpha \langle x,\Tilde{S}_1 x\rangle}\,,\quad
\varphi_2(x)\;\df\; \E^{-\alpha \langle x,\Tilde{S}_2 x\rangle}\,,\quad
\text{and\ \ } \varphi\;\df\; 1 + \varphi_1-\varphi_2\,.
\end{equation*}
Let $T_1=\diag(I_{(d-q)\times(d-q)}, 0_{q\times q})$,
and $T_2=\diag(0_{(d-q)\times(d-q)}, I_{q\times q})$.
Then, with
$\Lg_v f(x) \df \frac{1}{2}\Delta f(x) +
\bigl\langle Mx+v(x),\grad f(x)\bigr\rangle$,
we obtain
\begin{align}\label{PL1.16A}
\Lg_v \bigl(1-\varphi_2(x)\bigr) &\;=\; \alpha
\varphi_2(x)\Bigl(\trace(\Tilde{S}_2) 
- 2\alpha \bigl\langle x, \Tilde{S}_2^2 x\bigr\rangle + \abs{T_2 x}^2
 + 2\langle v(x), \Tilde{S}_2 x\rangle\Bigr)\\[5pt]
&\;\ge\;
\alpha\varphi_2(x) \Bigl(\trace(\Tilde{S}_2)
+ \abs{T_2 x}^2 - \tfrac{1}{2} \abs{T_2 x}^2
- 2\alpha\norm{\Tilde{S}_2}^2 \abs{T_2 x}^2
- 2\norm{\Tilde{S}_2}^2 \abs{v(x)}^2\Bigr)\nonumber\\[5pt]
&\;=\;
\alpha\varphi_2(x) \Bigl(\trace(S_2)
+ \bigl(\tfrac{1}{2}-2\alpha\norm{\Tilde{S}_2}^2\bigr)\abs{T_2 x}^2
- 2\norm{\Tilde{S}_2}^2 \abs{v(x)}^2\Bigr)\,.\nonumber
\end{align}
For the inequality in \eqref{PL1.16A} we use
\begin{align*}
2 \bigl\langle v(x), \Tilde{S}_2 x\bigr\rangle \;=\;
2 \bigl\langle \Tilde{S}_2v(x), T_2 x\bigr\rangle &\;\ge\;
- \babs{\tfrac{T_2 x}{\sqrt{2}}}^2 - \abs{\sqrt2 \Tilde{S}_2 v(x)}^2\\[5pt]
&\;\ge\; - \tfrac{1}{2}\abs{T_2 x}^2 - 2 \norm{\Tilde{S}_2}^2 \abs{v(x)}^2\,.
\end{align*}
Using the analogous inequality for $\Lg_v \varphi_1(x)$ and combining
the equations we obtain
\begin{align}\label{PL1.16B}
\Lg_v \varphi(x) 
&\;\ge\;
\alpha\E^{-\alpha \langle x, \Tilde{S}_1 x\rangle} \Bigl(-\trace(S_1)
+ \bigl(\tfrac{1}{2}+2\alpha\norm{\Tilde{S}_1}^2\bigr)\abs{T_1 x}^2
- 2\norm{\Tilde{S}_1}^2 \abs{T_1 v(x)}^2\Bigr)\\[5pt]
&\qquad\qquad\;+ \alpha\E^{-\alpha \langle x, \Tilde{S}_2 x\rangle} \Bigl(\trace(S_2)
+ \bigl(\tfrac{1}{2}-2\alpha\norm{\Tilde{S}_2}^2\bigr)\abs{T_2 x}^2
- 2\norm{\Tilde{S}_2}^2 \abs{T_2 v(x)}^2\Bigr)\nonumber\\[5pt]
&\;\ge\; \alpha\Bigl(-\trace(S_1)+\E^{-\alpha\langle x, S x\rangle}
\bigl(\tfrac{1}{2}-2\alpha\norm{S}^2\bigr)\abs{x}^2
- 2\norm{S}^2 \abs{v(x)}^2\Bigr)\,,\nonumber
\end{align}
with $S\df\diag(S_1,S_2)$.

Using It\^o's formula on \eqref{PL1.16B}, dividing by $\alpha$,
and also using the fact that
$\varphi\ge0$ and $\norm{\varphi}_{\infty}=2$,
we obtain
\begin{equation*}
\Exp_{x} \biggl[\int_0^T\Bigl(-\trace(S_1) 
+ \E^{-\alpha\langle X_t, S X_t\rangle}
\bigl(\tfrac{1}{2}-2\alpha\norm{S}^2\bigr) \abs{X_t}^2
- 2\norm{S}^2 \abs{v(X_t)}^2\Bigr)\,\D{t}\biggr]\;\le\;
\frac{2}{\alpha}\,.
\end{equation*}
Dividing by $T$, letting $T\nearrow \infty$ and rearranging terms, we conclude that
$\E^{-\alpha\langle x, S x\rangle}\abs{x}^2$ is integrable with
respect to invariant probability measure
$\mu_v$ under the control $v$
for any $\alpha< \frac{1}{4 \norm{S}^2}$, and the following bound holds
\begin{equation*}
\int_{\Rd} \E^{-\alpha \langle x, S x\rangle}\abs{x}^2 \,\mu_v(\D{x})
\;\le\; \frac{\trace(S_1)}{\tfrac{1}{2}-2\alpha\norm{S}^2}
+ \frac{2\norm{S}^2}{\tfrac{1}{2}-2\alpha\norm{S}^2}
\int_{\Rd} \abs{v(x)}^2 \,\mu_v(\D{x})\,.
\end{equation*}
Taking limits as $\alpha\searrow0$, using monotone convergence, we obtain
\begin{equation*}
\int_{\Rd} \abs{x}^2 \,\mu_v(\D{x})
\;\le\; 2\trace(S_1) + 4\norm{S}^2\, \int_{\Rd} \abs{v(x)}^2 \,\mu_v(\D{x})\,.
\end{equation*}
The proof is complete.
\end{proof}

\begin{proof}[Proof Theorem~\ref{T1.18}]
It is well known that there exists at most one symmetric matrix
$Q$ satisfying \eqref{ET1.18A}-\eqref{ET1.18B} \cite[Theorem~3, p.~150]{Brockett}.
For $\kappa>0$, consider the ergodic control problem of minimizing
\begin{equation}\label{PT1.18A}
J_{\kappa}(v)\;\df\;\limsup_{T\to\infty}\;\frac{1}{T}\,\Exp\biggl[\int_{0}^{T}
\Bigl(\kappa\,\abs{X_{s}}^{2} +
\tfrac{1}{2}\,\abs{v(X_{s})}^{2}\Bigr)\,\D{s}\biggr],
\end{equation}
over $v\in\bUssm$,
subject to the linear controlled diffusion
\begin{equation}\label{PT1.18B}
X_{t} \;=\; X_{0} + \int_{0}^{t}\bigl(M X_{s}+v(X_s)\bigr)\,\D{s}
+ W_{t}\,, \quad t \;\ge\; 0\,. 
\end{equation}
As is also well known, an optimal stationary
Markov control for this problem takes
the form $v(x)= -Q_{\kappa}x$, where $Q_{\kappa}$ is the unique positive
definite symmetric solution to the
matrix Riccati equation
\begin{equation}\label{PT1.18C}
Q_{\kappa}^{2} - M\transp Q_{\kappa} - Q_{\kappa} M \;=\; 
2\,\kappa I\,.
\end{equation}
Moreover, $Q_{\kappa}$ has the following property.
Consider a deterministic linear control system $\Dot{x}(t) = M x(t) + u(t)$,
with $x,u\in\Rd$, and initial condition $x(0)=x_0$.
Let $\cU$ denote the space of controls $u$ satisfying
$\int_{0}^T \abs{u(t)}^2\,\D{t}<\infty$ for all $T>0$, and
$\phi^u_t(x_0)$ denote the solution
of the differential equation under a control $u\in\cU$.
Then 
\begin{equation}\label{PT1.18D}
\bigl\langle x_0, Q_{\kappa} x_0\bigr\rangle
\;=\;\min_{u\in\cU}\;\int_0^\infty \Bigl(\abs{u(t)}^2
+ 2\kappa \abs{\phi^u_t(x_0)}^2\Bigr)\,\D{t}\,.
\end{equation}
For these assertions, see \cite[Theorem~1, p.~147]{Brockett}.

On the other hand,
$\Psi_{\kappa}(x) = \tfrac{1}{2}\bigl\langle x, Q_{\kappa} x\bigr\rangle$ is
a solution of the associated HJB equation
\begin{equation}\label{PT1.18E}
\tfrac{1}{2}\Delta \Psi_{\kappa}(x)
+ \min_{u\in\Rd}\;\Bigl[\bigl\langle Mx+u,
\grad \Psi_{\kappa}(x)\bigr\rangle +\tfrac{1}{2}\abs{u}^{2}\Bigr]
+ \kappa\,\abs{x}^{2}
\;=\; \tfrac{1}{2}\trace(Q_{\kappa})\,.
\end{equation}
The HJB equation \eqref{PT1.18E} characterizes the optimal cost, i.e.,
\begin{equation*}
\inf_{v\in\bUssm}\;J_{\kappa}(v)\;=\;\tfrac{1}{2}\trace(Q_{\kappa})\,.
\end{equation*}
Recall Definition~\ref{D1.17}.
Since the stationary probability distribution of \eqref{PT1.18B}
under the control $v(x)= -Q_{\kappa}x$ is Gaussian,
it follows by \eqref{PT1.18A} that $G=Q_\kappa$ minimizes
\begin{equation*}
\Tilde\cJ_{G;\kappa}(M)
\;\df\;\kappa\,\trace \bigl(\Sigma_{G}\bigr)+
\tfrac{1}{2}\trace\bigl( G\,\Sigma_{G}\, G\transp\bigr)
\end{equation*}
over all matrices $G\in\cG(M)$, where $\Sigma_{G}$ is as in \eqref{ED1.17A}
(note that $\Tilde\cJ_{G;0}(M)=\cJ_{G}(M)$
which is the right hand side of \eqref{ED1.17B}).
Combining this with \eqref{PT1.18E} we have
\begin{equation}\label{PT1.18F}
\inf_{G\in\cG(M)}\;\Tilde\cJ_{G;\kappa}(M)\;=\;
\Tilde\cJ_{Q_\kappa;\kappa}(M)\;=\;
\tfrac{1}{2}\trace(Q_{\kappa})\,.
\end{equation}
By Lemma~\ref{L1.16} we have
\begin{align}\label{PT1.18G}
\trace(\Sigma_{Q_\kappa})&\;\le\;
\widetilde{C}_0\bigl(1 +\cJ_{Q_\kappa;\kappa}(M)\bigr)\\[5pt]
&\;=\; \widetilde{C}_0\bigl(1 + \tfrac{1}{2}\,\trace(Q_{\kappa})\bigr)\,.\nonumber
\end{align}
It also follows by \eqref{PT1.18D} that
$Q_{\kappa'}-Q_{\kappa}$ is nonnegative definite
if $\kappa'\ge\kappa$.
Therefore $Q_{\kappa}$ has a unique limit $Q$
as $\kappa\searrow0$.
It is evident that $Q$ is nonnegative semidefinite, and
\eqref{PT1.18C} shows that it satisfies \eqref{ET1.18A}.
Since $\trace(\Sigma_{Q_\kappa})$ is bounded by \eqref{PT1.18G},
it follows that $\Sigma_{Q_\kappa}$ converges along
some subsequence $\kappa_n\searrow0$ to a symmetric positive semidefinite
matrix $\Sigma$.
Thus \eqref{ET1.18B} holds.
However, \eqref{ET1.18B} implies that $\Sigma$ is invertible, and therefore,
it is positive definite.  In turn, \eqref{ET1.18B} implies that $M-Q$ is Hurwitz.

Since $v_G(x)= -Gx$, $G\in\cG(M)$, is in general suboptimal
for the criterion $J_\kappa(v)$, applying Lemma~\ref{L1.16} once more,
we obtain
\begin{equation*}
\cJ_{Q_\kappa}(M)\;\le\;
\Tilde\cJ_{Q_{\kappa};\kappa}(M)\;\le\; \kappa\,\widetilde{C}_0\bigl(1 + \cJ_G(M)\bigr)
+\cJ_G(M)
\qquad \forall\,G\in\cG(M)\,.
\end{equation*}
Therefore, we have
\begin{equation*}
\cJ_*(M)\;\le\;
\Tilde\cJ_{Q_{\kappa};\kappa}(M)\;\le\; \kappa\,\widetilde{C}_0\bigl(1 + \cJ_*(M)\bigr)
+\cJ_*(M)\,,
\end{equation*}
and taking limits as $\kappa\searrow0$, this implies by \eqref{PT1.18F}
that $\cJ_*(M) \;=\; \tfrac{1}{2}\,\trace(Q)$.

It remains to show that
$\varLambda^+(M) \;=\; \tfrac{1}{2}\,\trace(Q)$.
Let $T$ be a unitary matrix such that $\Tilde{Q}\df T Q T\transp$
takes the form $\Tilde{Q} = \diag(0,\Tilde{Q}_2)$, with
$\Tilde{Q}\in\RR^{q\times q}$ a positive definite matrix.
Write the corresponding block structure of $T M T\transp$ as
\begin{equation*}
\Tilde{M} \;\df\; T M T\transp \;=\;
\begin{pmatrix}\Tilde{M}_{11}& \Tilde{M}_{12}\\[3pt]
\Tilde{M}_{21}&\Tilde{M}_{22}
\end{pmatrix}\,,
\end{equation*}
with $\Tilde{M}_{22}\in\RR^{q\times q}$.
Since $M\transp Q + Q M = Q^2$, we obtain
$\Tilde{M}\transp \Tilde{Q} + \Tilde{Q} \Tilde{M} = \Tilde{Q}^2$,
and block multiplication shows that $\Tilde{Q}_2\Tilde{M}_{21}=0$,
which implies that $\Tilde{M}_{21}=0$.
Since $M-Q$ is similar to $\Tilde{M}-\Tilde{Q}$ the latter must be Hurwitz,
which implies that $\Tilde{M}_{11}$ is Hurwitz.
By block multiplication we have
\begin{equation}\label{PT1.18H}
\Tilde{M}_{22}\transp \Tilde{Q}_2 + \Tilde{Q}_2 \Tilde{M}_{22} = \Tilde{Q}_2^2\,.
\end{equation}
Since  $\Tilde{Q}_2$ is positive definite, the matrix $-\Tilde{M}_{22}$ is Hurwitz
by the Lyapunov theorem.
Thus $\varLambda^+(M) = \trace(\Tilde{M}_{22})$.
Therefore, since $\Tilde{Q}_{2}$ is
invertible, and $\trace(Q)=\trace(\Tilde{Q})$, we obtain by \eqref{PT1.18H} that
\begin{align*}
\trace(Q) \;=\; \trace(\Tilde{Q}_{2})&\;=\;
\trace (\Tilde{M}_{22}\transp+\Tilde{M}_{22})
\nonumber\\
&\;=\;2\trace (\Tilde{M}_{22})\;=\;2\varLambda^+(M)\,.
\end{align*}
This proves part~(a).

Now let $\Hat{v}\in\bUssm$ be any control.
Let $\Bar{V}(x) = \tfrac{1}{2}\langle x,Qx\rangle$.
Then $\Bar{V}$ satisfies \eqref{PT1.18E} with $\kappa=0$.
Since
\begin{align*}
\min_{u\in\Rd}\;\Bigl[\bigl\langle Mx+u,
\grad \Bar{V}(x)\bigr\rangle +\tfrac{1}{2}\abs{u}^{2}\Bigr]
&\;=\; \bigl\langle Mx,\grad \Bar{V}(x)\bigr\rangle
-\tfrac{1}{2}\babs{Qx}^2\\[5pt]
&\;=\;
 \bigl\langle Mx+\Hat{v}(x),\grad \Bar{V}(x)\bigr\rangle
+\tfrac{1}{2}\abs{\Hat{v}(x)}^{2}
-\tfrac{1}{2}\babs{Qx +\Hat{v}(x)}^{2}\,,
\end{align*}
we obtain
\begin{equation}\label{PT1.18I}
\tfrac{1}{2}\Delta \Bar{V}(x)
+ \bigl\langle Mx+\Hat{v}(x),\grad \Bar{V}(x)\bigr\rangle
+\tfrac{1}{2}\abs{\Hat{v}(x)}^{2}
\;=\; \tfrac{1}{2}\trace(Q)+\tfrac{1}{2}\babs{Qx +\Hat{v}(x)}^{2}\,.
\end{equation}
Applying It\^o's formula to \eqref{PT1.18I}, and using the
fact that $\mu_{\Hat{v}}$ has finite second moments as shown in Lemma~\ref{L1.16},
and $\Bar{V}$ is quadratic, a standard argument gives
\begin{equation}\label{PT1.18J}
\int_{\Rd}  \Bigl( \tfrac{1}{2}\abs{\Hat{v}(x)}^{2}
-\tfrac{1}{2}\abs{Qx +\Hat{v}(x)}^{2}\Bigr)\,\mu_{\Hat{v}}(\D{x})\;=\;
\tfrac{1}{2}\trace(Q)\,.
\end{equation}
Thus
$\int_{\Rd}\tfrac{1}{2}\abs{\Hat{v}(x)}^{2}\,\mu_{\Hat{v}}(\D{x})\ge
\tfrac{1}{2}\trace(Q)=\cJ_*(M)$.
Hence \eqref{ET1.18C} holds.

Suppose $\Hat{v}$ is optimal, i.e., attains the infimum in \eqref{ET1.18C}.
By \eqref{PT1.18J}, we
obtain
\begin{equation*}
\lim_{\kappa\searrow0}\;
\int_{\Rd}\abs{Qx +\Hat{v}(x)}^{2}\,\mu_{\Hat{v}}(\D{x})\;=\;0\,.
\end{equation*}
Therefore, since $\mu_{\Hat{v}}$ has a positive density,
it holds that
$\Hat{v}(x)=-Qx$ a.e.\ in $\Rd$.
This completes the proof of part~(b).

We have shown that $\Bar{V}(x) = \tfrac{1}{2}\langle x,Qx\rangle$ satisfies
\eqref{ET1.18D} with $\Bar{\beta}_*=\varLambda^+(M)$
and the associated process is positive recurrent.
Therefore, as in the  proof of Theorem~\ref{T1.4} for a bounded $m$,
part~(c) follows by Theorems~2.1--2.2 and Corollary~2.3 in \cite{Ichihara-11}.
Note that Hypothesis (A4) in \cite{Ichihara-11} is
easily satisfied for the linear problem.
Since $M$ is exponentially dichotomous, then
as seen in the proof of Theorem~\ref{T2.2},
there exists symmetric matrices $S$ and $\Hat{S}$, with $\Hat{S}$ positive
definite such that
$M\transp S + SM = \Hat{S}$.
Consider the  function $\varphi_0(x) \df a\,\langle x,Sx\rangle$,
with $a\df \frac{1}{4}\bigl(\norm{\Hat{S}^{-1}}\norm{S}^{2}\bigr)^{-1}$.
Since
\begin{equation*}
\norm{\Hat{S}^{-1}}\,\langle x, \Hat{S}x\rangle \;\ge\;
\abs{x}^2\;\ge\; \norm{S}^{-2}\,\abs{Sx}^2\,,
\end{equation*}
we obtain
\begin{align*}
\Bar\cA[\varphi_0](x)&\;\df\; \tfrac{1}{2}\,\Delta\varphi_0(x)
+ \langle Mx,\nabla \varphi_0(x)\rangle - \tfrac{1}{2}\,
\abs{\nabla \varphi_0(x)}^{2}\\[5pt]
&\;=\; a\,\trace{S} + a\,\langle x, \Hat{S}x\rangle
- 2 a^2\,\abs{Sx}^2 \\[5pt]
&\;>\;\tfrac{a}{2}\,\bigl(2\trace{S} - \langle x, \Hat{S}x\rangle\bigr)\,.
\end{align*}
Thus  $\Bar\cA[\varphi_a](x)\to\infty$ as $\abs{x}\to\infty$.
This completes the proof.
\end{proof}

%%%%%%%%%%%%%%%%%%%%%%%%%%%%%%%%%%%%%%%%%%%%%%%%%%%%%%%%%%%%%%%%%%%%%%%%%%%%%%%%
\section*{Acknowledgments}
The authors are indebted to the anonymous referees for their 
constructive comments and suggestions.
This work was initiated during Vivek Borkar's visit to the Department of
Electrical Engineering, Technion, supported by Technion. Thanks are due to 
Prof.\ Rami Atar for suggesting the problem as well as for valuable discussions.

The work of Ari Arapostathis was supported in part by the Office of Naval
Research through grants N00014-14-1-0196 and N00014-16-1-2956, and in part by
the Army Research Office through grant W911NF-17-1-001.

The work of Anup Biswas was supported in part by an award
from the Simons Foundation (\# 197982) to The University of Texas at Austin,
in part by the Office of Naval Research grant N00014-14-1-0196, and in part by an 
INSPIRE faculty fellowship.

The work of Vivek Borkar was supported in part by a J.\ C.\ Bose Fellowship
from the Department of Science and Technology, Government of India.

%%%%%%%%%%%%%%%%%%%%%%%%%%%%%%%%%%%%%%%%%%%%%%%%%%%%%%%%%%%%%%%%%%%%%%%%%%%%%%%%
\def\cprime{$'$}

\end{document}